\documentclass[11pt,reqno]{paper}
\usepackage[pdftex]{graphicx}
\usepackage{fancyhdr}
\usepackage[a4paper, hmargin=2.5cm, vmargin={3.2cm, 3.2cm}]{geometry}
\usepackage[utf8]{inputenc}

\usepackage{titlesec}
\titleformat{\subsection}[runin] 
{\normalfont\bfseries}{\thesubsection}{1em}{}
\titleformat{\subsubsection}[runin] 
{\normalfont\bfseries}{\thesubsubsection}{1em}{}
\titleformat{\paragraph}[runin] 
{\normalfont\bfseries}{\theparagraph}{1em}{}
\makeatletter
\def\@seccntformat#1{\@ifundefined{#1@cntformat}%
	{\csname the#1\endcsname\quad}  
	{\csname #1@cntformat\endcsname}
}
\usepackage[titletoc,title]{appendix}

\usepackage{enumerate}
\usepackage{amsmath}%
\usepackage{amsfonts}%
\usepackage{amsthm}
\usepackage{amssymb}%
\usepackage{graphicx}
\usepackage{hyperref}
\usepackage{theoremref}
\usepackage{bbm}
\usepackage{amssymb}

\DeclareMathOperator{\re}{Re}
\DeclareMathOperator{\im}{Im}

\theoremstyle{plain}

\numberwithin{equation}{section}

\newtheorem{thm}{Theorem}

\newtheorem*{thmRW}{Theorem 7$^\prime$}

\newtheorem{lemma}{Lemma}[section]

\newtheorem{ThmKE}{Theorem}

\newtheorem{inThm}{Theorem}

\newtheorem*{theorem*}{Theorem}
\theoremstyle{definition}
\newtheorem{definition}{Definition}
\theoremstyle{plain}

\newtheorem{superlog example}[sublog]{Theorem}
\newtheorem{Real analytic}[sublog]{Theorem}
\newtheorem{corollary}{Corollary}
\newtheorem*{corollary'}{Corollary 1'}

\newtheorem{remark}{Remark}
\theoremstyle{definition}

\newtheorem{DefL*}[Def]{Definition}

\title{On Taylor coefficients of smooth functions}
\author{Avner Kiro
	\thanks{The author is partially supported by ERC Advanced Grant 692616,  ISF Grant~382/15 and BSF Grant~2012037.}
	\\
	School of Mathematical Sciences, Tel Aviv University\\
	Tel Aviv 69978, Israel
	\\
	E-mail address: {\tt avnerkiro@gmail.com}}
\begin{document}
	\maketitle
	\begin{abstract}
	We study the Borel map, which maps infinitely differentiable functions on an interval  to the jets of their Taylor coefficients at a given point in the interval. Our main results include a complete description of the image of the Borel map for Beurling classes of smooth functions and a moment-type summation method which allows one to recover a function from its Taylor jet. A surprising feature of this description is an unexpected threshold at the logarithmic class.  Another interesting finding is a ``duality'' between non-quasianalytic and quasianalytic classes, which reduces the description of the image of the Borel map for non-quasianalytic classes to the one for the corresponding quasianalytic classes, and complements  classical results of Carleson and Ehrenpreis.
	\end{abstract}
	\section{Introduction}\label{sec:Intro}
	\subsection{}\label{subsec:Intro.1}  We study the Borel map $\mathcal B: C^\infty (I)\to \mathbb{C}^{\mathbb{Z}_+}$, which maps infinitely differentiable functions on the interval $I\subseteq \mathbb{R}$ to the jets of their Taylor coefficients at a given point $x\in I$. From here on, we assume that $I$ contains the origin, and let $x=0$. Then $ \mathcal{B} f =\left(\widehat{f}(n)\right)_{n\geq 0}$,  where  $\widehat{f}(n)=\frac{f^{(n)}(0)}{n!}$.
	
	Given a class of smooth functions $A\subset C^\infty (I)$, the three classical problems that arise naturally, and go back to Borel and Hadamard, are as follows:
	\begin{enumerate}
		\item The uniqueness problem-- When is the restriction  $\mathcal{B}\vert_A$  injective?
		\item The punctual image problem-- What is the image $\mathcal{B}A$?
		\item The summation problem-- Given a sequence $a\in\mathcal{B}A$, to recover  a function $f\in A$ with $\widehat{f}=a$.
	\end{enumerate}
	These problems are usually  studied for Carleman and Beurling classes of smooth functions.
	
	\begin{def}\label{def:Beurling class}
		Let $I\subseteq \mathbb{R}$ be an interval and let $L:[0,\infty) \to [1,\infty)$ be a non--decreasing function with $\lim_{\rho\to\infty}{L(\rho)}=\infty$.
		\textit{The Carleman class} $C(L;I)$ consists of all $\displaystyle C^\infty(I)$--functions $f$ such that for every closed subinterval $J\subset I$, there exist constants  $C, C_1> 0 $ such  that 
		\[ \max_{J}|f^{(n)}|\leq C\cdot  \left(C_1 nL(n)\right)^n,\quad n\in\mathbb{Z}_+. \]
		\textit{The Beurling class} $C_0(L;I)$ consists of all $\displaystyle C^\infty(I)$--functions $f$ such that for every closed subinterval  $J\subset I$ and every $\delta>0$, there exists a constant $C>0$
		such that 
		\[ \max_{J}|f^{(n)}|\leq C\cdot  \left(\delta nL(n)\right)^n,\quad n\in\mathbb{Z}_+. \] 
	\end{def}
	Here and elsewhere, we assume that $0^0=1$.
	
	\subsection{}\label{subsec:Intro.2} 
		The classes $C_0(L;I)$ and $C(L;I)$ are called \textit{quasianalytic} if the Borel map is injective in these classes. 
	The solution of the uniqueness problem in these classes was given by Denjoy and Carleman \cite{carleman}: \textit{The classes $C(L;I)$ and $C_0(L;I)$ are quasianalytic if and only if  }
	\[ \int^\infty \frac{du}{uL(u)}=\infty. \]
	Throughout this work we will abuse terminology and refer to the  function $L$  as quasianalytic if the above integral diverges. 
	\subsection{} \label{subsec:Intro.3}
In 1925 \cite[p.162]{Vallee Poussin}   de la Vall\'ee Poussin wrote:
``The question of finding a criterion of consistence for the
initial values of a quasi-analytic function, of a certain class,
and its derivatives (\textit{i.e., a criterion for a given sequence to be the sequence of Taylor coefficients of a function in a given quasianalytic class } -- A.K. ), is thus before us. But it seems exceedingly
difficult to solve. We shall not undertake it here.'' Since that the  punctual image and the summation problems  were studied by many authors who obtained a number of non--trivial results. Here we mention  the  works  of Carleman \cite{carleman}, Bang \cite{Bang}, Carleson \cite{Carleson,Carleson2},  Ehrenpreis \cite{Ehrenpreis9}, Badalyan \cite{Badalyan}, and  {\'E}calle \cite{Ecalle2}.
However,	these works do not provide explicit answers to the punctual image and summation problems 	in  what is likely the  most interesting and delicate case, namely, when $L$ is slowly varying, i.e.,
	\begin{equation}
	\lim_{\rho\to\infty} \frac{L(\lambda \rho)}{L(\rho)}=1,\quad \forall \lambda>1.\label{slowlygrowing}
	\end{equation}
 On the other hand, if the function $L$ grows fast, that is, 
	\begin{equation}
	\liminf_{\rho\to\infty} \frac{L(2 \rho)}{L(\rho)}>1,\label{notslowlygrowing}
	\end{equation}
	a simple description of the punctual image follows from a more general result, proven independently by  Carleson \cite{Carleson}, Ehrenpreis \cite{Ehrenpreis9} and Mityagin \cite{Mitjagin}:
	\subsubsection{}\label{subsubsec:Intro.3.1} \textit{Put}
	\begin{equation}
	\mathcal{F}_0(L):=\big\{(a_n)_{n\geq 0}: |a_n|^{1/n}=o\left(L(n)\right),\: n\to \infty\big\}.\label{Fdef}
	\end{equation}
	\textit{Then}
	\begin{equation}
	\mathcal{B} C_0(L;I)= \mathcal{F}_0(L) \label{trivialeq}
	\end{equation}
	\textit{provided that  condition} \eqref{notslowlygrowing} \textit{is satisfied.} 
	
	Note that the inclusion
	\begin{equation}
	\mathcal{B} C_0(L;I)\subseteq \mathcal{F}_0(L) \label{trivialineq}
	\end{equation}
	follows from the definition of the Beurling classes $C_0(L;I)$, and that a statement similar to \eqref{trivialineq} also holds for the Carleman classes $C(L;I)$ with $L$ satisfying condition \eqref{notslowlygrowing}. 
	\subsubsection{}\label{subsec:Intro.3.2} On the other hand, for unbounded and slowly growing functions $L$ satisfying \eqref{slowlygrowing}, the inclusion in \eqref{trivialineq} is always proper. In the quasianalytic case, this  was shown by T\"{a}cklind \cite{tacklind1936classes} and Bang \cite{Bang2}, while in the non-quasianalytic case,  this follows from the above mentioned works of Carleson, Ehrenpreis and  Mityagin.
	
	\subsection{}\label{subsec:Intro.goal} Our goal is to give sufficiently explicit answers to the punctual image and summation problems, however, under rather restrictive smoothness assumptions imposed on the slowly growing functions $L$. Our inspiration comes from Beurling's work  \cite[pp. 420--429]{Beurling}, in which he gave a concise solution to these two problems in the logarithmic class when $L(\rho)=\log(\rho+e)$. 
	
	\subsection{Notation}\label{subsec:Intro.notation} We shall use the symbol  $C$ to denote   large positive constants which may change their values
	at different occurrences. If a constant $C$ depends on some  parameter $p$, we will write $C_p$ (again the value can change in different   occurrences). Given two functions $f$ and $g$ with the same domain of definition, we write $f\lesssim g$ whenever $f (x)\leq Cg(x)$ for some constant $C$. If the constant $C$ in the last inequality   depends on some parameter $p$, we will write $f\lesssim_p g$. We use the notation $f\asymp g$ ($f\asymp_p g$) if $f\lesssim g$ and $g\lesssim f$ ($f\lesssim_p g$ and $g\lesssim_p f$). If for some set $\Pi$ we have $f\vert_\Pi \lesssim g\vert_\Pi$, we will write $f\lesssim g$ in $\Pi$, and we will do the same for $\asymp$, $\lesssim_p$ and $\asymp_p$. If $\lim_{x\to\infty} \frac{f(x)}{g(x)}=1$, we will write $f\sim g$. For a function $f$ defined on the positive ray, we will say that a property $P$ is satisfied eventually, if there exists $\rho_0$ such that  $P$  is  satisfied in the interval $[\rho_0,\infty)$.

	\section{Basic notions}\label{sec:Basic notions}
	\subsection{Moment (Borel-type) summation of divergent series. }\label{subsec:moment summation}
Here we 	recall the classical moment summation method that goes back to Borel.  We follow   Hardy's treatise \cite{Hardy}. Let $(\gamma_n)_{n\geq 0}$ be a fast growing sequence of positive numbers, $\lim_{n\to\infty}\gamma_n^{1/n}=\infty$, and suppose it is a moment sequence for a function $K$ defined on $\mathbb{R}_+$, that is 
	\begin{equation}
	\int_0^\infty t^n K(t)dt= \gamma_{n+1},\quad \forall n\in\mathbb{Z}_+. \label{Kmoment}
	\end{equation}
	Let $(a_n)_{n\geq 0}$ be a sequence of complex numbers. If the series
	\[ A(t)=\sum_{n\geq 0} \frac{a_n t^n}{\gamma_{n+1}} \]
	is convergent for $0\leq t<t_0$, and has an analytic continuation on the whole positive ray $\mathbb{R}_+$ such that 
	\[\int_0^\infty A(t)K(t)=b,  \] 
	then the series $\sum_{n\geq 0} a_n$ is said to be $(\gamma)$\textit{--summable} to the value $b$. 
	\subsubsection{}\label{subsubsec:moment summation.1}The entire function
	\[  E(z)=\sum_{n\geq 0}\frac{z^n}{\gamma_{n+1}}\]
	plays an important role in the theory of  moment summation. Recall that a summation method is called \textit{regular} if it agrees with actual limits of convergent series, and is called \textit{stable} if the difference of the values it assigns to $\sum_{n\geq 0}a_n$ and $\sum_{n\geq 1}a_n$ is $a_0$. A necessary condition for regularity of the moment method is 
	\[ \int_0^\infty E(xt)K(t)dy=\frac{1}{1-x},\quad 0\leq x<1, \]
	while a necessary condition for its stability  is
	\[ \int_0^\infty E(t)K(t)dy=\infty. \]
	This hints at a strong connection between the growth of $E$ and the decay of $K$ on $\mathbb{R}_+$.
	
	\subsubsection{}\label{subsubsec:moment summation.2} Note that the sequence $\gamma_n=n!$ corresponds to the classical Borel summation. In this case, $K(t)=e^{-t}$ and $E(z)=e^z$, that is $E\cdot K= 1$. The sequences $\gamma_n=\Gamma(\alpha n)$, $\alpha>0$, and $\gamma_n=\log^n(n+e)$ correspond to the Mittag--Leffler and Beurling summations, respectively. In these cases, the asymptotics of the functions $K$ and $E$ are classical. In particular, there is also a very strong connection between the growth of $E$ and the decay of $K$ on $\mathbb{R}_+$ (in both  cases, $E(x)K(x)=O(\log(E(x))$ as $x\to\infty$, see Lindel\"of \cite[pp. 113--114]{Lindelof}).
	\subsection{Beurling's approach.}\label{subsec:Beurling's approach}
	In the preprint \cite{Beurling2} written in 1936 and reproduced in his collected works \cite[pp. 420--429]{Beurling}, Beurling applied the idea of  moment summation to the punctual image and summation problems in the logarithmic class $C_0(L;I)$  with $L(\rho)=\log(\rho+e)$. A somewhat similar approach to the  summation problem was independently developed by Moroz \cite{Moroz,Moroz2} in the context of divergent asymptotic series, which appear in mathematical physics  for functions analytic in cusp domains.
	
	\subsubsection{}\label{subsubsec:Beurling's approach.1} To explain Beurling's approach, we fix an increasing 
	function $L:[0,\infty) \to [1,\infty)$  with $\lim_{\rho\to\infty}{L(\rho)}=\infty$ and put $\gamma(\rho)=L(\rho)^\rho$. Given a function $f\in C_0(L;I)$, consider the Taylor series 
	\begin{equation}
	\sum_{n\geq 0} \frac{\widehat{f}(n)}{\gamma(n+1)}z^n \label{singtransdef}
	\end{equation}
	and note that, by the definition of the Beurling class $C_0(L;I)$, this series has an infinite radius of convergence. Following Beurling, we call this series \textit{the singular transform } of $f$ and denote it by $S_L f$.
	
	\subsubsection{}\label{subsubsec:Beurling's approach.2} The first observation  is that $S_L f$ depends on the sequence $\widehat{f}$ of Taylor coefficients of $f$ at the origin, but does not depend on the values of $f$ at other points of the interval $I$. Therefore, studying the image of $C_0(L;I)$ under the map $S_L$ is equivalent to studying the punctual image of $C_0(L;I)$. 
	
		\subsubsection{}We define the maps $\widehat{S}_L$ and $\widehat{R}_L$ that act on arbitrary sequences $(a_n)\in\mathbb{C}^{\mathbb{Z}_+}$ as  
	\[ \left(\widehat{S}_L a\right)(n)=\frac{a_n}{\gamma(n+1)},\quad n \in\mathbb{Z}_+, \]
	and
	\[ \left(\widehat{R}_L a\right)(n)=a_n\gamma(n+1),\quad n \in\mathbb{Z}_+. \]
	Then, $\widehat{R}_L =\widehat{S}_L^{-1}$ and $\widehat{S}_L \mathcal{B}=\mathcal{B}S_L$. Note that  if $a$ is an  sequence in   $$\mathcal{F}_0(L):=\left\{(a_n)_{n\geq 0}: |a_n|^{1/n}=o\left(L(n)\right),\: n\to \infty\right\},$$  then $\widehat{S}_L a$ are the Taylor coefficients of an arbitrary entire function.  However, as Beurling observed, there are certain restrictions on the growth of entire functions in $S_L C_0(L;I)$ in horizontal strips. In principle, these restrictions can be used to characterize the punctual image of $C_0(L;I)$.
	
	\subsubsection{}\label{subsubsec:Beurling's approach.3} Note that a similar idea can also be  used for the Carleman classes. However, in that case, the Taylor series \eqref{singtransdef} may have a finite radius of convergence. Thus, the description of the class $S_L C(L;I)$ will include the fact that analytic functions from this class must have an analytic continuation to a horizontal strip. To fix ideas, we will discuss only the more transparent case of Beurling classes $C_0(L;I)$.
	
	\subsubsection{}\label{subsubsec:Beurling's approach.4} To approach the summation problem, one needs to define the kernel $K$ that solves the moment problem 
	\[ \int_0^\infty t^n K(t)dt= \gamma(n+1),\quad n\in\mathbb{Z}_+. \]
	If the function $\gamma$ is nice (in particular, is analytic in the right half-plane), we can define $K$ as the inverse Mellin transform 
	\begin{equation}
	K(t)=\frac{1}{2\pi i}\int_{c-i\infty}^{c+i\infty} t^{-s}\gamma(s)ds ,\quad c>0.\label{eq:K def}
	\end{equation}
	Then, one needs to check that for any entire function $F=S_L f$, $f\in C_0(L;I)$,
	\[ \int_0^\infty F(xt)K(t)dt=f(x),\quad x\in I  \]
	(cf. equation \eqref{Kmoment}  in  Section 2.1). Following Beurling, we call the integral 
	\[ \int_0^\infty F(xt)K(t)dt\]
	\textit{the regular transform} of the function $F$, and denote it by $R_L F$. Note that  $\mathcal{B}R_L=\widehat{R}_L \mathcal{B}$.
	
	\subsubsection{}\label{subsubsec:Beurling's approach.5}  To make this approach work, we need estimates on the asymptotic behavior of the functions $K$ and $E$. To do so, we  impose certain regularity conditions on the weight function $L$  (a systematic study of the asymptotic behavior of these functions can be found in \cite{KiroSodin}).
	
	These regularity conditions are quite technical and different in each part of this work. In order to
	overcome  technical issues,    we will first describe our results only for a particular choice of weight functions $L$ (the so called Denjoy weights)
	which will satisfy all the regularity assumptions that  will be imposed below. Then, before the proof of each result, we will restate it under the   more general regularity assumptions.

		\subsubsection{Denjoy weights.}\label{subsubsec:weight functions.Denjoy weights}    A \textit{Denjoy weight} is a function of the  form 
	\begin{equation}
	L_\alpha(s)=e^{\log^{\alpha_0}(s+1)}\prod_{k\geq 1}\log_k^{\alpha_k}(s+\exp_k(1)), \label{Denjoy}
	\end{equation}
	where $\alpha=(\alpha_0,\alpha_1,\alpha_2,\ldots)\in\mathbb{R}_+^{\mathbb{Z}_+} $ is a multi--index with finitely many nonzero components  $\alpha_k$ and with $0\leq \alpha_0<1$. Here 
 $\log_k$ is the $k$-th iterate of the logarithm function $s\mapsto \log s$, and
 $\exp_k$ is the $k$-th iterate of   the exponential  function $x \mapsto e^x$. We consider only $\alpha_0<1$ since, for $\alpha_0\geq 1$, the function $L_\alpha$ is growing fast, that is, relation \eqref{notslowlygrowing} holds. 
	
		\section{Main results for Denjoy weights}
	 \subsection{The class $A(L;I)$ of entire functions.}\label{subsec: the space A} Following Beurling, we introduce a class of entire functions. This class is defined in terms of the function 
	 \[ E(z)=\sum_{n\geq 0} \frac{z^n}{\gamma(n+1)},\quad \text{where} \quad \gamma(n)=L(n)^n. \]
The asymptotics of $E(z) $ as $z\to\infty$  will be used repeatedly in this work. Here we will only mention that if $L$ is a Denjoy weight, then the corresponding function $E$  grows very rapidly on the positive half-line ($\lim_{x\to\infty} x^{-a}\log E(x)=+\infty $ for any $a>0$), but it is bounded  on any infinite sector that does not meet the positive  half-line (see Section 6 for the exact asymptotics of the function $E$).
	 \begin{definition}\label{def: the space  A}
	 	 Let  $L:[0,\infty)\to [1,\infty)$ be a continuous and eventually increasing function such that $\lim_{\rho\to\infty}L(\rho)=\infty$. Put $\gamma(\rho)=L(\rho)^\rho$  and $E(z)=\sum_{n \geq 0 } \frac {z^n}{\gamma(n+1)}$. Let $I$ be an interval  containing $0$. \textit{The class  $A(L;I)$}  consists of all entire functions $F$ with the property that, for every $c_-\in I\cap (-\infty,0),\: c_+\in I\cap (0,\infty)$, and $Y>0$,
	 	\begin{equation}
	 	\max_{|v|\leq Y} |F(u+iv)|\lesssim_{c_\pm,Y}E(u/c_\pm),\quad u \to\pm\infty.  \label{inq1}
	 	\end{equation}
	 \end{definition}

	 Note that if  $L(\rho)=o(L_1(\rho))$ as $\rho\to\infty$, then   $A(L_1;I)\subset A(L,\mathbb{R})$ for any open interval $I$. Moreover, if $L$ is slowly growing (in particular, if $L$ is a Denjoy weight), then the RHS of \eqref{inq1} can be replaced by $\mu(u/c_\pm)$, where $\mu(r)=\max_{n\geq 0} \frac{r^n}{\gamma(n+1)}$, as well as by  $\exp(L^{-1}(u/c_\pm))$,  where  $L^{-1}$ is the inverse function to $L$ (cf.  Lemma \ref{lem: K and E ray lemma} and its proof).
	 
 	 \subsection{Extension of Beurling's theorem.} 
 	 Our first two results are as follows. 
 	 \begin{inThm}\label{thm: Main thm,reg}
 	 	Suppose  $L$ is a Denjoy weight and  $I$  is an interval containing the origin, such that $I\cap(0,\infty)$ and $I\cap(-\infty,0)$ are open. Then, the  regular transform $R_L$ maps $A(L;I)$ into $C_0(L;I)$.	
 	 \end{inThm}
 	 \begin{inThm} \label{thm: Main thm,sing}
 	 	Suppose  $L$ is a Denjoy weight and  $I$  is an open interval containing the origin. Then,
 	 	the singular transform $S_L$ maps $C_0(L;I)$ into $A(L;I)\cup A(\tfrac{1}{\varepsilon};\mathbb{R})$,
 	 	where $\varepsilon(\rho):=\frac{\rho L'(\rho)}{L(\rho)}$.
 	 \end{inThm}
\subsubsection{Functions of at most logarithmic and of super-logarithmic growth.}   
  We say that a  function    $L$ has \textit{at most logarithmic growth} if  $L(\rho)=O(\log \rho)$ as $\rho\to \infty $, and that $L$ has \textit{a super-logarithmic  growth} if $\log \rho= o\left(L(\rho)\right)$ as $\rho\to \infty $. Note that any Denjoy weight  must have either at most logarithmic or super-logarithmic growth (unlike more general  functions $L$). Theorems \ref{thm: Main thm,reg} and \ref{thm: Main thm,sing} exhibit an essential difference between these two cases.
 	 
 	 \begin{corollary}\label{cor: subsuplog}
Suppose that $L$ is a Denjoy weight and that $I$ is an open interval containing the origin.
\begin{enumerate}
	\item If  $L$ has at most  logarithmic growth, then the singular transform $S_L$ maps $C_0(L;I)$ bijectively onto the space $A(L;I)$. Moreover, if $f\in C_0(L;I)$, then $R_L S_L f=f$.
	\item If $L$ has super-logarithmic growth, then the singular transform $S_L$ maps $C_0(L;I)$ into $A(\tfrac{1}{\varepsilon};\mathbb{R})$.
\end{enumerate}
 	 \end{corollary}
 Notice that in the case where $L$ has at most logarithmic growth 
 we always have $L(\rho)\varepsilon(\rho)=O(1)$ as $\rho\to\infty$ and therefore, $A(L;I)\supset A(\tfrac{1}{\varepsilon};\mathbb{R})$. However, in the super-logarithmic case,  we always have $\tfrac1{\varepsilon(\rho)}=o(L(\rho))$ as $\rho\to\infty$, and therefore $A(\tfrac{1}{\varepsilon};\mathbb{R})\supset A(L;I)$.
  
The  first part of the corollary gives a full description of the punctual image and  a summation method for (a divergent) Taylor series of functions in $C_0(L;I)$ when $L$ has at most logarithmic growth. In the case $L(\rho)=\log(\rho+e)$, this is the aforementioned result of Beurling. On the other hand, in the super-logarithmic case,  these results do not give full answers, but still provide a non-trivial information about the punctual image and summation problem.

 \subsubsection{Sharpness of Theorem 2.} Our next  result shows that the second statement of Corollary~\ref{cor: subsuplog} cannot be essentially improved.
 \begin{inThm}\label{thm: example3}
Suppose that $L$ is a Denjoy weight with super--logarithmic growth and  that  $I$ is an open interval containing 0. Then for any function $L_2$ satisfying 	$\tfrac{1}{\varepsilon(\rho)}=o(L_2(\rho))$ as $\rho\to\infty$ and any $\delta>0$,  $S_L C_0(L;(-\delta,\delta))\nsubseteq A(L_2;\mathbb{R})$. In particular  $S_L C_0(L;(-\delta,\delta))\nsubseteq A(L;\mathbb{R})$.
 \end{inThm}
  \subsection{Duality between non-quasianalytic and quasianalytic classes.} Given a non-quasianalytic function $L$, put
 \[ \widetilde{L}(\rho)=L(\rho)\int_{\rho }^{\infty}\frac{du}{uL(u)}, \quad \rho>1. \]
 Note that the function $\widetilde{L}$ is always quasianalytic.
 Indeed, integration by parts yields
 \[ \int_{\rho_0}^{\rho}\frac{du}{u\widetilde{L}(u)}= \log \left(\int_{\rho_0 }^{\infty}\frac{du}{uL(u)}\right)-\log \left(\int_{\rho }^{\infty}\frac{du}{uL(u)}\right) \]
 and therefore
 \[ \int_{\rho_0}^{\infty}\frac{du}{u\widetilde{L}(u)}=\infty. \]
 
  In the next table, we give some examples for non--quasianalytic functions $L$ and their duals $\widetilde{L}$ (given here up to the asymptotic equivalence $\asymp$). 

 \renewcommand{\arraystretch}{1.2} 
 \begin{center}
 	\begin{tabular}{|l|c|c|l|}
 	\hline
 	& $L$ & $\widetilde{L}$ & Parameters\\
 	\hline
 	I & $ \log^{\alpha }(\rho+e)$& $ \log(\rho+e)$&$\alpha>1$ \\ \hline
 	II & $\log(\rho+e)\log^{\beta}\log(\rho+e^e)$& $ \log(\rho+e)\log\log(\rho+e^e)$&$\beta>1$ \\ \hline
 	III & $\exp[\log^\alpha (\rho)]$& $\log^{(1-\alpha)}(\rho+e)$&$0<\alpha<1$\\ \hline
 	IV& $\exp\left[\frac{\log (\rho)}{\log^\beta \log(\rho+e)}\right]$ &$\log^{\beta}\log (\rho+e^e)$&$\beta>0$\\
 	\hline V& $\rho^\alpha$ &$1$&$\alpha>0$\\
 	\hline
 \end{tabular}
 \newline $ $\newline 
 \end{center}
 \noindent
 Our main results regarding non-quasianalytic classes are the following.  
 \begin{inThm}\label{thm:NqaBeu}
 	Suppose  $L$ is a non-quasianalytic Denjoy weight and  $I$ is an open interval containing the origin. Then
 	\[ S_L C_0(L;I)=S_{\widetilde{L}} C_0(\widetilde{L};\mathbb{R}). \]
  \end{inThm}
There is also an analogous result for Carleman classes. We will make an exception and state it here because  this is the only place in this work where the treatment of Beurling and Carleman classes  requires different techniques.

We introduce the notation
\[ C(L;0):=\bigcup_{\delta>0} C(L;(-\delta,\delta)).\]
That is, $C(L;0)$ consists of germs around the origin of the corresponding Carleman class. 
\begin{inThm}\label{thm:NqaCar}
	Suppose  $L$ is a non-quasianalytic Denjoy weight  and  $I$ is an open interval containing the origin.  Then
	\[ S_L C(L;I)=S_{\widetilde{L}} C(\widetilde{L};0). \]
  \end{inThm} 
As we will discuss in Section \ref{sec: Car and Ehr thm}, Theorems \ref{thm:NqaBeu} and \ref{thm:NqaCar} are closely related to classical results by Carleson \cite{Carleson} and Ehrenpreis \cite{Ehrenpreis9} and can be viewed as an improvement of their results but for a restricted class of weights.
  \subsubsection{} Theorems \ref{thm:NqaBeu} and \ref{thm:NqaCar}  allows us to reduce the punctual image problem in the non--quasianalytic case to the same problem in the  quasianalytic case. It particular, together with Corollary \ref{cor: subsuplog}, we obtain the following:
 \begin{corollary}
 Suppose  $L$ is a Denjoy weight  and  $I$ is an open interval containing the origin. If  $\log^a \rho\lesssim L(\rho)$ for some $a>1$, 
 	then the singular transform $S_L$ maps $C_0(L;I)$  onto  $A(\widetilde{L};\mathbb{R})$.
 \end{corollary}
 \subsection{The splitting.}
  The  results  presented so far describe fully  the punctual image of  Beurling classes, only in the cases that the Denjoy weight function $L$ satisfies  $   L \lesssim \log$  or	$\log^a\lesssim L$, for some $a>1$. On the other hand, when the function $L$ is closer to the quasianalyticity threshold, such as when $L(\rho)=\log \rho \log^\beta \log \rho$, with $\beta>0$, our results, so far,  do not give a full description of the punctual image of  $C_0(L;I)$. Note that in the latter case, the inclusion $S_L C_0(L;I)\subset A(\log ;\mathbb{R})$ of Corollary \ref{cor: subsuplog}, part 2, is proper.  So, to treat this case, new ideas are needed. 
 
In order to overcome these difficulties, we will decompose the class $C_0(L;I)$ into the sum of two classes $C^\pm_0(L;I)$, in a way somewhat reminiscent to the decomposition of a Fourier series into the sum of its analytic and anti--analytic parts. Then, we will describe the image of the singular transform on each of the parts $C_0^\pm(L;I)$. This description is valid for any Denjoy weights $L$, both quasianalytic and non-quasianalytic, satisfying a very mild growth bound $L(\rho)\lesssim\exp(\log^a \rho)$ with some $a<\tfrac{1}{2}$ (see the discussion after Theorem \ref{thm:spliting}).
 \begin{figure}[h!]
	\centering
	\includegraphics[scale=0.7]{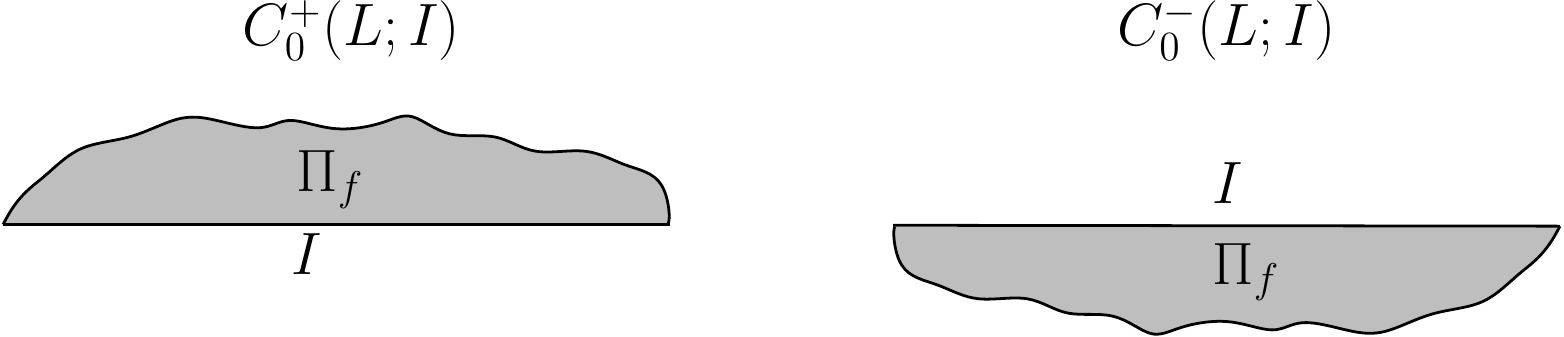}
	\caption{The classes $C_0^\pm(L;I)$}
	\label{fig: The classes Cpm}
\end{figure}
 \begin{definition}
 	Let  $I$ be  an interval and  $L:[0,\infty) \to [1,\infty)$  a non--decreasing function with $\lim_{\rho\to\infty}{L(\rho)}=\infty$. Then, 
 	$C_0^+(L;I)$ is the class of all functions $f\in C_0(L;I)$ for which there exists a domain $\Pi=\Pi_f\subset\{z:\; \im(z)>0\} $ with $\mathbb{R}\cap \partial \Pi=\text{clos}(I)$ such that $f\in C^\infty(\Pi\cup I)\cap \text{Hol}(\Pi).$  Similarly, $C_0^-(L;I)$  is the class of all functions $f\in C_0(L;I)$ for which there exists a domain $\Pi=\Pi_f\subset\{z:\; \im(z)<0\} $ with $\mathbb{R}\cap \partial \Pi=\text{clos}(I)$ such that $f\in C^\infty(\Pi\cup I)\cap \text{Hol}(\Pi).$
 \end{definition} 
Note that the Borel map may be  injective in the  classes $C_0^\pm(L;I)$ even if in the original Beurling class $C_0(L;I)$ it is not. In fact  (see \cite{Korenbljum}, \cite{Carleson3} and \cite{Salinas}), the Borel map is   injective in the  classes $C_0^\pm(L;I)$ if and only if 
\[ \int^\infty \frac{du}{\sqrt{u L(u)}}=\infty. \]

 Also note that  $C_0^+(L;I)\cap C_0^-(L;I)=C^\omega(I)$ is the class of all real--analytic functions on $I$, and  that it is not hard  to show (cf. Section \ref{subsubsec: dec} below) that 
 \[ C_0(L;I)=C_0^+(L;I)+C_0^-(L;I). \]
\subsubsection{Modifying  the regular transform.} Fix a Denjoy weight function $L$ and put 
	\[ \gamma(s)=L(s)^s,\quad K(t)=\frac{1}{2\pi i}\int_{c-i\infty}^{c+i\infty} t^{-s}\gamma(s)ds ,\quad c>0.\] 
	It follows from \cite{KiroSodin} (see also Section 6 below) that there exists $r_0>0$ such that   $K$ is analytic in the set 
 \[ \Omega:=\left\{z: \log z=\log L(s)+\varepsilon(s), |\arg (s)|\leq \tfrac{\pi}{2}, |z|>r_0\right\}, \quad \varepsilon(s)=\frac{s L'(s)}{L(s)} \]
 and is $o(|z|^{-n})$ as $z\to\infty$ uniformly therein, for any $n>0$. 
 Therefore, by Cauchy's Theorem,  if $\Psi$ is any curve in the right half-plane, joining 0 and $\infty$, such that $\Psi\cap\{|z|>r_0\}\subset \Omega$, 
 then 
\begin{equation}
\int_0^\infty t^n K(t)dt=\int_{\Psi} z^n K(z) dz=\gamma(n+1), \quad n\in\mathbb{Z}_+.\label{eq: genrlaized moments}
\end{equation}
 
 We denote by $\Psi_+$ and by $\Psi_-$ two curves joining 0 and $\infty$ in the first and fourth quadrants respectively, such that   $\Psi_\pm\cap\{|z|>r_0\}\subset\partial \Omega$ (i.e., $\Psi_\pm$  coincide with the upper and lower parts of $\partial{\Omega}$) (see Figure \ref{fig: The curves PsiPm}).  
 \begin{figure}[h]
 	\centering
 	\includegraphics[scale=0.7]{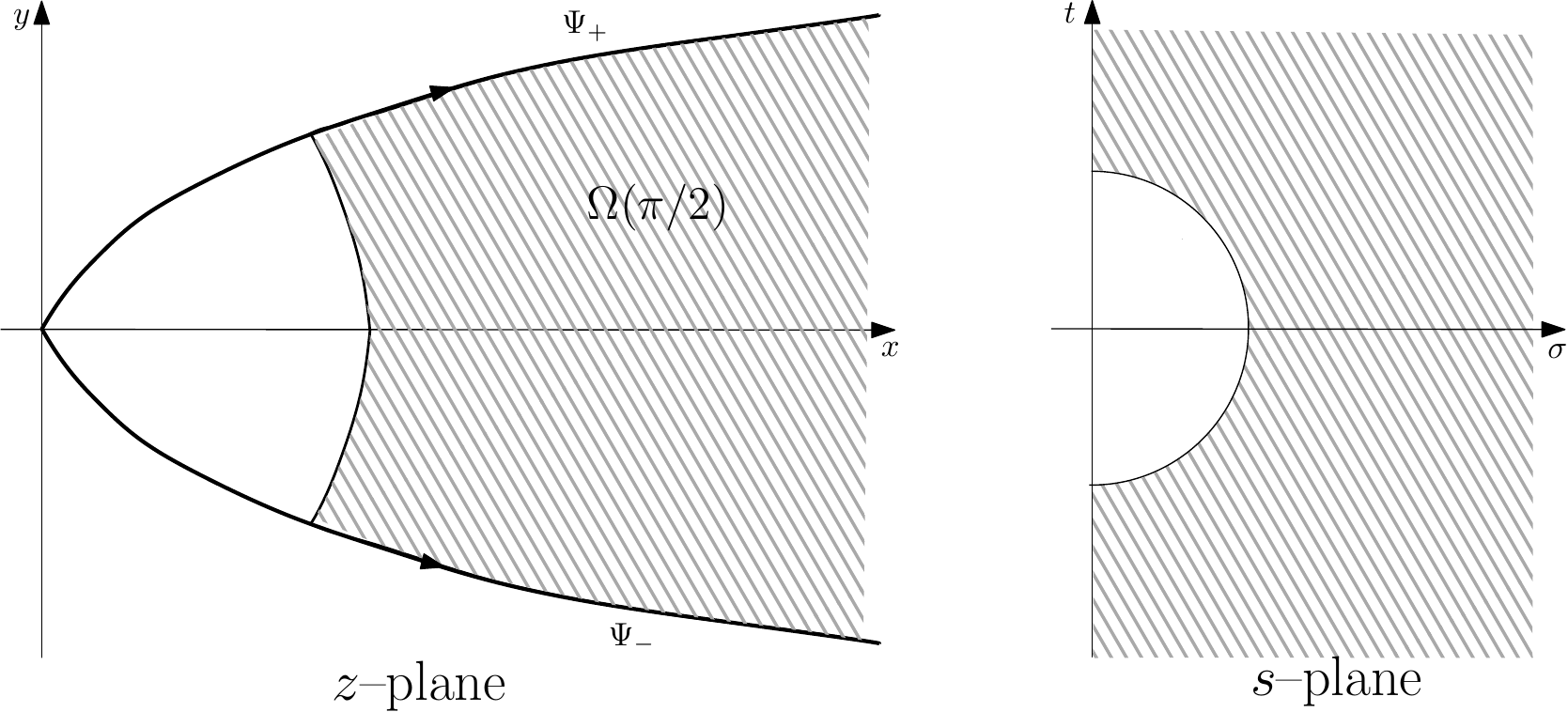}
 	\caption{The curves $\Psi_\pm$.}
 	\label{fig: The curves PsiPm}
 \end{figure}
 We modify the  regular transform, putting
\[  (R_L^+ F)(t):=\int_{\Psi_\pm} F(tz)K(z)dz,\quad \pm t\geq 0,\]
and 
\[  (R_L^- F)(t):=\int_{\Psi_\mp} F(tz)K(z)dz,\quad \pm t\geq 0.\]
whenever the integrals on the  right-hand sides converge.

It follows from \eqref{eq: genrlaized moments} that 
 for any polynomial $P$, $$R_L P=R_L^+ P=R_L^- P.$$
So,   we have $R_L^\pm S_L =\text{Id}$ in the space of all polynomials. We will see that $R_L^\pm S_L =\text{Id}$ in $C_0^\pm(L;I)$ as well.
\subsubsection{The spaces of entire functions $A^\pm(L;I)$.}
 Fix a Denjoy weight  $L$.   It is easy to see, that there exists $\delta>0$ and $\rho_0>0$ such that for $0<\psi<2\delta$, there exist a unique solution to the equation
 \[ \psi=\im\left(\log L(i\rho)+\varepsilon(i\rho)\right),\quad \rho>\rho_0. \]
We denote this solution by $\rho=\rho(\psi)$, and put
\[H(\psi)=\exp\left[\re\left(i\rho(\psi)\varepsilon\left(i\rho(\psi)\right)\right)\right],\quad 0<\psi<2\delta.  \]

By choosing the above  $\delta$ sufficiently small, we can take $H$ to be  positive, $C^1$, decreasing and with with $H(0^+)=+\infty$.
By changing the values of $H$ in the interval $(\delta,2\delta)$, if necessary,  we extend $H$ to a positive, $C^1$ and decreasing   function defined in   $(0,\tfrac{\pi}{2})$, and then we extend the domain of definition of   $H$ to  $(0,\pi)$, putting $H(\psi)=H(\pi-\psi)$.
For example, if $L(\rho)=\log^\alpha \rho$, with $\alpha>0$, then
\[ \psi\sim \frac{\pi}{2}\cdot\frac{a}{\log \rho}\quad \text{and} \quad \log \log H(\psi)\sim  \frac{\pi}{2} \cdot\frac{a}{\psi} ,\quad \psi \to 0^+.\]
  \begin{definition}
  	Suppose that $I$ is an open interval containing the origin. The class $A^+(L;I)$ consists of all entire functions $F$ such that  for any $B>0$, $c_\pm\in I$ with $c_-<0<c_+$, there exists  $\Delta>0$ with 
 \begin{equation}
 |F(re^{i\psi}) |\lesssim_{B,c_\pm} H\left(\psi\pm\frac{2B}{r}\right)+E\left(\frac{r}{|c_\pm|}+\Delta r\sin\psi\right)\label{eq:A^+def}
 \end{equation}
  	whenever $0\leq \pm (\tfrac{\pi}{2}-\psi) \leq \tfrac{\pi}{2}+\tfrac{B}{r}$ (asymptotically this is the  upper half-plane ${\im z> -B}$).
  	We also put $A^-(L;I)=\{F: \bar{F}(\bar{z})\in A^+(L;I)\}$.
  \end{definition}
Note that in the case where $L$ has super--logarithmic growth (i.e., $\rho L'(\rho)\to\infty$ as $\rho\to\infty$), in the right-hand side of \eqref{eq:A^+def} the second term dominates above the curve $\Psi_+$, the first term dominates below the curve $\Psi_+$, while on $\Psi_+$ both terms have roughly the same growth. We also mention that in the logarithmic and sub--logarithmic cases, the first term on the RHS of \eqref{eq:A^+def} is always small compared with the second term and therefore can be discarded.

\subsubsection{}
The main result of this part reads as follows.
\begin{inThm}\label{thm:spliting}
	Suppose  $L$ is a Denjoy weight with $L(\rho)\lesssim\exp(\log^a \rho)$ for some $a<1/2$, and   $I$ is an open interval containing the origin. Then, the singular transform $S_L$ maps $C_0^\pm(L;I)$ bijectively onto the space $A^\pm(L;I)$ with  inverse $R_L^\pm$. 
\end{inThm}
We will use the assumption $L(\rho)\lesssim\exp(\log^a \rho)$ for some $a<1/2$ to guarantee that the estimate $\int_{\Psi_+
}|z^n K(z)dz|\leq C^{n+1}\gamma(n+1)$  holds for $n\in \mathbb{Z}_+$. This assumption is essential for our techniques.

  The definition of the classes $A^\pm(L;I)$ can be simplified if we restrict the growth of $L$. For instance, if $L(\rho)\lesssim \log^2 \rho$, then  the inequality \eqref{eq:A^+def}  can be replaced by 
  \[ |F(x+iy)|\lesssim h\left(\frac{\pi}{2}\frac{|x|}{2B+y}\right)+ E\left(\frac{x}{c_\pm } +\Delta|y|\right),\quad y>-B, \;\pm x>0,\]
  where  $h$ is the inverse function to $x\mapsto \frac{1}{\varepsilon(\log(x))}$.  Even more explicitly, if  $L(\rho)=\log^a\rho$ with $1<a<2$, then the inequality \eqref{eq:A^+def}  can be replaced by
\[  |F(x+iy)|\lesssim \exp\exp\left[{\frac{a\pi}{2}\frac{|x|}{2B+y}}\right]+\exp\exp\left[{\left(\frac{x}{c_\pm } +\Delta|y|\right)^{1/a}}\right] ,\quad y>-B, \;\pm x>0, \]
while for $0<a\leq 1$, inequality \eqref{eq:A^+def}  can be replaced by
\[  |F(x+iy)| \lesssim\exp\exp\left[{\left(\frac{x}{c_\pm } +\Delta|y|\right)^{1/a}}\right] ,\quad y>-B, \;\pm x>0. \]
\subsubsection{}
Combined with  the decomposition $C_0(L;I)= C_0^+(L;I)+C_0^-(L;I)$, Theorem \ref{thm:spliting} immediately yields
\begin{corollary}
Suppose that $L$ satisfies the assumptions of Theorem  \emph{\ref{thm:spliting}} and  $I$ is an open interval containing the origin. Then 
\[ S_L C_0(L;I)= A^+(L;I)+A^-(L;I). \] 
Moreover, if $f=f_+ +f_-\in C_0(L;I)$, with $f_\pm \in C_0^\pm(L;I)$, then 
\[ f=R_L^+ S_L f_+ +R_L^- S_L f_- . \]  
\end{corollary}
\noindent
Note  that the last equality holds even if the class $C_0(L;I)$ is non-quasianalytic. 
\subsection{Structure of this work.} 
 The rest of this paper is organized as follows. In the next section we give some applications and examples of our results.  In Section 5 we discuss a variety of regularity assumptions on the function $L$, which are used in the rest of this paper.  Section 6 is devoted  to the functions $K$ and $E$. There, we summarize the results of \cite{KiroSodin}, and state additional estimates under different  regularity assumptions from  Section 5. In Section 7, we introduce the function $E_1$ which is the singular transform of the exponential function, and give a variety of bounds on it, to be used in the proofs of Theorems 3, 4 and 6. Section 8 is devoted to estimates of the singular transform of polynomials which are used in the proofs of Theorems 2, 4, 5 and 7. Finally in Section 9,  we restate our main results in a more general form using the regularity assumptions of Section 5, and present proofs of these results. 
\section{Applications and examples}\label{sec: applictions}
Here we present some applications and  examples to our results. 
\subsection{Real-analytic functions.} As a byproduct of our techniques used in the proof of Theorems 1 and 2, we also provide a  description of the image of the class $C^\omega(I)$, consisting  of real-analytic functions,  under the singular transform.  
\begin{definition}
	Let  $L:[0,\infty)\to [1,\infty)$ be a continuous and eventually increasing function such that $\lim_{\rho\to\infty}L(\rho)=\infty$. Put $\gamma(\rho)=L(\rho)^\rho$  and $E(z)=\sum_{n \geq 0 } \frac {z^n}{\gamma(n+1)}$. Let $I$ be an interval  containing $0$. \textit{The class } 	$A^\omega(L;I)$ consists of all entire functions $F$  such that, for every $c_-\in I\cap (-\infty,0),\: c_+ \in  I\cap (0,\infty)$   there exists  $\Delta=\Delta_{c_\pm}>0$ so that 
	\begin{equation}
	|F(u+iv)|\lesssim_{c_+,c_-}  E \left(\frac{u}{c_\pm}+\Delta|v|\right),\quad \pm u\geq 0.\label{eq:real analytic image def}
	\end{equation}
\end{definition}
\begin{inThm}\label{thm: Real analytic}
	Suppose that $L$ is a Denjoy weight, and that   $I$ is an open interval containing \emph{0}. Then the singular transform maps the class of real-analytic functions $C^\omega(I)$ bijectively onto the set
	$A^\omega(L;I)$. 
	In particular, $R_L S_L f=f$ for any $f\in C^\omega(I)$.
\end{inThm}

The proof of Theorem \ref{thm: Real analytic} is given in Section 9.
 \subsubsection{Summation in the Mittag--Leffler star.}
Recall that a domain $\Omega\subseteq\mathbb{C}$ containing the origin is called \textit{star-shaped} (with respect to the origin), if for any $z\in \Omega$, $[0,z]\subset \Omega$. If $f(z)=\sum_{z\geq 0} \widehat{f}(n)z^n$ is analytic in a neighborhood of the origin, we denote by $\Omega_f$ the largest star-shaped domain to which $f$ can be analytically continued , and call $\Omega_f$ the \textit{Mittag--Leffler star} of $f$. Next, we introduce a class of entire functions. 
\begin{definition}
	Let  $L:[0,\infty)\to [1,\infty)$ be a continuous and eventually increasing function such that $\lim_{\rho\to\infty}L(\rho)=\infty$. Put $\gamma(\rho)=L(\rho)^\rho$  and $E(z)=\sum_{n \geq 0 } \frac {z^n}{\gamma(n+1)}$. Let $\Omega\subseteq\mathbb{C}$ be a star-shaped domain with respect to the origin. \textit{The class } $A^\omega(L;\Omega)$ consists of all entire functions $F$ such that for any $\delta>0$,
	\[ |F(w)|\lesssim_\delta E\left(H_\Omega(w)+\delta|w| \right), \]
	where $H_\Omega(w)=\inf\{\lambda>0:w\in \lambda \Omega\}$ is the Minkowski functional of $\Omega$.
\end{definition}
The next result follows  from Theorem~\ref{thm: Real analytic}.
\begin{inThm}\label{thm:  Mittag-Leffler star.}
	Suppose  $L$ is a Denjoy weight,   and   $\Omega\subseteq\mathbb{C}$ is a star--shape domain with respect to the origin. Then the singular transform  maps the class $\emph{Hol}(\Omega)$ bijectively onto the class
	$A^\omega(L;\Omega)$. Moreover, if $f\in\emph{Hol}(\Omega)$, then $R_L S_L f\equiv f$ in $\Omega$.
\end{inThm}
In particular, the above theorem shows that if $f(z)=\sum_{n\geq0 } \widehat{f}(n) z^n$ has a positive radius of convergence, then $R_L S_L f(z)$ is the analytic continuation of $f$ to its Mittag--Leffler star 
$\Omega_f$.  This  fact is in contrast to the classical  Borel and Mittag--Leffler  moment summations methods (see Section 2.1.2), which usually do not converge to $f$ in the whole Mittag--Leffler star 
$\Omega_f$. For example, the series $\sum_{n\geq 0} z^n$ is Mittag--Leffler $\Gamma(\alpha n+1)$-summable to the function $\frac{1}{1-z}$ only in the domain
\[ \left\{z: \arg(z-1)>\frac{\pi}{2}\alpha\right\}, \]
while, if $L$  is any function satisfying  the conditions of Theorem \ref{thm:  Mittag-Leffler star.}, then
\[ R_L S_L\left(\sum_{n\geq 0} x^n\right)(z)=\frac{1}{1-z},\quad \text{for} \quad  0<\arg(z-1)<2\pi.\]
\subsection{Examples.} \label{subsec: applictions.examples}
\subsubsection{} \label{subsec: applictions.examples.rotation}
Suppose that $L_1$ and $L_2$ are two Denjoy weights satisfying $L_1(\rho)=o\left(L_2(\rho)\right)$, as $\rho\to\infty.$
Put $\gamma_j(\rho)=L_j(\rho)^\rho$, $j=1,2$. Then the sequence 
\[ \widehat{f_\theta}(n)=e^{i n \theta}\frac{\gamma_2(n+1)}{\gamma_1 (n+1)},\quad n\geq 0,\]  
belongs to the punctual image of the class $C_0(L_2;\mathbb{R})$ for  $0<\theta<\pi$, and to the punctual image of the class $C_0(L_2;[0,\infty))$ for $\theta=\pi$. 
\begin{proof}
Put
	\[ E(z)=\sum_{n\geq 0} \frac{z^n}{\gamma_1(n+1)},\quad E_\theta(z)=E(ze^{i\theta}). \]
	Since $L_1$ is a Denjoy weight, the function $E$ is bounded on any sector that does not meet the positive ray $\{z: \delta<\arg (z)<2\pi-\delta\}$ (see Theorem \ref{TheoremE} and Lemma \ref{lem: E complex lemma} in Section 6 below). Thus, for any $0<\theta<\pi$ (respectively  $\theta=\pi$) the function $z\mapsto E_\theta(z)$ is bounded on any horizontal strip $\{x+iy:|y|<C\}$ (respectively, on  any half-strip $\{x+iy:x>0,|y|<C\}$). In particular, $E_\theta\in A(L_2;\mathbb{R})$ for any $0<\theta<\pi$ and $E_\pi\in A(L_2;[0,\infty))$. By  Theorem \ref{thm: Main thm,reg}, for any $0<\theta<\pi$,
\[ f_\theta := R_{L_2} E_\theta \in C_0(L_2;\mathbb{R}),\quad f_\pi := R_{L_2} E_\pi\in C_0(L_2;[0,\infty)). \]
By the definition of the regular transform, $f_\theta$ are the desired functions.
\end{proof}
\subsubsection{}Suppose that $L$ is a Denjoy weight. Then for any entire functions $F$ with a real period, the sequence
\[ \widehat{f}(n)=\frac{F^{(n)}(0)}{n!}\gamma(n+1),\quad n\geq 0, \]
belongs to the punctual image of the class $C_0(L;\mathbb{R})$.

\begin{proof}
	Let $F$ entire function with a real period. Clearly $F$ is bounded on any horizontal strip and therefore it belongs to $A(L;\mathbb{R})$. By  Theorem \ref{thm: Main thm,reg}, $f:=R_L F\in C_0(L;\mathbb{R})$.
\end{proof}
\subsection{Non-extendable Beurling classes.}
Any real-analytic function $f$ in the interval $[0,1)$ can be analytically extended to an interval $(-\delta_f,1)$, where $\delta_f>0$. Moreover, it is known (see for instance  \cite{Sodin}) that for any quasianalytic  function $L$, there exists $f\in C_0(L;[0,1))$ such that $f$ cannot be extended to a function in $C_0(L_1;(-\delta,1))$ with any $\delta>0$ and any other quasianalytic weight function $L_1$.

Using Theorems \ref{thm: Main thm,reg} and  \ref{thm: Main thm,sing},  we can show that  for any Denjoy weight  $L$, quasianalytic or not,  \textit{there exists a function  $f\in C(L;[0,1))$ such that $f$ cannot be extended to a function in $ C(L;(-\delta,1))$ with any $\delta>0$, or, which is the same,
	\[\bigcup_{\delta>0}C_0(L;(-\delta,1))\nsupseteq C_0(L;[0,1)). \]
}This  shows that the extendability question     for  Beurling or Carleman classes is not trivial also in non-quasianalytic classes.
\begin{proof}
	Fix a Denjoy weight  $L$, and put $\varepsilon(\rho)=\tfrac{\rho L'(\rho)}{L(\rho)}$. Let $L_1$ be  another Denjoy weight  satisfying  $L_1(\rho)=o\left(L(\rho)\right)$ and $L_1(\rho)=o\left(\tfrac{1}{\varepsilon(\rho)}\right)$ as $\rho\to\infty$. Put
	\[ E_1(z)=\sum_{n\geq 0}\frac{z^n}{L_1(n+1)^{n+1}},\quad E(z)=\sum_{n\geq 0}\frac{z^n}{L(n+1)^{n+1}} ,\quad \widetilde{E}(z)=\sum_{n\geq 0} \varepsilon(n+1)^{n+1}z^n. \]
For any $\delta >0$ we then  have, 
	\[E(x)+\widetilde{E}(x)\lesssim_\delta E_1(\delta x), \quad x>0.  \]
	 In particular,  $E_1(-z)\notin A(L;I)\cup A(\tfrac{1}{\varepsilon};\mathbb{R})$ for any open interval $I$ containing 0. On the other hand, as we have already mentioned in Section \ref{subsec: applictions.examples.rotation}, it follows from 
	 Theorem \ref{TheoremE} below that the function $E_1(-z)$ is bounded in the  right half-plane. In particular,  $E_1(-z)\in A(L;[0,\infty))$. Put $f=R_L(E(-z))$. By Theorem \ref{thm: Main thm,reg}, $f\in C_0(L;[0,\infty))$, and by  Theorem \ref{thm: Main thm,sing},   $f\notin C_0(L;I)$ for any open interval containing $0$.
\end{proof}
\begin{remark}
	The above assertion is true not only for Denjoy weights. In fact, it is true for any Beurling class $C_0(L,[0,1))$, with 
	\begin{equation}
\log L(\rho)=o(\log \rho),\quad \rho\to\infty. \label{eq: svl}
	\end{equation}
	This  follows from Theorems \emph{\ref{thm': Main thm,reg}}   and \emph{\ref{thm': Main thm,sing}} of Section \emph{9}, since for any  function $L$ satisfying \eqref{eq: svl}, we can find a function $L_1$ satisfying the assumptions of  Theorems \emph{\ref{thm': Main thm,reg}}   and \emph{\ref{thm': Main thm,sing}} and such that $C_0(L;I)\subseteq C_0(L_1;I)$. The proof of this fact is the same as of the preceding assertion.
\end{remark}
\subsection{Functions of class $C_0(L;I)$ with positive or sparse Taylor series.}\label{subsec: postive taylor} It is well known that if the class $C_0(L;I)$ is quasianalytic and a function  $f\in C_0(L;I)$ has  positive (see for instance \cite{Sodin} or \cite{tacklind1936classes}) or lacunary Taylor series (see \cite{carleman} and \cite{Carleson2}), then  $f$ must be analytic in some neighborhood of the origin, i.e., there exists $C>0$ such that $|\widehat{f}(n)|\lesssim C^{n} $. 
We are going to show that  similar (at least in spirit) phenomena occur for  arbitrary Denjoy   weights.
\subsubsection{} Let $I=(-a,a)$ and let  $L$ be a Denjoy weight. Put 
\[ L_{*}(\rho)=\frac{L(\rho)}{\rho L'(\rho)+1}\quad \text{and} \quad \gamma_{*}(\rho)=L_{*}(\rho)^{\rho}
 \]
 (i.e., $L_{*}$ is the harmonic mean of the functions $L$ and $\tfrac{1}{\varepsilon}$).
 Suppose $f\in C_0(L;I)$ satisfies one of the following two conditions:
\begin{enumerate}
	\item $\widehat{f}(n)\geq 0,\quad n\in \mathbb{Z}_+.$
	\item $\displaystyle f\sim \sum_{n\geq 0} {\widehat{f}}(\lambda_n)x^{\lambda_n},\quad \sum_{n\geq 0} \lambda_n^{-1}<\infty.$
\end{enumerate}
Then,
\[ |\widehat{f}(n)|\lesssim a^{-n}\frac{\gamma(n+1)}{\gamma_{*}(n+1)},\quad n\in \mathbb{Z}_+. \]

\begin{proof}
Fix $f\in C_0(L;I)$ and put $F=S_L f$. Since $L_*(\rho)\leq L(\rho)$, we have for any $x>0$
\[ E(x)\leq E_{*}(x), \quad  \text{where} \quad E_{*}(x):=\sum_{n\geq  0}\frac{x^n}{\gamma_{*}(n+1)}. \]
By  Theorem \ref{thm: Main thm,sing}, for any $\delta>0$,
\[ |F(x)|\lesssim_\delta E\left(\frac{1+\delta}{a}x\right)\leq E_{*}\left(\frac{1+\delta}{a}x\right)\quad \]
Assume now that $f$ satisfies one of the  conditions 1 or 2. Therefore, 
 so does $F$, which in turns yields
 \[ |F(z)|\lesssim_\delta E_{*}\left(\frac{1+\delta}{a}|z|\right),\quad z\in\mathbb{C}.\]
 In the first case this follows from the fact that $|F(z)|\leq F(|z|)$ for any entire function with positive Taylor coefficients, while in the second case the estimate follows from a theorem of Anderson and  Binmore~\cite{Anderson}. 
 
 Applying Theorem~\ref{thm:  Mittag-Leffler star.} to the Denjoy weight $L_{*}$, we find that
 $$g:=R_{L_{*}}F\in \text{Hol}(\{|z|< a\}),$$
 which in turn yields $|\widehat{g}(n)|\lesssim a^{-n}$.
 The assertion follows from the relations $$\displaystyle \widehat{g}(n)=\gamma_{*}(n+1)\widehat{F}(n),\quad \widehat{f}(n)=\gamma(n+1)\widehat{F}(n),\quad n\geq 0.$$

\end{proof}

\subsection{}\label{subsec: Ehrenpries}
So far most of our results and  applications may be interpreted  as telling us how ``small'' the punctual image of $C_0(L;I)$ can be as a subset of $\mathcal{F}_0(L)$ (defined in \eqref{Fdef}). The next result goes in some sense in the opposite direction.
\subsubsection{ } Suppose that $L$ is a Denjoy weight. Then for any sequence $(a(n))_{n\geq 1}\in \mathcal{F}_0(L)$ there exist functions $f_1,f_2\in C_0(L;\mathbb{R})$ such that 
\[ a(n)=\widehat{f_1}(n)+i^n\widehat{f_2}(n) \quad n\geq 0.\] 
\begin{proof}
	Put $A(z):=\sum_{n\geq 0} \frac{a(n)}{\gamma(n+1)}z^n$. By definition, this series has infinite radius of convergence, i.e., $A$ is entire. By a theorem of Ehrenpreis  \cite[p.  131]{Ehrenpreis9}, there exist entire functions $F_1$ and $F_2$, bounded on any horizontal strip, such that $A(z)=F_1(z)+F_2(iz)$. In particular $F_1, F_2\in A(L;\mathbb{R})$. Put $f_1=R_L F_1$ and $f_2=R_L F_2$. By Theorem \ref{thm: Main thm,reg}, $f_1,f_2\in C_0(L;\mathbb{R})$. The equality $A(z)=F_1(z)+F_2(iz)$ implies that
	\[ \frac{a(n)}{\gamma(n+1)}=\widehat{F_1}(n)+i^n\widehat{F_2}(n),\quad n\geq 0, \]
		which in turn yields the assertion.
\end{proof}
\subsection{Non-quasianalytic classes with the same image under the singular transform.}
Fix  a non-quasianalytic Denjoy weight $L$ and  denote by $\widetilde{L}$ its quasianalytic dual, i.e.
\[ \widetilde{L}(\rho):=\int_{\rho}^{\infty}\frac{du}{uL(u)},\quad \rho\geq 1. \]
As we mentioned, $\widetilde{L}$ is always quasianalytic. Differentiating the definition of $\widetilde{L}$ yields
\[ \frac{\rho\widetilde{L}'(\rho)+1}{\widetilde{L}(\rho)}=\frac{\rho L'(\rho)}{L(\rho)}, \]
and thus,
\[ \widetilde{L}(\rho)\exp\left(\int_1^\rho\frac{du}{u\widetilde{L}(u)}\right)=C L(\rho),  \]
where $C^{-1}=\int_{1}^{\infty}\frac{du}{uL(u)}$. We note that starting with an arbitrary quasianalytic function $\widetilde{L}$ and defining $L$ by the above equality will always yield a non-quasianalytic $L$.  
 
For $a>0$, consider the function 
\[ L_a= \widetilde{L}\left(\frac{L}{\widetilde{L}}\right)^a. \]
   It follows from the above considerations that  $L_a$ is again a non-quasianalytic and satisfies
   \[ \widetilde{L_a}=a^{-1}\widetilde{L}. \]
   Notice that the Beurling class $C_0(\widetilde{L};\mathbb{R})$ and the Carleman class $C(\widetilde{L};0) $ (defined in Section 3.3) remain unchanged when we replace $\widetilde{L}$ with $a^{-1}\widetilde{L}$. Theorems \ref{thm:NqaBeu}  and \ref{thm:NqaCar}  then admit.
   \begin{corollary}
   	Suppose  $L$  is a non--quasianalytic function  that satisfies the assumptions of Theorem \emph{\ref{thm:NqaBeu}}. Then for any $a>0$,
   	\[S_{L_a} C_0(L_a;I)= S_{L} C_0(L;I),\quad S_{L_a} C(L_a;I)= S_{L} C(L;I).\]
   \end{corollary} 
For example, consider $L(\rho)=\log^2(\rho+e)$. In this case,  $\widetilde{L}(\rho)\asymp \log \rho$ and $L_a(\rho)\asymp \log^{1+a}\rho$. Thus 
\[ S_{\log^{1+a}}C_0(\log^{1+a};I)=S_{\log^2}C_0(\log^2;I)\quad \text{for any}\quad a>0. \]
Note  that, by Corollary 2,  the  classes $S_{\log^{1+a}}C_0(\log^{1+a};I)$, $0<a$,  are in fact $A(\log;\mathbb{R})$.
\subsection{A Phragmén--Lindel\"of type theorem.}
So far  our applications used function theory in order to obtain results about Beurling classes. The next one goes in the opposite direction and provides a non-trivial result about the growth of entire functions in the case where $L$ is quasianalytic with super-logarithmic growth. 
\begin{corollary}\label{cor: cor5}
Suppose  $L$ is a function   satisfying the assumptions of Theorem \emph{\ref{thm:spliting}}, and  $I$ is an open interval containing the origin. Then, $L$ is quasianalytic if and only if 
\[ A^+(L;I)\cap A^-(L;I)=A^\omega(L;I).\]	
\end{corollary}
For instance, consider the case, $L(\rho)=\log(\rho+e)\log\log(\rho+e^e)$, and  $I=(-1,1)$. Suppose that  $F\in A^+(L;I)\cap A^-(L;I)$, i.e., for any $\delta>0$ and $B>0$, there exists  $\Delta>0$, such that 
\[ |F(x+iy)|\lesssim_{\delta,B} \exp\exp\left( {\frac{\pi}{2}\cdot\frac{|x|}{|y|+B}}\right)+\exp\exp\left({\frac{(1-\delta)|x|+\Delta|y|}{\log(|x|+|y|)}}\right),\quad |x|+|y|\geq 2.\]
 Corollary \ref{cor: cor5} implies that in fact  $F\in A^\omega(L;I)$, i.e.,
for any $\delta>0$ there exists  $\Delta>0$, such that 
\[ |F(x+iy)|\lesssim_\delta \exp\exp\left({\frac{(1-\delta)|x|+\Delta|y|}{\log(|x|+|y|)}}\right),\quad |x|+|y|\geq 2.\]
This conclusion is no longer true if we replace $L$ with $\log(\rho+e)\log^\beta\log(\rho+e^e)$ for some $\beta>1$ (i.e., if we replace the $\log$ term in the right-hand side of the above majorants with $\log^\beta$).

\begin{proof}[Proof of Corollary \ref{cor: cor5}]
The Beurling class $C_0(L;I)$ is quasianalytic if and only if $S_L: C_0(L;I)\to \text{Hol}(\mathbb{C})$ is injective. We have already mentioned that 
\[ C_0(L;I)=C_0^+(L;I)+C_0^-(L;I),\quad  C_0^+(L;I)\cap C_0^-(L;I) = C^\omega(I).\]
 By the uniqueness theorem for analytic functions, the restriction of $S_L$ to $ C^\omega(I)$ is clearly injective. Since $S_L$ is also linear, we conclude that $S_L: C_0(L;I)\to \text{Hol}(\mathbb{C})$ is injective if and only if
\[ S_L C_0^+(L;I) \cap S_L C_0^-(L;I)=S_L C^\omega(I). \]
By Theorems and \ref{thm:spliting} and \ref{thm: Real analytic}, the latter holds if and only if 
\[ A^+(L;I)\cap A^-(L;I)=A^\omega(I).\]	
\end{proof}
\section{Regularity assumptions}
Here we present a list of regularity assumptions on the weight $L$ which  will be used below.  These  assumptions are divided into two sets: the first set gathers  assumptions on the behavior  of $L$ for large positive numbers, while the second concerns the behavior of $L$ in the complex plane. We begin with the first set.

Let $L:[0,\infty)\to[1,\infty)$ be a $C^3$, unbounded and eventually increasing function with $L(0)=1$. Put $\ell(t)=\log L(e^t)$ and consider the following regularity assumptions:
\begin{itemize}
	\item[(R1)]$ \ell'(t)=o(1),\quad t\to+\infty$.
	\item[(R2)]  The function $\ell$ is eventually concave.
	\item[(R3)]  The function $\ell'$ is bounded from above and $ \ell''(t)=o(\ell'(t)),\quad t\to+\infty$.
	\item[(R4)]  $ \ell''(t)\ell(t)=o(\ell'(t)),\quad t\to+\infty$.
\item[(R5)] The function $|\ell''|$ is eventually decreasing $ \ell'''(t)\ell(t)=o(|\ell''(t)|),\quad t\to+\infty$.
\item[(R6)]  $ \ell''(t)\log\tfrac{1}{\ell'(t)}=o\left(\ell'(t)\right),\quad t\to+\infty$.
\item[(R7)]  $ \ell'(t)\ell(t)=o(1),\quad t\to+\infty$.
\end{itemize}
Note that assumptions  (R1)--(R6) are met for any Denjoy weight $L$, while assumption (R7)  is fulfilled if and only if $L(\rho)\lesssim\exp(\log^a\rho)$ for some $a<\tfrac{1}{2}$. 

Assumptions (R1)--(R3) are standard. Assumption  (R1) is equivalent to fact that the function $L$ is slowly varying (i.e., $L$ satisfies  \eqref{slowlygrowing}), assumption (R2) to the fact that the function $\varepsilon(\rho)=\frac{\rho L'(\rho)}{L(\rho)}$ is eventually non-increasing, and assumption (R3)  to the fact that the function $\varepsilon$ is slowly varying. On the other hand, assumptions (R4)--(R7) are less standard and are related to the notion of \textit{super-slow variation} (see \cite[\S3.12.2]{Bingham}). We will use them  in the context of Lemma \ref{lem: subReplacmentLem}. 

We will use  assumptions (R5) and (R7)  in the proof of Theorem 6. These assumptions  are the only ones that restrict the growth of $L$ (apart from the natural requirement that $L$ is slowly growing). Namely, if $L$ is an eventually increasing and unbounded  function such that $L(\rho)\lesssim_\delta \rho^\delta$ for any $\delta>0$, then there exists a function $L_{\tt{n}}$ satisfying assumptions (R1)--(R4)  and (R6) and such that $L\lesssim L_{\tt{n}}$.

For the second set of regularity assumptions, we assume that $L$ is analytic and non-vanishing in an  angle
$\{s: |\arg(s)|<\alpha_0\}$ with $\frac{\pi}{2}<\alpha_0\leq\pi$. Put $\varepsilon(s)=\frac{sL'(s)}{L(s)}$, and consider the following assumptions:
\begin{itemize}
	\item[(R8)] $\varepsilon(s)=(1+o(1))\varepsilon(|s|)$ uniformly in $\{s: |\arg(s)|<\alpha_0\}$ as $s\to\infty$.
	\item[(R9)] $s\varepsilon'(s)=(1+o(1))|s|\varepsilon'(|s|)$ uniformly in $\{s: |\arg(s)|<\alpha_0\}$ as $s\to\infty$.  
\end{itemize}	
Note that  assumptions  (R8)--(R9) are met for any Denjoy weight. These  assumptions can always be satisfied by regularization  of the original weight function $L$ as described in \cite{KiroSodin}. For instance, if $L$ satisfies assumptions (R1) and (R3), then the function 
$$ L_{\tt{a}}(s)=\exp\left(s\int_0^\infty\frac{\log L(u)}{(s+u)^2}du\right) $$
satisfies assumptions (R1), (R3) and (R8), with $ L_{\tt{a}}\sim L$ (in particular the corresponding Beurling and Carleman classes coincide). Moreover, if $L$ satisfies any of the assumptions (R1)--(R7),  so does $L_{\tt{a}}$. 
\subsection{Some lemmas about regular functions.}  
Given a function $L:\mathbb{R}_+\to\mathbb{R}_+$, we denote by $\Lambda_L$ the logarithm of the corresponding Ostrowski function, i.e. 
\[ 	{\Lambda}_L(r):=\log \sup_{x>0}\frac{r^x}{\gamma(x)}.\]

Here,  we summarize some properties of the functions $L$ and  $\Lambda_L$  that depend on our regularity assumptions.
\begin{lemma}\label{lem: widetildeE}
	Suppose that the function $L$ satisfy assumptions \emph{(R1)} and \emph{(R3)}. Then the following hold:
	\begin{enumerate}
		\item[\emph{(1)}] $\Lambda_L(e rL(r))\sim  r,\quad r\to\infty;$
		\item[\emph{(2)}] $\Lambda_L(r t)\leq \Lambda_L(r) t,\quad r>r_0, t>1;$
	\item[\emph{(3)}] $\Lambda_L(\lambda r)\sim \Lambda_L(r) ,\quad r\to\infty, \lambda>0; $
		\item[\emph{(4)}] for any $0\leq n\leq k $, $ n \left(\log L(k)-\log L (n)\right)\lesssim k+n.$
	\end{enumerate}
\end{lemma}
\begin{lemma}\label{lem: subReplacmentLem}
	Suppose that $L:[0,\infty)\to[1,\infty)$ is  $C^2$, unbounded and eventually increasing.
	\begin{enumerate}
\item[\emph{(1)}] If   \emph{(R1)}, \emph{(R2)} and \emph{(R4)} are satisfied, then\\ $\varepsilon(\Lambda_L(\rho))\sim \varepsilon(\rho)\sim \varepsilon(\rho L(\rho))$, as $\rho\to\infty$.
\item[\emph{(2)}] If   \emph{(R7)} is  satisfied, then $L(\Lambda_L(\rho))\sim L(\rho)\sim L(\rho L(\rho))$, as $\rho\to\infty$.
\item[\emph{(3)}] If   \emph{(R5)} is  satisfied, then $L(\rho^2|\varepsilon'(\rho)|)\sim L(\rho)$, as $\rho\to\infty$.
\item[\emph{(4)}] If   \emph{(R2)} and \emph{(R6)} are  satisfied, then $\varepsilon(\rho\varepsilon(\rho))\sim \varepsilon(\rho)\sim\varepsilon(\rho/\varepsilon(\rho))$, as $\rho\to\infty$.
	\end{enumerate}	
\end{lemma}
\begin{lemma}\label{lem: LcomplexLemma}
If  $L$  satisfies assumption   \emph{(R8)}, then 
\[ \log L(s)=\log L(\rho)+i\theta \varepsilon(\rho)(1+o(1))\quad s=\rho e^{i\theta},\;|\theta|\leq  \alpha_0-\delta,\; \rho\to\infty. \]  
If  $L$  satisfies assumption \emph{(R9)}, then 
\[ \varepsilon(s)=\varepsilon(\rho)+i\theta \rho\varepsilon'(\rho)(1+o(1))\quad s=\rho e^{i\theta},\;|\theta|\leq  \alpha_0-\delta,\; \rho\to\infty. \] 
\end{lemma}
The proofs of these lemmas are given in Appendix A. 
\section{The functions $K$ and $E$}
This section is devoted to the asymptotics of the functions $K$ and $E$ under different types of regularity assumptions on the function  $L$.  Let $L$ be a function that satisfies assumptions (R3) and (R8). We associate with  $L$ the functions
\[\gamma(s)=L(s)^s,\quad E(z)=\sum_{n \geq 0 }\frac{z^n}{\gamma(n+1)},\quad	K(t)=\frac{1}{2\pi i}\int_{c-i\infty}^{c+i\infty} t^{-s}\gamma(s)ds ,\quad c>0.\]
 Note that under assumption (R8) the asymptotic behavior in the angle $|\arg(s)|\leq \alpha_0$ is determined by the behavior on the positive ray.  Namely,  by Lemma \ref{lem: LcomplexLemma}, we have
\begin{equation}
\log L(s)=\log L(\rho)+i\theta \varepsilon(\rho)(1+o(1)),\quad s=\rho e^{i\theta},\;|\theta|\leq  \alpha_0-\delta,\; \rho\to\infty.\label{eq:1stLasym}
\end{equation}
In particular, $K$ is well defined under these assumptions. 
We begin this section with a summary of the results in \cite{KiroSodin}.
\subsection{The saddle point equation.} The asymptotics  of the functions $K$ and $E$ for large $z$ are determined by the  saddle--point of the function $s\mapsto \log \gamma(s)-s\log z=s\log L(s)-s\log z$, that is, by the equation
\begin{equation}
\log L(s)+s\frac{L'(s)}{L(s)}=\log z.\label{saddlePoint}
\end{equation}
We remark that under the  assumptions  (R3) and (R8), the saddle--point equation can be written more explicitly (see Lemma \ref{lem: LcomplexLemma}), namely
\begin{equation}
\log L(s)+s\frac{L'(s)}{L(s)} =\log L(\rho)+\varepsilon(\rho)+i\left(\theta+o(1)\right)\varepsilon(\rho),\quad s=\rho e^{i\theta}. \label{eq: L saddele point assym}
\end{equation}

For $0<\alpha<\alpha_0$ and $\rho_0>0$, put
\[ S(\alpha, \rho_0)=\{s: \; |\arg(s)|<\alpha,\; |s|>\rho_0\}. \]
Then,	it is not difficult to show that under the assumptions (R3) and (R8), the LHS of the saddle-point equation \eqref{saddlePoint} is a univalent function in $S(\alpha, \rho_0)$ (see \cite[\S1.3]{KiroSodin}).  From here on, we assume that this is the case, and put
\[ \Omega(\alpha)=\left\{z\colon \log z= \log L(s)+s\frac{L'(s)}{L(s)},\, s\in S(\alpha, \rho_0)\right\}. \]
In general, this  is a domain in the Riemann surface of $\log z$, but by   choosing $\rho_0$ sufficiently large, we can treat it as a subdomain of the slit plane $\mathbb{C}\setminus\mathbb{R}_-$, provided that
\[ \limsup_{\rho\to\infty} \varepsilon(\rho)<\frac{\pi}{\alpha}, \]
in particular, whenever $\varepsilon(\rho)=o(1)$, as $\rho\to\infty$ (which is equivalent to the fact that $L$ is slowly varying or to assumption (R1)).

In what follows,
  we  denote  by $s_z=\rho_z e^{{\rm i}\theta_z}$ the unique solution of the saddle-point equation \eqref{saddlePoint}.
\subsection{Asymptotic behavior of the functions  $K$ and $E$}
The next two theorems are proven in  \cite{KiroSodin}.
\begin{ThmKE}\label{TheoremK}
	Suppose that the function $L$ satisfies assumptions \emph{(R3)} and \emph{(R8)}. Then, for any $\delta>0$, the function $K$ is analytic in $\Omega(\alpha_0-\delta)$ and 
	\[K(z)=\left(1+o(1)\right)\sqrt{\frac{s}{2\pi \varepsilon(s)}}\exp\left(-s\varepsilon(s)\right), \quad z\to\infty,\]
	uniformly in $\Omega(\alpha_0-\delta)$. Here $s=s_z$ and the   branch of the  square root is  positive on the positive half-line.
\end{ThmKE}

\begin{ThmKE}\label{TheoremE}
	Suppose that the function $L$ satisfies assumptions \emph{(R3)} and \emph{(R8)}, and that
	\begin{equation}
	\limsup_{\rho\to\infty}\varepsilon(\rho)<2\,. \label{varepsilon}
	\end{equation}
	Then, given a sufficiently small $\delta>0$, we have
	\[zE(z)= \left(1+o(1)\right)\sqrt{2\pi\frac{s}{\varepsilon(s)}}
	\exp\left(s\varepsilon(s)\right)+o(1), \quad z\to\infty,\]
	uniformly in $\Omega(\pi/2+\delta)$, and
	\[zE(z) = o(1),\quad z\to\infty \]
	uniformly  in $\mathbb{C}\setminus \Omega(\pi/2+\delta)$.
	Here, also  $s=s_z$ and the   branch of the  square root is  positive on the positive half-line.
\end{ThmKE}
We note that the conclusion of Theorem \ref{TheoremE} is valid for the values of $z$ on the positive ray  without the additional assumption  \eqref{varepsilon}. That is, if assumption (R3) holds, then  
\[rE(r) =(1+o(1)) \sqrt{2\pi\frac{\rho}{\varepsilon(\rho)}}
\exp\left(\rho\varepsilon(\rho)\right), \quad r\to\infty,\]
where $r=L(\rho)e^{\varepsilon(\rho)}$.
\subsection{Lemmata.}
Here we present   auxiliary results regrading the asymptotics of the functions $K$ and $E$ needed in this work. 
\begin{lemma}\label{lem: K and E ray lemma} 
	Suppose  the function  $L$  satisfies assumption \emph{(R1)}. 
	Then for  any $\eta<1$,
	$$L^{-1}(\eta r)\lesssim_\eta \log E( r)\lesssim L^{-1}(r),$$  where $L^{-1}$ is the inverse function to $L$ in $[C,\infty)$ for sufficiently large $C>0$. In particular, for any $\delta>0$,
	\[ E^2(r)\lesssim_\delta E((1+\delta)r)  \]
\end{lemma}
\begin{lemma}\label{lem: KE-A}
	Suppose  the function  $L$  satisfies assumptions \emph{(R3)} and \emph{(R8)}. Then for any $\delta>0$ there exists $\delta_1>0$ such that 
	\[ E(x\delta_1)E\left(x(1-\delta)\right)|K(x)|\lesssim 1. \]
\end{lemma}

\begin{lemma}\label{lem: E complex lemma} 
Suppose  the function  $L$  satisfies assumptions \emph{(R2)}, \emph{(R3)} and \emph{(R8)}. Then, there exists a $C>0$ such that  $\log |E(z)|=O(\log |z|)$, uniformly in the set
	\[ \big\{r e^{i\psi}\;:\; C\varepsilon\left(L^{-1}(r)\right)\leq |\psi|\leq\pi \big\}. \] 
\end{lemma}
\subsubsection{Asymptotics of $K$ and $E$ in the set $\Omega(\alpha)$.}
Recall the definition of the domain $\Omega(\alpha)$:
\[ \Omega(\alpha)=\left\{z\colon \log z= \log L(s)+s\frac{L'(s)}{L(s)}, |\arg s|<\alpha,\;\rho>\rho_0\right\}. \]
Further, recall that  $\Psi_+$ and  $\Psi_-$ are  two curves joining 0 and $\infty$ in the first and fourth quadrants respectively, such that  for sufficiently large $r_0$, $\Psi_\pm\cap\{|z|>r_0\}\subset\partial \Omega(\tfrac{\pi}{2})$ (i.e., $\Psi_\pm$  coincide with the upper and lower parts of $\partial \Omega(\tfrac{\pi}{2})$).
\begin{lemma}\label{lem: K and E last}
	Suppose  the function  $L$  satisfies assumptions \emph{(R1)}, \emph{(R2)}, \emph{(R3)} and \emph{(R8)}. For any  $\delta>0$ and  $0<\alpha<\tfrac\pi2$, there exists a constant $C>0$, such that 
	$$\int_{r_0}^\infty |z|^n e^{-\delta \rho_z\varepsilon(\rho_z)}d|z|\leq C 2^n\gamma(n+1),\quad n\geq 0, \;z\in \Omega(\alpha),$$
	where  $s_z=\rho_z e^{i\theta_z}$ is related to $z$ by the saddle--point equation.
\end{lemma}
\begin{lemma} \label{lem:KE strong A for spliting.} 
	Suppose  the function  $L$  satisfies assumptions \emph{(R1), (R2), (R3), (R5)} and \emph{(R9)}. Then for any $\delta>0$ there exists $\delta_1>0$ such that 
	\[ E((1-\delta)|z|)|K(z)|\lesssim E(\delta_1 |z|),\quad z\in\Psi_\pm,\;|z|>1. \]
\end{lemma}
\begin{lemma} \label{lem:spliting H intgral K small} 
	Suppose  the function  $L$  satisfies assumptions \emph{(R1), (R2), (R3), (R5), (R7)} and \emph{(R9)}. Then there exists a $C>0$, such that
	\[ \int_{\Psi_+\cap\{|z|>1\}} |z|^n \left|E(z)\right|^{-1/|z|}d|z|\leq C^{n+1}\gamma(n+1).\]
\end{lemma}

Recall that  $H$ is a positive $C^1$-function, decreasing on $(0,\tfrac{\pi}{2})$,  satisfying $H(\pi-\psi)=H(\psi)$ for  $\psi\in(0,\pi)$, and defined for $\psi\in (0,\delta)$ (with $\delta>0$ sufficiently small) by the equations
\[ \psi=\im\left(\log L(i\rho)+\varepsilon(i\rho)\right),\quad \rho>\rho_0, \]
\[H(\psi)=\re\left(i\rho\varepsilon\left(i\rho\right)\right),\quad 0<\psi<\delta.  \]
\begin{lemma}\label{lem: loglog H} 
		Suppose  the function  $L$  satisfies assumptions \emph{(R2), (R3)} and \emph{(R9)}. Then, there exists $A>0$, s.t.
	\[ \log \log H\left(\psi+\frac{A}{r}\right)\leq \log\log H(\psi)-\frac{3}{r},\quad r\geq1,\; 0<\psi\leq\tfrac{\pi}{2}.\]
\end{lemma}
The proofs are given in Appendix B. 
\section{Singular transform of the  exponential, the function $E_1$}
Given a function $L$, denote by $E_1$ the singular transform of the exponential function $x\mapsto\exp x$, i.e.,

\[ E_1(z)=\sum_{n\geq 0} \frac{z^n}{n!\gamma(n+1)},\quad\text{where} \quad \gamma(n)=L(n)^n.\]
This is an entire function of zero exponential type. We will  use the following estimates of the function $E_1$:
\begin{lemma}\label{lem: widetildeE trvial bound}
	Suppose that $L$ satisfies assumption \emph{(R1)}. Then
	there exists a $C>0$ such that $\log |E_1(z)|\leq C\Lambda_L(|z|).$ Here, as before, $\Lambda_L(r)=\sup_{x\geq 0} \left[x\log r-x\log(x L(x))\right]$.
\end{lemma}
\begin{lemma}\label{lem: widetilde E 1st}
	Suppose that $L$  satisfies assumptions \emph{(R1),(R2), (R4)} and \emph{(R8)}.   Then 
	\[\log |E_1(ix)|\asymp \Lambda_L(x)\varepsilon(x),\quad x>1. \]
\end{lemma}

	 \begin{lemma}\label{lem: mainSpilitEstimate+}
Suppose that $L$ satisfies assumptions \emph{(R1), (R2), (R3)} and \emph{(R9)}.  For any $B>0$, there exists $C_B>0$, such that 
	\[ \big|E_1(i t r e^{i\psi})\big|\leq e^{C_B(\Lambda_L(t)+1)}H\left(\psi+\frac{B}{r}\right) ,\quad r>B,\;B/r<\psi<\tfrac{\pi}{2},\; t>1. \]
\end{lemma}

The proofs are given in Appendix C.
\section{Singular transforms of polynomials}
In this section, we  prove estimates on the singular transform of bounded polynomials on an interval $(a,b)$. All the estimates  will follow from the next lemma.
\begin{lemma}\label{lemma:estimatpsi} Suppose that $L$ satisfies assumptions \emph{(R2), (R3)} and \emph{(R8)}. 
	Let  $P$ be  a polynomial   of degree $n$ such that $|P|\leq 1$ on $[a,b]$ \emph{(}$a<0<b$\emph{)}.  There exists $R_0>0$, such that  for any $a<c_-<0<c_+<b$ and any $R>\max\left\{R_0,|v|,\frac{u}{c_-},\frac{u}{c_+}\right\}$, we have
	\[\log|\left(S_L P\right)(u+iv)|\leq C_{c_-,c_+}\frac{n}{R}\left(|u|\varepsilon\left(L^{-1}(R) \right)+|v|\right)+\log E(R)+C\log R.\]
\end{lemma}
In order to prove Lemma \ref{lemma:estimatpsi}, we will need the following lemmas. The first one is a classical result of S. Bernstein. 
The proof of Lemma  \ref{lemma:ellipse} follows by inspection of 
 Figure \ref{fig: ellipse}.
\begin{lemma}[Bernstein]\label{lemma:Bernstein}
	Suppose that $P$  is a polynomial of degree $n$ such that $|P|\leq1$ on $[a,b]$. Then 
	\[\left|P(z)\right|<\rho^n, \quad z\in\ \mathbf{T}_{\rho}(a,b),
	\]
	where $\mathbf{T}_{\rho}(a,b)$  is  the ellipse with  foci  $a$ and $b$, and the sum of axes $\rho(b-a)$. 
\end{lemma}
\begin{lemma}\label{lemma:ellipse}
	Suppose that $0<\eta<1$ and $a<0<b$. For any $a<c_-<0<c_+<b$, we have 
	\[
	\left[c_-,c_+\right]\times\left[-\eta,\eta\right] \subseteq \mathbf{T}_{1+C_{c_-,c_+}\eta}(a,b).
	\]
	Here  $\mathbf{T}_{\cdot}(\cdot,\cdot)$ is the ellipse defined as in Lemma \emph{\ref{lemma:Bernstein}}. 
\end{lemma}
	\begin{figure}[h]
		\centering
	\includegraphics{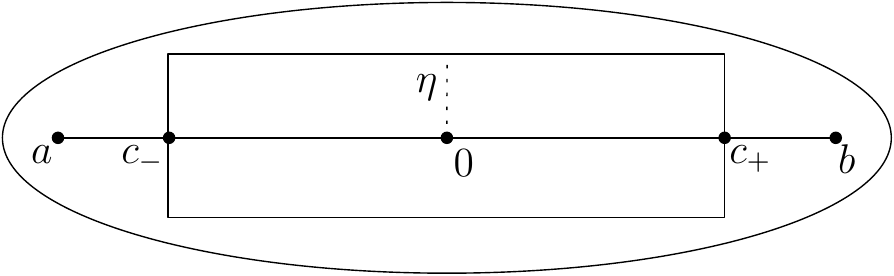}
	\caption{$\left[c_-,c_+\right]\times\left[-\eta,\eta\right] \subseteq \mathbf{T}_{1+C_{c_-,c_+}\eta}(a,b) $}
	\label{fig: ellipse}
\end{figure}
\begin{proof}[Proof of Lemma~\ref{lemma:estimatpsi}] 
	Fix $a<c_-<0<c_+<b$  and let $[c_-,c_+]\subset(c'_-,c'_+)\subset[a,b]$. Put
	\[	Q(w):=\left(S_L P\right)(w)=\sum_{k=0}^{n} \widehat{P}(k) \frac{w^k}{\gamma(k+1)}.
	\]
	By the Cauchy integral formula, 
	\[ Q(w)=\frac{1}{2\pi i} \int E(s)P\left(\frac{w}{s} \right)\frac{ds}{s},
	\]
	where the contour of integration encloses the origin. Given a sufficiently large positive $R$, we deform the contour to the one as in Figure \ref{fig: Curve for polynomials}, i.e., 
	\[\Gamma_R= \big\{Re^{i\theta}:\;|\theta|\leq \theta_R \big\} \cup\big\{R'e^{i\theta_R}:\;R\leq R'\leq R^2 \big\} \cup\big\{R^2 e^{i\theta}:\;\theta_R\leq|\theta|\leq\pi \big\},
	\]
	where, $\theta_R=C\varepsilon(L^{-1}(R))$ and $C$ is chosen as in Lemma~\ref{lem: E complex lemma}. By Lemma~\ref{lem: E complex lemma},
	\begin{equation}
	|E(s)|\lesssim |s|^a,\quad s\in \Gamma_R\setminus \big\{Re^{i\theta}:\;|\theta|\leq \theta_R \big\} \label{firstE}
	\end{equation}
	for some $a>0$. 
	Since the Taylor coefficients   of $E$ are all positive,
	\begin{equation}
	\max_{|\theta|\leq\theta_R}|E(R e^{i\theta})|\leq E(R). \label{seoundE}
	\end{equation}
	If $w=u+iv$ is such that  $\displaystyle u>0, \quad R>\max\big\{R_0,\;\tfrac{u}{c_+},\;|v|\big\},$ where $R_0$ is  large  enough
	and $s=Re^{i\theta}$, $|\theta|\leq \theta_R$,  then 
	\[\re \frac{w}{s}=\frac{u\cos\theta+v \sin\theta}{R}<c_+ +\theta_R \leq c'_+,\]
	\[\re \frac{w}{s}>-\theta_R>-c'_-,\]
	and
	\[\left|\im \frac{w}{s}\right|=\frac{u\sin\theta-v\cos\theta}{R}\leq \frac{u\theta_R}{R}+\frac{|v|}{R}.\]
	If $s=R^2 e^{i\theta}$, $\theta_R\leq|\theta|\leq \pi$, then clearly 
	\[ c'_- \leq\re \frac{w}{s}\leq c'_+\quad \text{and} \quad  \left|\im \frac{w}{s}\right|\leq \frac{u\theta_R}{R}+\frac{|v|}{R}.\]
	Since the set $\left[ c'_-, c'_+\right]\times 
	\left[-\frac{u\theta_R+|v|}{R},\frac{u\theta_R+|v|}{R}\right]$ is convex, we conclude that
	\begin{figure}[t]
		\includegraphics[scale=0.75]{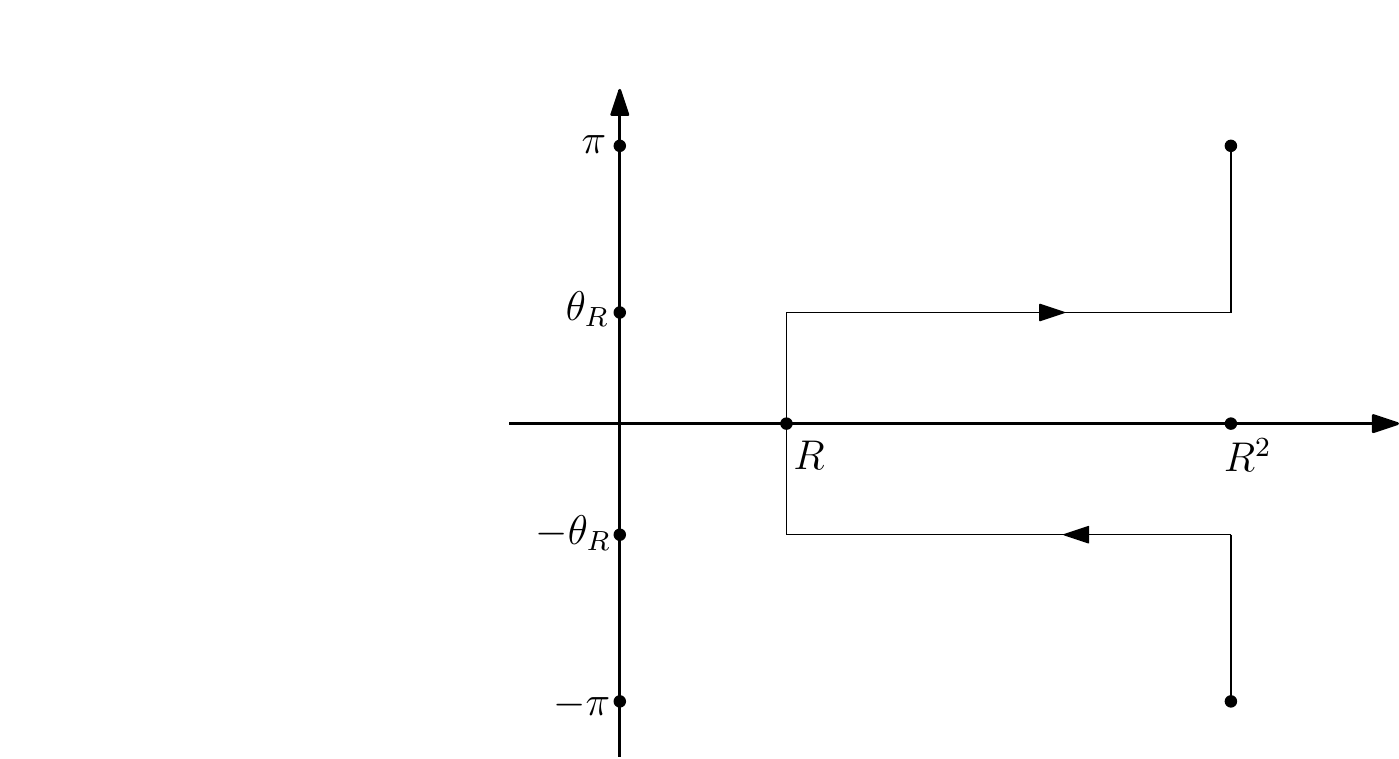}
		\caption{$\Gamma_{R} $}
		\label{fig: Curve for polynomials}
	\end{figure}
	\begin{equation}
	\frac{w}{s}\in\left[ c'_-, c'_+\right]\times 
	\left[-\frac{u\theta_R+|v|}{R},\frac{u\theta_R+|v|}{R}\right],\quad s\in\Gamma_R .\label{w/s}
	\end{equation}
	Therefore, by Lemma \ref{lemma:ellipse}
	\begin{equation}
	\frac{w}{s} \in \mathbf{T}_{1+C_{c_+,c_-}(u\theta_R+|v|)/R}(a,b). \label{Change}
	\end{equation}
	Then, according to Lemma \ref{lemma:Bernstein} applied to $P(\tfrac{w}{s})$, we have
	\[| Q(w)|\leq \exp\left[C_{c_+,c_-}\frac{n}{R}\left(u\theta_R+|v|\right)\right] \int_{\Gamma_R} |E(s)|\,\frac{ds}{|s|}.
	\]
	For $R $ large enough, estimates \eqref{firstE} and \eqref{seoundE}  yield
	\[	
	\int_{\Gamma_R} |E(s)|\frac{ds}{|s|} \leq 2\pi R \theta_R E(R)+C R^{2a}.
	\]
	Therefore
	\[\log | Q(u+iv)| \leq C_{c_+,c_-}\frac{n}{R} \left(u\theta_R+|v|\right)+\log E(R)+2(a+1)\log R+C.\]
	The last   estimate finishes the proof for  $u\geq 0$. The proof for   $u<0$ is similar.
\end{proof}
We finish this section with three estimates for  singular transforms of polynomials, which are based on Lemma \ref{lemma:estimatpsi} 
\begin{lemma}\label{lemma:Co1}
Suppose that the function $L$ satisfies assumptions  \emph{(R1),  (R2), (R4)} and \emph{(R8)}.
	If  $P$ is a polynomial of degree $n$ with $|P|\leq 1$ on $[a,b]$ ($a<0<b$), then for any  $a<c_-<0<c_+<b$, $Y>0$ and $\delta >0$,
	\[|\left(S_L P\right)(u+iv)|\lesssim e^{C_{c_\pm,Y,\varepsilon}\Lambda_L(n)}E\left(\frac{u}{c_\pm}\right)\widetilde{E}(\delta|u|),\quad 0<\pm u\;,|v|<Y,\]
	where $\widetilde{E}(z)=\sum_{n\geq 0} \varepsilon(n+1)^{n+1}z^n$.
\end{lemma}
\begin{proof}
	Fix the parameters $Y$, $a$, $b$ ,$\delta,$ $c_-$ and $c_+$ and  let $\eta>1$ be such that $[\eta  c_-,\eta c_+]\subset (a,b)$.  By Lemma~\ref{lemma:estimatpsi} applied with the  parameters $\eta c_+$ and $\eta c_-$ (instead of $c_-$ and $c_+$), there exists a large numerical constant $R_0>0$, such that 
	\[\log|\left(S_L P\right)(u+iv)|\leq C_{c_-,c_+,\eta}\frac{n}{R}\left(|u|\varepsilon(L^{-1}(R))+|v| \right)+\log E(R)+C\log R,\]
	for $R\geq\max\left\{R_0,|v|,\frac{u}{\eta c_-},\frac{u}{\eta c_+}\right\}$.
	Suppose that $u>0$ and   $|v|\leq Y$. First, we consider the case when  $u\leq\min\left\{ \eta c_+ L\left(\Lambda_L(n)\right),\tfrac{2e}{\delta \varepsilon(n)}\right\}$. Then, we choose $R=L\left(\Lambda_L(n)\right)$ and find that  
	\[\log|\left(S_L P\right)(u+iv)|\leq C_{c_-,c_+,\eta}\frac{n}{L\left(\Lambda_L(n)\right)}\left(|u|\varepsilon\left(\Lambda_L(n)\right)+|v|\right)+2\log E\left(L\left(\Lambda_L(n)\right)\right).\]
Since assumption (R4) holds,  so does assumption (R3). Thus, by Lemma \ref{lem: widetildeE}, assertion 1,
	 we find that
	\[\log|\left(S_L P\right)(u+i v)|\leq C_{Y,c_-,c_+,\eta,\varepsilon} \Lambda_L(n)\left(\frac{\varepsilon\left(\Lambda_L(n)\right)}{\varepsilon(n)}+1\right)+2\log E\left(L\left(\Lambda_L(n)\right)\right).\] 
	By Lemma \ref{lem: subReplacmentLem}, assertion 1, 
 we have $\varepsilon(n)\asymp \varepsilon\left(\Lambda_L(n)\right)$ and   Lemma~\ref{lem: K and E ray lemma} yields, $$\log E\left(L\left(\Lambda_L(n)\right)\right)\lesssim \Lambda_L(n).$$ Therefore, in this case 
	\[\log|\left(S_L P\right)(u+i v)|\lesssim_{c_-,c_+,\eta,\varepsilon} \Lambda_L(n). \]
	Now, if $u\geq \min\left\{ \eta c_+ L\left(\Lambda_L(n)\right),\tfrac{2e}{\delta \varepsilon(n)}\right\}$, then we choose $R=\frac{u}{\eta c_+}$ and find that 
	\[\log|\left(S_L P\right)(u+i v)|\leq C_{c_-,c_+,\eta}\frac{n}{u}\left(u\varepsilon(L^{-1}(\tfrac{u}{\eta c_+}))+|v|\right)+\log E\left(\frac{u}{\eta c_+}\right)+C_{c_-,c_+,\eta} \log u .\]
	If  $u>\tfrac{2e}{\delta \varepsilon(n)}$,  then 
	\[ C_{c_-,c_+,\eta}  \varepsilon\left(L^{-1}\left(\frac{u}{\eta c_+}\right)\right) n\leq  C_{c_-,c_+,\eta}  \varepsilon\left(L^{-1}\left(\frac{u}{\eta c_+}\right)\right) \left(\frac{1}{\varepsilon}\right)^{-1}\left(\frac{\delta u}{e}\right). \]
	where $\left(\frac{1}{\varepsilon}\right)^{-1}$ is the inverse function to $\rho\mapsto\tfrac{1}{\varepsilon}(\rho)$ (defined for $\rho>\rho_0$ large enough).
	Applying  Lemma~\ref{lem: K and E ray lemma} to the  function $\widetilde{E}$ (instead of $E$), we get
	\[  \left(\frac{1}{\varepsilon}\right)^{-1}\left(\frac{\delta u}{2e}\right)\leq  C+ \log \widetilde{E}(\delta u/2). \] 
	Therefore, in this case
	\[ \log|\left(S_L P\right)(u+i v)|\leq C(1+\log u)  +\log \widetilde{E}(\delta u/2)+ 2\log E\left(\frac{ u}{\eta c_+}\right). \]
	On the other hand, if $u\leq \tfrac{2e}{\delta \varepsilon(n)}$, then 
	\begin{multline*}
\log|\left(S_L P\right)(u+i v)|\lesssim_{Y,c_-,c_+,\eta,\delta} \frac{n}{L\left(\Lambda_L(n)\right)}\left(\frac{\varepsilon\left(\Lambda_L(n)\right)}{\varepsilon(n)}+1\right)+\log E\left(\frac{ u}{\eta c_+}\right)\\\lesssim \Lambda_L(n)+\log E\left(\frac{ u}{\eta c_+}\right),
	\end{multline*}
	where we have used once again that $\varepsilon(n)\asymp \varepsilon\left(\Lambda_L(n)\right)$.
	
	We have established that
	\[ |\left(S_L P\right)(u+iv)|\lesssim e^{C_0\Lambda_L(n)}E^2\left(\frac{u}{\eta c_+}\right)\widetilde{E}\left(\delta|u|/2\right)u^{C},\quad 0< u\;,|v|<Y. \]  
	By \ Lemma~\ref{lem: K and E ray lemma} applied separately to the functions $\widetilde{E}$ and $E$, we find that 
 $$\widetilde{E}\left(\delta|u|/2\right)u^{C}\lesssim_{C,\delta} \widetilde{E}\left(\delta|u|\right),\quad  E^2\left(\frac{u}{\eta c_+}\right)\lesssim_\eta E\left(\frac{u}{ c_+}\right).$$
We conclude that
	\[|\left(S_L P\right)(u+iv)|\lesssim e^{C_0\Lambda_L(n)}E\left(\frac{u}{ c_+}\right)\widetilde{E}\left(\delta|u|\right),\quad 0< u\;,|v|<Y. \]  
	This completes the proof of Lemma \ref{lemma:Co1}  for  $u\geq 0$. The proof for  $u\leq0$ is similar.\end{proof}
 \begin{lemma}\label{lemma:Co2}
 	Suppose that the function $L$ satisfies assumptions  \emph{(R1),  (R2), (R4)} and \emph{(R8)}. If $P$ 
 	is a polynomial of degree $n$ with $|P|\leq 1$ on $[a,b]$ ($a<0<b$), then, for any  $a<c_-<0<c_+<b$ and  $0<\delta<1$, there exists  $\Delta>0$ such that  
 	
 	\[|(S_L P)(u+iv)|\lesssim_{\delta,c_-,c_+} e^{\delta n}E\left(\frac{u}{c_\pm}+ \Delta |v|\right),\quad \pm u\geq0.\] 
 	
 \end{lemma}
 \begin{proof}
 	Fix the parameters $a,b,c_-,c_+$ and $\delta$. Set $Q=S_L P$ and let $\eta>1$ be such that $[\eta  c_-,\eta c_+]\subset (a,b)$.  Lemma~\ref{lemma:estimatpsi} with parameters $\eta c_+$ and $\eta c_-$ (instead of $c_-$ and $c_+$) shows that there exists a sufficiently large numerical constant $R_0$ such that
 	\[\log|Q(u+iv)|\leq C_{c_-,c_+,\eta}\frac{n}{R}\left(|u|{\varepsilon(L^{-1}(R))}+|v|\right)+\log E(R)+C\log R,\]
 	for $R>\max\left\{R_0,|v|,\frac{u}{\eta c_-},\frac{u}{\eta c_+}\right\}$.
 	Suppose that $u\geq 0$ , and that $\frac{u}{\eta c_+}+|v|$ is large enough.
 	For $\Delta>2$ that will be chosen later, we choose $R=\frac{u}{\eta c_+}+\frac{\Delta |v|}{2}$
 	and find that 
 	\small{
 		\begin{multline*}
 		\log|Q(u+iv)|\\\leq C_{c_-,c_+,\eta}\frac{n}{\frac{u}{\eta c_+}+\frac{\Delta |v|}{2}}\left({|u|}{ \varepsilon\left(L^{-1}\left(\frac{u}{\eta c_+}+\frac{\Delta|v|}{2}\right)\right)}+|v|\right)+\log E\left(\frac{u}{\eta c_+}+\frac{\Delta |v|}{2}\right)+C\log(u+\Delta|v|).
 		\end{multline*}}
 	The function $R\mapsto \varepsilon(L^{-1}(R))$ tends to zero as $R\to\infty$. Therefore, we can choose $\Delta$ so large, such that 
 	\[\log|Q(u+iv)|\leq\delta n+\log E\left(\frac{u}{\eta c_+}+\frac{\Delta |v|}{2}\right)+C\log(u+\Delta|v|), \quad |u+iv|>C,\quad u\geq 0,\]
 	and therefore
 	\[\log|Q(u+iv)|\leq\delta n+\log E\left(\frac{u}{ c_+}+\Delta |v|\right), \quad |u+iv|>C,\quad u\geq 0.\]	
 	This finishes the proof for $u\geq 0$, the proof for $u\leq 0$ is similar.
 \end{proof}
\begin{lemma}\label{lemma:estimatpsi--for nqa}
	Suppose that the function $L$ satisfies assumptions  \emph{(R1),  (R2), (R3)} and \emph{(R8)}.
	If  $P$ is  a polynomial   of degree $n$ such that $|P|\leq 1$ on $[-1,1]$, then there exists a constant $C>0$ such that  
	\[ \left|S_L P(u+iv)\right|\leq \exp\left(C\cdot\left(\Lambda_{L/\varepsilon}(n)+1\right)\right),\quad |u|\leq 1,\; |v|\leq \varepsilon\left(\Lambda_{L/\varepsilon}(n)\right) \]
\end{lemma}
\begin{proof}
	Put $Q=S_L P$. Lemma~\ref{lemma:estimatpsi} with parameters $c_\pm=\pm\tfrac12$  shows that there exists a sufficiently large numerical constant $R_0$ such that
		\[\log|Q(u+iv)|\leq C\frac{n}{R}\left(|u|{\varepsilon(L^{-1}(R))}+|v|\right)+\log E(R)+C\log R,\]
	for $R>\max\left\{R_0,|v|,2|u|\right\}$.
	Choosing $R=L(\Lambda_{L/\varepsilon}(n))$ , and assuming that $n$ is sufficiently large, we find that 
	\[  \log|Q(u+iv)|\leq  C\frac{n}{L(\Lambda_{L/\varepsilon}(n))}\cdot\varepsilon\left(\Lambda_{L/\varepsilon}(n)\right) +2\log E(L(\Lambda_{L/\varepsilon}(n))), \]
	where $|u|\leq 1$ and $|v|\leq \varepsilon(\Lambda_{L/\varepsilon}(n))$. By Lemma \ref{lem: K and E ray lemma},  $ \log E(L(\Lambda_{L/\varepsilon}(n)))\leq \Lambda_{L/\varepsilon}(n)$, and by Lemma \ref{lem: widetildeE}, assertion 1, $\frac{n}{L(\Lambda_{L/\varepsilon}(n))}\cdot\varepsilon\left(\Lambda_{L/\varepsilon}(n)\right)\asymp \Lambda_{L/\varepsilon}(n)$.
	Therefore,
	\[\log |Q(u+iv)|\lesssim \Lambda(n),\quad  |u|\leq 1\; |v|\leq \varepsilon(\Lambda_{L/\varepsilon}(n)), \]
 which is the desired estimate.
\end{proof}
\section{Proof of theorems}\label{sec:proof of Theorem 1'}
Here we   state and prove stronger analogues of Theorems \ref{thm': Main thm,reg}--7$^\prime$. The theorems are stated using the regularity assumptions of Section 5, and they are somewhat stronger then the corresponding Theorems 1--7.
\subsection{Theorem \ref{thm: Main thm,reg}}\label{subsec:proof of Theorem 1' part 1}
 \begin{thm}\label{thm': Main thm,reg}
	Suppose that the function $L$ satisfies assumptions \emph{(R3)} and \emph{(R8)}, and that $I$  is an interval containing the origin, such that $I\cap(0,\infty)$ and $I\cap(-\infty,0)$ are open. Then, the  regular transform $R_L$ maps $A(L;I)$ into $C_0(L;I)$.	
\end{thm}
Note that  Theorem  \ref{thm': Main thm,reg} is valid not only for slowly growing functions $L$, but also for functions that grow faster, such as $L(\rho)=\rho^a $, for some $a>0$ (Gevrey classes).
\begin{proof}[Proof  of  Theorem \ref{thm': Main thm,reg}]
	Fix an interval $I$ such that $0\in I$ and  $I\cap (0,\infty)$ and $I\cap (-\infty,0)$ are open. The proof treats the intervals  $I\cap [0,\infty)$ and $I\cap [-\infty,0)$ separately; the two cases are analogous. Thus there is no loss of generality in assuming that $I\subset [0,\infty)$. Fix $F\in A(L;I)$ and put 
	\[f(x)=\left(R_LF\right)(x)=\int_0^\infty F(xt)K(t)dt\]
	(convergence will follow from the proof).
	We wish to show that $f\in C_0(L;I)$. Fix $0<c_+ \in I$ and let $\eta>1$ such that $\eta^2 c_+\in I$.
	By the definition of $A(L;I)$, 
	\[|F(u+iv)|\lesssim_{Y,c_+} E\left(\frac{u}{c_+\eta^2}\right),\quad u>0\;, |v|<Y.\]
	For $u\geq 0$ and $n \geq  0$, Cauchy's formula yields
	\[|F^{(n)}(u)|=\frac{ n!}{2\pi} \left|\int_{|w-u|=Y} \frac{F(w)dw}{(w-u)^{(n+1)}}\right|\lesssim_{Y,c_+} \frac{n!}{Y^n} E\left(\frac{u+Y}{c_+ \eta^2}\right)\lesssim_{Y,c_+} \frac{n!}{Y^n} E\left(\frac{u}{c_+\eta}\right).\] 
For $x\in[0,c_+]$, we further obtain
\begin{equation}
|f^{(n)}(x)|=\left|\int_0^{\infty} t^nF^{(n)}(xt)K(t)dt\right|\lesssim_{c_+,Y} \frac{ n!}{Y^n}\int_0^{\infty} t^n E\left(\frac{t}{\eta} \right)|K(t)|dt .\label{f deri esti}
\end{equation}
By Lemma~\ref{lem: KE-A}, applied with $1-\delta=\tfrac{1}{\eta}$, there exists $\delta_1:=\delta_1(\eta)>0$, such that 
\[  E(\delta_1 t)E\left(\frac{t}{\eta} \right) |K(t)|\lesssim_\eta 1.\]
Taking this into account in \eqref{f deri esti}, we obtain
\[ |f^{(n)}(x)|\lesssim_{c_+,Y} \frac{ n!}{Y^n}\int_0^\infty t^n  \frac{dt}{E ( t\delta_1)}\stackrel{(t\delta_1=u )}{=}\frac{ n!}{Y^n} \delta_1^{-n-1}\int_0^\infty u^n \frac{du}{E(u)} \leq \frac{ n!}{Y^n} \delta_1^{-n-1} \left[C+ \int_1^\infty u^n \frac{du}{E(u)}\right]. \]
By the definition of $E$,
\[ E(u)=\sum_{k\geq 0 }\frac{u^k}{\gamma(k+1)}\geq\frac{u^{n+2}}{\gamma(n+3)}.  \]
Therefore 
\[ |f^{(n)}(x)|\lesssim_{a_+,Y,\eta} \frac{ n!}{Y^n} \delta_1^{-n-1}\left[C+\gamma(n+3)\right]\lesssim_{a_+,Y}\frac{ n!}{Y^n} \delta_1^{-n-1} \gamma(n+3).  \]
Since $\left|\frac{\gamma(n+3)}{\gamma(n)}\right|^{1/n}\sim 1$   and $Y$ can be taken arbitrarily large, the last inequality shows that $f\in C_0(L;[0,c_+])$. Since $c_+\in I$ was arbitrary, we conclude that $f\in C_0(L;I)$. This completes  the proof of  Theorem \ref{thm: Main thm,reg}.	\end{proof}
\subsection{Theorem \ref{thm: Main thm,sing}}\label{subsec:proof of Theorem $1'$ part 2}
 \begin{thm} \label{thm': Main thm,sing}
	Suppose that the function  $L$ satisfies assumptions \emph{(R1), (R2), (R4)} and \emph{(R8)}, and that $I$  is an open interval containing the origin. Then,
	the singular transform $S_L$ maps $C_0(L;I)$ into $\bigcap_{\delta>0}A\left(\frac{L(\rho)}{\delta\rho L'(\rho)+1};I\right)$
\end{thm}
Note that if $L$ is a Denjoy weight, then 
$$\bigcap_{\delta>0}A\left(\frac{L(\rho)}{\delta\rho L'(\rho)+1};I\right)=A(L;I)\cup A(\tfrac{1}{\varepsilon},\mathbb{R}).$$
So Theorem \ref{thm: Main thm,sing} follows from Theorem \ref{thm': Main thm,sing}.
\subsubsection{Chebyshev polynomials expansion for functions in $C_0(L;[-1,1])$.}\label{subsubsec:Chebyshev}
Denote by $T_n(x)=\cos\left((n\arccos(x)\right)$ the Chebyshev polynomials. We will use the following lemma (see \cite[pp.44]{Mandelbrojt}).
\begin{lemma}\label{lemma:estimatcoef}
	 If   $f \in C_0(L;[-1,1])$ with the Chebyshev expansion $f=\sum_{n\geq 0}c_n T_n$. Then  for any $\delta>0$,
	\[|c_n|\lesssim_{f,\delta} \inf_{n\leq r} \frac{n!\gamma(n+1)\delta^n}{r^n}.\]
	Moreover, if $f \in C^\omega([-1,1])$, then there exists $\delta=\delta_f>0$ such that 
	\[\log |c_n|\lesssim e^{-\delta n},\quad n\geq 0 \]
\end{lemma}
The next lemma, enables us to express the majorant of the coefficients of the Chebyshev expansion in terms of the function $\Lambda_L(r)=\sup_{x>0}\left[x\log r-x\log(xL(x))\right]$.
\begin{lemma}\label{lemma:estimatcoef2}
	Suppose that the function $L$ satisfies assumptions \emph{(R1)} and \emph{(R3)}. Then
	\[ \Lambda_L(r)\asymp \log\left(\sup_{n\leq r} \frac{r^n} {n!\gamma(n+1)\delta^n}\right)\asymp\log\left( \sup_{n\in \mathbb{N}} \frac{r^n} {n!\gamma(n+1)\delta^n}\right). \]
	In particular, if 
  $f \in C_0(L;[-1,1])$ with the Chebyshev expansion $f=\sum_{n\geq 0}c_n T_n$, then
	\[ |c_n|\lesssim_{f,\delta}  e^{-\delta^{-1}\Lambda_L(n)}. \] 
\end{lemma}
The proof of this lemma is given in Appendix A.
\subsubsection{Proof of   Theorem \ref{thm': Main thm,sing}}\label{subsubsec:proof of Theorem $1'$ part 2}
\begin{proof}
We fix an open interval $I$ and a function $f\in C_0(L;I)$. Let  $[c'_-,c'_+]\subset I$,   and let $a,b\in\mathbb{R}$, such that $[c'_-,c'_+]\subset( a, b)\subset[a,b]\subset I$. Denote by $\chi=\chi(a,b)$ the linear function that maps the interval $[a,b]$ onto the interval $[-1,1]$. 

The function $f\circ\chi$ belongs to $C_0\left(L;[-1,1]\right)$. By Lemma \ref{lemma:estimatcoef2}, it can be expanded into a series of Chebyshev polynomials with rapidly decaying coefficients, namely
\[ f\circ\chi =\sum_{n\geq 0} c_n T_n,\quad \text{where, for any } \delta_1>0,\quad |c_n|\lesssim_{\delta_1} e^{-\delta_1^{-1}\Lambda_L(n)}. \] 
We claim that the function $F:=\sum_{n \geq 0 }c_n S_L(T_n\circ\chi^{-1}) $ is the singular transform  of $f=\sum_{n\geq 0} c_n T_n\circ\chi^{-1} $. To show this we will use Lemma \ref{lemma:Co1}. 

Note that $T_n\circ\chi^{-1}$ are  polynomials of degree $n$ which are bounded by $1$ on the interval $[a,b]$. Therefore, by Lemma~\ref{lemma:Co1}, for any $Y>0$ and $\delta_2>0$, there exists  a constant $C>0$ such that 
\begin{equation}
|F(u+i v)|\leq E\left(\frac{u}{c'_\pm}\right)\widetilde{E}\left(\delta_2|u|\right) \sum_{n\geq 0} |c_n| e^{C \Lambda_L(n)}\lesssim E\left(\frac{u}{c'_\pm}\right)\widetilde{E}\left(\delta_2|u|\right),\quad |v|<Y, 0\leq u, \label{eq: Last estimate Th1'}
\end{equation}
where we are using $c'_+$ when $+u\geq0$ and $c'_-$ when $-u\geq0$. The same inequality also holds for $u\leq 0$, with $c'_+$ replaced by $c'_-$. 
In particular, this shows that the function $F$ is analytic in any strip $|v|<Y$ (and therefore  is entire). The function $S_L f$ is also entire and has the same Taylor coefficients at the origin as $F$, therefore $S_L f=F$. 

 Using  Lemma \ref{lem: K and E ray lemma}, we find that
\[ E\left(\frac{u}{c'_\pm}\right)\widetilde{E}\left(\delta_2|u|\right)\leq E^2\left(\frac{u}{c'_\pm}\right)+\widetilde{E}^2\left(\delta|u|\right)\lesssim_{\delta_2,\delta} E\left(\frac{u}{c'_\pm}(1+\delta_2)\right)+\widetilde{E}\left(2\delta_2|u|\right). \]

Since $a$, $b$, $c'_\pm$, , $Y$, and $\delta_2$ were arbitrary,   for any $\delta>0$ and  $c_\pm \in I $ (with $c_-<0<c_+$), we get
\[|S_Lf (u+iv) |\lesssim_{\delta,Y,c_\pm} \sum_{n\geq 0}\left(\frac{\delta(n+1)L'(n+1)+1}{L(n+1)}\right)^{n+1}\left(\frac{u}{c_\pm}\right)^n \] 
for all $|v|<Y$ and $\pm u\geq 0$ (where we are using $c_+$ when $+u\geq0$ and $c_-$ when $-u\geq0$). The latter inequality yields   $$S_L f\in \bigcup_{\delta>0} A\left(\frac{L(\rho)}{\delta \rho L'(\rho)+1};I\right).$$  
 This complete the proof of Theorem \ref{thm': Main thm,sing}. \end{proof}
\subsection{Theorem \ref{thm: Real analytic}}\label{subsec:Theorem 4}
\begin{thmRW}
	\label{thm': Real analytic}
	Suppose  the function  $L$ satisfy assumptions \emph{(R1), (R2), (R4)} and \emph{(R8)}, and  $I$  is an open interval containing the origin. Then the singular transform maps the class $C^\omega(I)$ of real analytic functions  bijectively onto the set
	$A^\omega(L;I)$. 
	In particular, $R_L S_L f=f$ for any $f\in C^\omega(I)$.
\end{thmRW}
The proof uses  ideas similar to the ones we used in the proofs of Theorems \ref{thm': Main thm,reg} and \ref{thm': Main thm,sing}.
First we show that if an entire function $F\in A^\omega(L;I)$, then its regular transform belongs to $C^\omega(I)$. Then we use Chebyshev polynomials expansion of a functions  belonging to $C^\omega(I)$ together with Lemma \ref{lemma:Co2}, to show that $S_L f\in A^{\omega}(L;I)$, whenever $f\in C^\omega(I)$. Throughout this section, we fix a function $L$ satisfying the assumptions of Theorem 7$^\prime$. 
\subsubsection{Proof that $R_L F\in C^\omega(I)$ for any entire function $F\in A^{\omega}(L;I)$ }
\begin{proof}
 Fix the  open interval $I$. Let $F\in A^\omega(L;I)$,  that is,  for every $c_-\in I\cap (-\infty,0),\: c_+ \in  I\cap (0,\infty)$,   there exists  $\Delta=\Delta_{c_\pm}>0$ such that 
\[  |F(u+iv)|\lesssim_{c_+,c_-}  E \left(\frac{u}{c_\pm}+\Delta|v|\}\right),\quad \pm u\geq 0. \]
Put 
	\[f(x)=\left(R_LF\right)(x)=\int_0^\infty F(x t)K(t)dt.\] By  Theorem~\ref{thm': Main thm,reg}, $f\in C_0(L;I)$. We will show that $f\in C^\omega(I)$ by showing that it has a positive radius of convergence at any point $x\in I$. Fix $c_-,c_+\in I $ and $\Delta=\Delta_{c_\pm}$ as above.
		For any $0<\delta<1 $, the Cauchy formula yields
		\[|F^{(n)}(u)|=\frac{ n!}{2\pi} \left|\int_{|w-u|=\delta(|u|+1)} \frac{F(w)dw}{(w-u)^{(n+1)}}\right|\lesssim_{\delta,c_+,c_-} \frac{n!}{\delta^n (|u|+1)^n} E\left(\frac{ \pm(1+\delta)u+1}{c_\pm}+\delta \Delta|u|\right),\] 
 for any $\pm u\geq 0$. For any $x\in I$, 
	\[ |f^{(n)}(x)|=\left|\int_0^{\infty} t^nF^{(n)}(xt)K(t)dt\right|.\]
	Therefore, for $x\in (0,c_+)$,
	\[  |f^{(n)}(x)|	\lesssim_{\delta,c_+,c_-} n! \int_0^{\infty} \frac{t^n}{\delta^n (tx+1)^n} E\left(\frac{ (1+\delta)x t+1}{c_+}+\delta \Delta tx\right)|K(t)|dt . \]
We choose $\delta$ so small  that 
	\begin{equation}
	|f^{(n)}(x)|	\lesssim_{c_+,c_-,x}\frac{n!}{\delta^n x^n} \int_0^\infty E((1-\delta)t)|K(t)|dt.\label{eq: Realanalyticmid}
	\end{equation}
	By Lemma~\ref{lem: KE-A}, there exists an $\delta_1>0$ such that  
	$E((1-\delta)t)|K(t)|\lesssim\frac{1}{E(\delta_1 t)}$. In particular, the integral in the RHS of estimate \eqref{eq: Realanalyticmid} converges. We obtain 
\[ |f^{(n)}(x)|	\lesssim_{c_+,c_-,x}\frac{n!}{\delta_1^n x^n},  \]
and hence $f\in C^\omega(0,c_+)$. 
	
To show analyticity at the point $x=0$, we  use  Cauchy's formula once again, and find that
		\[|F^{(n)}(0)|=\frac{ n!}{2\pi} \left|\int_{|w|=\rho} \frac{F(w)dw}{w^{n+1}}\right|\lesssim_{\delta,c_+,c_-} \frac{n!}{ \rho^n} E\left(\frac{2\rho}{ \min\{c_+,-c_-,\Delta^{-1}\}}\right),\quad \rho>0.\]
		Therefore, by choosing $\rho=\delta t$ with $\delta>0$ sufficiently small, we find that 
\[ |f^{(n)}(0)|=\left|\int_0^{\infty} t^nF^{(n)}(0)K(t)dt\right|\lesssim_{\delta}\frac{n!}{\delta^n}\int_0^{\infty} E(\delta t) |K(t)|dt\lesssim_{\delta} \frac{n!}{\delta^n}.\]
We have shown that $f\in C^\omega[0,c_+)$. The proof that $f\in C^\omega(c_-,0]$ is similar. Since $c_-,c_+\in I$ were arbitrary, we conclude that  $f\in C^\omega (I)$.
\end{proof}

\subsubsection{Proof that  $S_L f\in A^\omega(L;I)$ for any $f\in C^\omega(I)$. }\label{subsubsec:Proof of Theorem 4 part 2.}
Now we finish the proof of Theorem 7$^\prime$.
\begin{proof}
Fix an open interval $I$ and a function $f\in C^\omega(I)$. Let  $[c_-,c_+]\subset I$, $\delta>0$  and let $a<c_-<0<c_+<b$, such that $[ a, b]\subset I$. Denote by $\chi=\chi(a,b)$ the linear function that maps the interval $[a,b]$ onto the interval $[-1,1]$. 

The function $f\circ\chi$ belongs to $C^\omega[-1,1]$. By Lemma \ref{lemma:estimatcoef}, it can be expanded into a series of Chebyshev polynomials with fast decaying coefficients, namely
\[ f\circ\chi =\sum_{n\geq 0} c_n T_n,\quad |c_n|\lesssim_\delta e^{-\delta n},\] 
for some $\delta>0$. 

We claim that then the function $F:=\sum_{n \geq 0 }c_n S_L(T_n\circ\chi^{-1}) $ is the singular transform  of $f=\sum_{n\geq 0} c_n T_n\circ\chi^{-1} $. To show this, we will use Lemma \ref{lemma:Co2}. 

Note that $T_n\circ\chi^{-1}$ are indeed polynomials of degree $n$ which are bounded by $1$ on the interval $[a,b]$. Therefore,
\[ |F(u+i v)|\lesssim_{c_-,c_+}  E\left(\frac{u}{c_\pm}+ \Delta |v|\right)\sum_{n\geq 0} |c_n| e^{\delta n/2}\lesssim_{c_+,c_-} E\left(\frac{u}{c_\pm}+ \Delta |v|\right),\quad \pm u>0. \]
In particular, this shows that the functions $F$ entire. The function $S_L f$ is also entire and has the same Taylor coefficients at the origin as $F$, therefore $S_L f=F$. This completes the proof of Theorem 7$^\prime$.
\end{proof}
\subsection{Theorem \ref{thm: example3}}
 \begin{thm}\label{thm': example3}
	Suppose  the function $L$ satisfies assumption \emph{(R1), (R2), (R4)} and \emph{(R8)} and that $\lim_{\rho\to+\infty} \rho L'(\rho)=+\infty$. Then for any function $L_2$ satisfying 	$\tfrac{1}{\varepsilon(\rho)}=o(L_2(\rho))$ as $\rho\to\infty$ and any $\delta>0$,  $S_L C_0(L;(-\delta,\delta))\nsubseteq A(L_2;\mathbb{R})$. In particular  $S_L C_0(L;(-\delta,\delta))\nsubseteq A(L;\mathbb{R})$.
\end{thm}
Here we prove Theorem \ref{thm': example3}. For a function  $L$ satisfying $\lim_{\rho \to \infty }\rho L'(\rho)=\infty$, and another function $L_2$ with $\tfrac{1}{\varepsilon(\rho)} =o(L_2(\rho))$,  we will construct a lacunary Fourier  series
\[ f(x)=\sum_{k\geq 0} e^{-\omega_{n_k} \Lambda_L(n_k)}e^{i n_k x },\quad n_k \in\mathbb{N},\]
where, as before, $ \Lambda_L(r):=\sup_{x>0}\left[ x\log r-x\log (x L(x))\right]$. This Fourier  series defines an element in the Beurling class $C_0(L;\mathbb{R})$ provided that $\omega_{n_k}\to \infty$ as $k\to\infty$ (the proof is below). Then using the linearity of the singular transform, we will show that 
\[ F(z)=\left(S_L f\right)(z)=\sum_{k\geq 0} e^{-\omega_{n_k} \Lambda_L(n_k)} \left(S_L e^{i n_k x }\right)(z) =\sum_{k\geq 0} e^{-\omega_{n_k} \Lambda_L(n_k)} E_1(i n_k z), \]
where $ E_1(z):= S_L(\exp)(z)=\sum_{n\geq 0}\frac{z^n}{n!\gamma(n+1)}.$
The plan is now to choose the sequence $n_k$ so lacunary and $\omega_k$ increasing to $\infty$ so slowly,  that  on a special sequence of points $r_k\uparrow\infty$, we will have 
\[|F(r_k)|\asymp  e^{-\omega_{n_k} \Lambda_L(n_k)} |E_1(i n_k r_k)|> e^{L_2^{-1}(r_k)}.   \]
\begin{proof}
Fix a function  $L$  satisfying  the assumptions of Theorem~\ref{thm': example3}, and another function $L_2$ with $\tfrac{1}{\varepsilon}= o(L_2(\rho))$  as $\rho\to\infty$. 
 Let $L_3$ be yet another function  satisfying  the assumptions of Theorem~\ref{thm: Main thm,sing} and such that 
 $$\tfrac{1}{\varepsilon(\rho)}= o(L_3(\rho)),\quad L_3(\rho)=o\left(L_2(\Lambda_L(\rho))\right),\quad L_3(\rho)=o(L(\rho)),\quad \rho\to\infty. $$
 The function $L_3$ exists, since by Lemma \ref{lem: subReplacmentLem}, assertion 1, $\varepsilon(\rho)\sim \varepsilon(\Lambda_L(\rho))=o\left(L_2(\Lambda_L(\rho))\right)$ as $\rho\to\infty$. 
For a sequence $\omega_n\to\infty$, and a sequence of natural numbers $n_k\uparrow\infty$ that will be chosen later,  put 
\[ f(x)=\sum_{k\geq 0} e^{-\omega_{n_k} \Lambda_L(n_k)}e^{i n_k x }. \]

Clearly $f$ is a $2\pi$--periodic function in $C^\infty(\mathbb{R})$. Moreover, 
\[ f^{(j)}(x)=\sum_{k\geq 0} (in_k)^j e^{-\omega_{n_k} \Lambda_L(n_k)}e^{i n_k x }, \]
which yields
\[ \Vert f^{(j)} \Vert_\infty\leq  \sum_{k\geq 0} n_k^j e^{-\omega_{n_k} \Lambda_L(n_k)}. \]
Since, $\omega_{n}\to\infty$, the definition of $\Lambda_L$  yields
\[e^{-\omega_{n} \Lambda(n)} \leq C_\delta \frac{\delta^{j+2}\gamma(j+2)(j+2)^{j+2}}{n^{j+2}},\quad \delta>0,\; n\in\mathbb{N}.  \]
Thus, 
\[  \Vert f^{(j)} \Vert_\infty \lesssim_\delta \delta^{j+2}\gamma(j+2)(j+2)^{j+2} \lesssim_\delta \delta^{j}\gamma(j)j^{j} ,\quad \delta>0, \]
which means that $f\in C_0(L;\mathbb{R})
$.
 
Now put
\[ F(z)=\sum_{k\geq 0} e^{-\omega_{n_k} \Lambda_L(n_k)} S_L(e^{i n_k x })(z) =\sum_{k\geq 0} e^{-\omega_{n_k} \Lambda_L(n_k)} E_1(i n_k z). \]
By  Lemmas  \ref{lem: widetildeE trvial bound} and \ref{lem: widetildeE}, assertion 2,  there exists $C>0$ such that  
\[ \log |E_1(i n_k z)|\leq C|z| \Lambda_L(n_k)+C,\quad k\geq 0,\; z\in\mathbb{C}.  \]
Since $ \omega_{n_k}\to\infty$, the sum that defines $F$ converges  uniformly on compact subsets of $\mathbb{C}$. Thus, $F$ is an entire function. The entire functions $F$ and $S_L f$ have the same Taylor coefficient, and therefore $F=S_L f$.

Put $r_n=L_3(n)$.  
By Lemma \ref{lem: widetilde E 1st}, there exists  $\delta>0$ such that for sufficiently large $n$ we have 
\[ \log  |E_1(i n r_n)| \geq 2\delta \Lambda_L(n r_n) \varepsilon(n r_n). \]
Fixing this value of $\delta>0$, we put $\omega_n:= \delta  \frac{\Lambda_L(n r_n) \varepsilon(n r_n)}{ \Lambda_L(n)}$.
Let as verify that the sequence $\omega_n$   tends to $\infty$. Indeed, by Lemma \ref{lem: widetildeE}, part 2, for sufficiently large $r$ we have,
\[ \frac{\Lambda_L(n r_n) \varepsilon(n r_n)}{ \Lambda_L(n)}\geq r_n  \varepsilon(n r_n)\geq r_n\varepsilon(n)\to\infty,\quad n\to\infty.  \]
We conclude that 
\begin{equation}
-\omega_{n} \Lambda_L(n)+\log |E_1(i n r_n)|\geq \delta\Lambda_L(n r_n) \varepsilon(n r_n).\label{eq:main term}
\end{equation}
Making use of Lemma \ref{lem: widetilde E 1st} once again, we see that there exists  $A>0$ such that for sufficiently large $n$,
\[ \log  |E_1(i n r_n)| \leq A\cdot \Lambda_L(n r_n) \varepsilon(n r_n). \]
Now,choose the sequence $n_k\uparrow\infty$ so sparse  that 
\[ A\cdot\Lambda_L(n_{k+1}r_{n_k})\varepsilon(n_{k+1}r_{n_{k}}) \leq \frac{1}{2} \Lambda(n_{k+1}r_{n_{k+1}}) \varepsilon(n_{k+1}r_{n_{k+1}}), \]
and 
\[ A\cdot\Lambda_L(n_{k-1}r_{n_k})\varepsilon(n_{k-1}r_{n_k}) \leq \delta \Lambda(n r_{n_k}) \varepsilon(n r_{n_k})-\log(2k).\]

Combining this with \eqref{eq:main term}, we find that 
\begin{align*}
|F(r_{n_m})|&=\left|\sum_{k\geq 0} e^{-\omega_{n_k} \Lambda_L(n_k)} E_1(i n_k r_{n_m})\right|\\&=\left|\sum_{k<m} e^{-\omega_{n_k} \Lambda_L(n_k)} E_1(i n_k r_{n_m})+e^{-\omega_{n_m} \Lambda_L(n_m)} E_1(i n_m r_{n_m})+\sum_{k>m} e^{-\omega_{n_k} \Lambda_L(n_k)} E_1(i n_k r_{n_m})\right|\\
&\geq \left|e^{-\omega_{n_m} \Lambda_L(n_m)} E_1(i n_m r_{n_m})\right|-\left|\sum_{k<m} e^{-\omega_{n_k} \Lambda_L(n_k)} E_1(i n_k r_{n_m})\right|-\left|\sum_{k>m} e^{-\omega_{n_k} \Lambda_L(n_k)} E_1(i n_k r_{n_m})\right|\\
&\geq  e^{\delta \Lambda_L(n_m r_{n_m}) \varepsilon(n_m r_{n_m})}-\frac{m}{2m} e^{\delta \Lambda_L(n_m r_{n_m}) \varepsilon(n_m r_{n_m})}-\sum_{k>m} e^{-\tfrac{\omega_{n_k}}{2} \Lambda_L(n_k)} +O(1)\\
&= \frac{1}{2} e^{\delta \Lambda_L(n_m r_{n_m}) \varepsilon(n_m r_{n_m})}+O(1).
\end{align*}
By Lemma \ref{lem: widetildeE}, part 2, for sufficiently large $m$ we have 
 \[ |F(r_{n_m})|\geq \frac{1}{2} e^{\delta \Lambda_L(n_m r_{n_m}) \varepsilon(n_m r_{n_m})}+O(1)= e^{\delta \Lambda_L(n_m r_{n_m}) \varepsilon(n_m r_{n_m})}\geq e^{\Lambda_L(n_m)}. \]
 Since $n_m=L_3^{-1}(r_{n_m})$, we conclude
 \[ \log|F(r_{n_m})|\geq \Lambda_L\left(L_3^{-1}(r_{n_m})\right) \] 
Fix a large constant $M>0$. Since $L_3(\rho)=o\left(L_2(\Lambda_L(\rho))\right)$, we have
\[ L_2^{-1}(2 M r_{n_m})=o\left(\Lambda_L\left(L_3^{-1}(r_{n_m})\right)\right) ,\quad r\to\infty. \]
Thus,  for sufficiently large $m$, Lemma \ref{lem: K and E ray lemma} yields 
\[ \log|F(r_{n_m})|\geq L_2^{-1}(2M r_{n_m})\geq \log E_2(M r_{n_m}),\quad \text{where } E_2(z):=\sum_{n\geq 0} \frac{z^n}{L_2(n+1)^{n+1}}.  \]
We conclude that $F=S_L f\notin A(L_2,I)$ for any open interval $I$, which completes the proof Theorem \ref{thm': example3}.\end{proof}
\medskip
Note that our construction of the function $f$  is analytic in the upper-half plane and therefore we actually proved that
\[ S_L (C_0^+(L;\mathbb{R}))\nsubseteq A(L_2;I), \]
where the classes $C_0^+(L;\mathbb{R})$ are defined in Section 3.4.
\subsection{Theorem \ref{thm:NqaBeu}.}
 Throughout this section, given a non-quasianalytic eventually  growing function $L$, we put 
\[ \widetilde{L}(\rho)=L(\rho)\int_{\rho }^{\infty}\frac{du}{uL(u)}, \quad \rho>1,\]
and 
\[ \gamma(\rho)=L(\rho)^\rho,\quad \widetilde{\gamma}(\rho)=\widetilde{L}(\rho)^\rho. \]
 \begin{thm}\label{thm':NqaBeu}
Let $L$ be a non-quasianalytic function and let  $I$ be an open interval that contains the origin. Suppose that the function $\rho\mapsto L(\rho)/\widetilde{L}(\rho)$  satisfies assumptions \emph{(R1), (R2), (R4), (R6)} and \emph{(R8)} ,  and that $I$ is an open interval containing the origin. Then
	\[ S_L C_0(L;I)=S_{\widetilde{L}} C_0(\widetilde{L};\mathbb{R}). \]
\end{thm}
Note that $$\rho\frac{d}{d\rho}\log \frac{L(\rho)}{\widetilde{L}(\rho)}=\frac{1}{\widetilde{L}(\rho)},$$
and so, assumption (R2) implies that $\widetilde{L}$ is eventually increasing and part 1 of Lemma \ref{lem: subReplacmentLem} (i.e., assumptions (R1) and (R4)) implies that $\widetilde{L}(\rho)\sim\widetilde{L}(\rho L(\rho)/\widetilde{L}(\rho))$. In particular, under these assumptions $\widetilde{L}$ is eventually slowly growing.

 In this section we will study the singular and regular transforms $S_{L/\widetilde{L}}$ and $R_{L/\widetilde{L}}$. The relevant functions associated with these transforms are
\[ E_*(z)=\sum_{n\geq 0} \frac{\widetilde{\gamma}(n+1)}{\gamma(n+1)} z^n,\quad K_* (z):=\frac{1}{2\pi i}\int_{c-i\infty}^{c+i\infty}\frac{\gamma(s)}{\widetilde{\gamma}(s)}z^{-s}ds,\quad c>0. \]
\subsubsection{Proof of the inclusion $S_L C_0(L;I)\subseteq S_{\widetilde{L}}C_0(\widetilde{L};\mathbb{R})$.}
\begin{proof}
The first observation we make is that the set $S_L C_0(L;I)$ does not depends on the interval $I$. Indeed, for  $\delta>0$, consider the function $\xi_\delta\in C_0(L;\mathbb{R})$ which is  identically  $1$ in the interval $(-\delta,\delta)$ and identically  $0$ outside the interval $(-2\delta,2\delta)$. If $f\in C_0(L;I)$, then $\xi_\delta f\in C_0(L;I)$ and $S_L (f)\equiv S_L(f\cdot \xi_\delta)$. By choosing $\delta$ so that $[-2\delta,2\delta]\subset I$, the function  $f\cdot \xi_\delta$ can be extended to an element of $ C_0(L;\mathbb{R})$.

Fix   $f\in C_0(L;\mathbb{R})$. For $A>0$, we put $f_A(x)=f(Ax)$. Since $f\in C_0(L;\mathbb{R})$, so does $f_A$. We can consider  the Chebyshev series expansion of the function $f_A$ in the interval $[-1,1]$,
\[ f_A=\sum_{k\geq 0} c_{A,k}\cdot T_k. \]
By Lemma \ref{lemma:estimatcoef},
\[ \lim_{k\to\infty}\frac{\log|c_{A,k}|}{\Lambda_L(k)}
=-\infty,\]
where as before $$\Lambda_L(k) = \sup_{x\geq 0} \left[x\log k-n\log(xL(x))\right].$$
We put 
\[ g_A:= \sum_{k\geq 0 } c_{A,k}\cdot S_{L/\widetilde{L}} \left(T_k\right)=\sum_{k\geq 0 } c_{A,k}\cdot Q_k. \]
and notice that, formally, 
\[ \frac{\widehat{f_A}(n)}{\gamma(n+1)}=\frac{\widehat{g_A}(n)}{\widetilde{\gamma}(n+1)},\quad n\geq 0. \]
Next, we   show that  $g_A\in C(\widetilde{L};[-\tfrac{1}{2},\tfrac{1}{2}])$. By Lemma \ref{lemma:estimatpsi--for nqa} (applied with  the functions $L/\widetilde{L}$ and $1/\widetilde{L}$  instead of $L$ and $\varepsilon$),
\[|Q_k(x+iy)|\leq e^{C\Lambda_L(k)} ,\quad |x|\leq 1,\; |y|\leq \frac{1}{\widetilde{L}\left(\Lambda_L(k)\right)}.  \]
Therefore, the Cauchy estimates for the derivatives yield
\begin{equation}
|Q^{(n)}_k(x)|\leq C n! \cdot\widetilde{L}^n\left(\Lambda_L(k)\right)\cdot e^{C\Lambda_Ls(k)} ,\quad |x|\leq \frac{1}{2},\; 0\leq n\leq k.\label{eq: Q der 1st estimate}
\end{equation}
By Lemma \ref{lem: widetildeE}, part 4, there exists $C>0$, such that,
\[n\log \widetilde{L}(k)-n\log \widetilde{L}(n)\leq C(n+k),\quad  0\leq n\leq k, \]
which in turn implies that 
\[n\log \widetilde{L}\left(\Lambda_L(k)\right)-n\log \widetilde{L}(n)\leq C\left(n+\Lambda_L(k)\right),\quad  0\leq n\leq k. \]
Substituting this estimate into \eqref{eq: Q der 1st estimate}, we get \[|Q^{(n)}_k(x)|\leq C_1^{n+1} e^{C\Lambda_L(k)} n! \cdot\widetilde{L}^n(n)  ,\quad |x|\leq \frac{1}{2},\; 0\leq n\leq k\]
Therefore, 
\[ |g_A^{(n)}(x)| \leq C_A\cdot C_1^{n} \cdot n!\cdot \widetilde{\gamma}(n+1),\quad |x|\leq \frac{1}{2},\; n\geq 0,\]
where the constant $C_1$ is independent of $A$, i.e., $g_A\in C(\widetilde{L};[-\tfrac{1}{2},\tfrac{1}{2}])$.

For $x\in [-\tfrac{A}{2},\tfrac{A}{2}]$, put  $g(x)=g_A(\tfrac{x}{A})$. Note that $g$ belongs to the quasianalytic class $C(\widetilde{L},[-\tfrac{A}{2},\tfrac{A}{2}])$, and 
\[ \frac{\widehat{g}(n)}{\widetilde{\gamma}(n+1)}=A^{-n}\frac{\widehat{g_A}(n)}{\widetilde{\gamma}(n+1)}=A^{-n}\frac{\widehat{f_A}(n)}{\gamma(n+1)}=\frac{\widehat{f}(n)}{\gamma(n+1)}. \]
 Since the class $C(\widetilde{L};\mathbb{R})$ is quasianalytic, we conclude that  $g\in C(\widetilde{L};\mathbb{R})$.

Let $B>0$. Taking $A>2B$, we find that
\[ \max_{|x|\leq B} |g^{(n)}(x)|\leq C_A\cdot \left(\frac{C_1}{A}\right)^{n} \cdot n!\cdot \widetilde{\gamma}(n+1) .\]
Since $A$ can be taken arbitrarily large, we have $g\in C_0(\widetilde{L};\mathbb{R})$, with 
\[ \frac{\widehat{f}(n)}{\gamma(n+1)}=\frac{\widehat{g}(n)}{\widetilde{\gamma}(n+1)},\quad n\geq 0. \]
This finishes the proof the inclusion $S_L C_0(L;I)\subseteq S_{\widetilde{L}}C_0(\widetilde{L};\mathbb{R})$. 
\end{proof}
\subsubsection{The inclusion $S_L C_0(L;I)\supseteq S_{\widetilde{L}}C_0(L;\mathbb{R})$.}
\paragraph{The theorem of Carleson and Ehrenpreis.}\label{sec: Car and Ehr thm}
Our proof relies on a theorem independently proven by  Carleson \cite{Carleson} and Ehrenpreis \cite{Ehrenpreis9}. Here we give the statement of this theorem and discuss its relation to our results.

Recall that for a non-decreasing function $L:[0,\infty)\to[1,\infty)$,   $$\Lambda_L(r)= \sup_{\rho \geq 0}\left( \rho\log r-\rho \log (\rho L (\rho)) \right).$$
Note that if the function $\rho\mapsto \rho  \log (\rho L (\rho))$ is a convex function of $\log \rho$, then $L$ can be recovered from $\Lambda_L$ by the relation
\[ \rho \log (\rho L (\rho))=\sup_{r>0}\rho \log r-\Lambda_L(r).\]

\begin{theorem*}
Suppose that   $L$  is  a  non--quasianalytic, eventually increasing and unbounded from above function such that  $\rho\mapsto \rho  \log (\rho L (\rho))$  is a convex function of $\log \rho$. Let  $L_1$ be an increasing function.
	Then 
	\begin{equation}
	\mathcal{F}_0(L_1)=\{(a_n)_n\;:\; |a_n|^{1/n}=o(L_1(n)),\:\: n\to\infty\;\}\subseteq \mathcal{B} C_0(L;\mathbb{R}) \label{eq: C&E}
	\end{equation}
	if and only if 
	\[ \int_0^\infty \frac{r}{r^2+t^2}\Lambda_L(t)dt=O(\Lambda_{L_1}(r)),\quad r\to\infty. \]
\end{theorem*}
We remark that the ``only if'' part is  due to Ehrenpreis and that Carleson proved this theorem for the Carleman classes (though his proof also works for the Beurling classes).

The assertion \eqref{eq: C&E} can  be recast as 
\[ \left\{\left(\frac{a_n}{\gamma(n+1)}\right)_n\;:\; |a_n|^{1/n}=o(L_1(n)),\:\: n\to\infty\;\right\}\subseteq \mathcal{B} S_L C_0(L;\mathbb{R}),  \]
which under the assumption \[ \frac{L_1(\rho+1)}{L_1(\rho)}<C \]
can be also written as
\[ S_{L_1/L} \text{Hol}(\mathbb{C})\subseteq S_L C_0(L;\mathbb{R}). \]
The next lemma shows the connection between the theorem of Carleson and Ehrenpreis and our results.
\begin{lemma}\label{lem: L_2}
	Suppose that $L$ is   non-quasianalytic and  slowly growing. Then 
	\[ \int_0^\infty \frac{r}{r^2+t^2}\Lambda_L(t)dt\asymp\Lambda_{L/\widetilde{L}}(r),\]
	where, as before, $\widetilde{L}(\rho):=\int_1^\infty \frac{du}{uL(u)}$.
\end{lemma}
In particular, the above lemma and theorem of Carleson and Ehrenpreis imply that
\begin{equation}
S_{\widetilde{L}} \text{Hol}(\mathbb{C})\subseteq S_L C_0(L;\mathbb{R}),\label{eq: C&E2}
\end{equation}
and that $\widetilde{L}$ in the left-hand side of \eqref{eq: C&E2} cannot be replaced by any function $L_2$ with $L_2(\rho)=o(\widetilde{L}(\rho)),\; \rho \to \infty$, while Theorem \ref{thm':NqaBeu} states that $S_{\widetilde{L}} C_0(\widetilde{L};\mathbb{R})=S_L C_0(L;\mathbb{R})$, but under additional regularity conditions. The proof of Lemma \ref{lem: L_2} is given in Appendix A.

\paragraph{Ehrenpreis representation.}
We will  use  the following representation of functions in the Beurling class,   due  to Ehrenpreis.
\begin{theorem*}
Let $L:[0,\infty)\to (0,\infty)$ be a function such that $\lim_{\rho\to\infty}L(\rho)=\infty$ and  the function $\rho\mapsto\rho\log (\rho L(\rho))$ is eventually strictly convex. If $g\in C_0(L;\mathbb(R))$, then there exists a representation
\[ g(t)=\iint_{\mathbb{C}} e^{iwt}\frac{d\mu(w)}{k(w)}, \]
where $\mu$ is a
finite complex-valued measure  and  $k\in C(\mathbb{C})$ is a non-negative function such that, for any $a,\;b>0$,
\[ \lim_{|w|\to \infty} \frac{k(w)}{\exp\left(a|\im w|+\Lambda_{L}(b|w|)\right)}=\infty.  \] 
Here, as before,
\[ \Lambda_{L}(r)=\sup_{\rho\geq 0}\left[\rho\log r-\rho\log\left(\rho L(\rho)\right)\right]. \]
\end{theorem*}
From here on, we will refer to such representation of functions in $C_0(L;\mathbb{R})$, as the Ehrenpreis representation.
 Such representations are not unique (the construction of $\mu$ and $k$ uses the Hahn–Banach theorem). The proof can be found in \cite[\S V.6]{Ehrenpreis9} or in \cite{Taylor}.

It is worth mentioning that it is possible to study singular transforms of Beurling classes via the Ehrenpreis representation: observing that if $g\in C_0(L;\mathbb{R})$ has the representation
\[ g(t)=\iint_{\mathbb{C}} e^{iwt}\frac{d\mu(w)}{k(w)}, \] 
then 
\[ \left(S_Lg\right)(z)=\iint_{\mathbb{C}} \left(S_L \exp\right)(iwz)\frac{d\mu(w)}{k(w)}:=\iint_{\mathbb{C}} E_1(iwz)\frac{d\mu(w)}{k(s)}, \]
where 
\[ E_1(z)=\sum_{n \geq 0 } \frac{z^n}{n!\gamma(n+1)}. \]
For instance, in this way, one could  prove Theorem~\ref{thm: Main thm,sing} in the case $I=\mathbb{R}$. The drawback of such an approach is its inability to treat  intervals $I$ which are different from the whole real line. On the other hand, it has the nice feature that its  easily extends to Beurling classes in several variables. We will not pursue this approach here. 
 	\paragraph{The functions $K_*$ and $E_*$}\label{par: K_* def}
We fix a non--quasianalytic $L$ satisfying the assumptions of Theorem \ref{thm':NqaBeu}. Recall the definitions of the associated functions: 
\[ \widetilde{L}(\rho)=L(\rho)\int_{\rho }^{\infty}\frac{du}{uL(u)}, \quad \rho>1,\]

\[ \gamma(\rho)=L(\rho)^\rho,\quad \widetilde{\gamma}(\rho)=\widetilde{L}(\rho)^\rho \]
and 
\[ E_*(z)=\sum_{n\geq 0} \frac{\widetilde{\gamma}(n+1)}{\gamma(n+1)} z^n,\quad K_* (z)=\frac{1}{2\pi i}\int_{c-i\infty}^{c+i\infty}\frac{\gamma(s)}{\widetilde{\gamma}(s)}z^{-s}ds,\quad c>0. \]
 Theorems \ref{TheoremK} and \ref{TheoremE} provide us with the asymptotics  of $K_*$ and $E_*$. Note that 
\[ \rho\frac{d}{d\rho} \log \frac{L(\rho)}{\widetilde{L}(\rho)}=\frac{1}{\widetilde{L}(\rho)}, \]
and so  the corresponding saddle point equation is 
\[ \log z= \log \frac{L(s)}{\widetilde{L}(s)}+\frac{1}{\widetilde{L}(s)}. \]
For $z\in \Omega(\pi/2)$, $|z|>r_0$, we denote by $s=s_z=\rho_z e^{i\theta_z}$ the unique solution to this saddle point equation. It follows from Theorem \ref{TheoremK}, 
\begin{equation}
\log K_*(z)\sim -\cos\theta \frac{\rho_z}{\widetilde{L}(\rho_z)} ,\quad |z|\to\infty,\quad z\in \Omega(\pi/2).\label{eq: K assymptotic }
\end{equation}
\begin{figure}[h]
	\centering
	\includegraphics[scale=0.75]{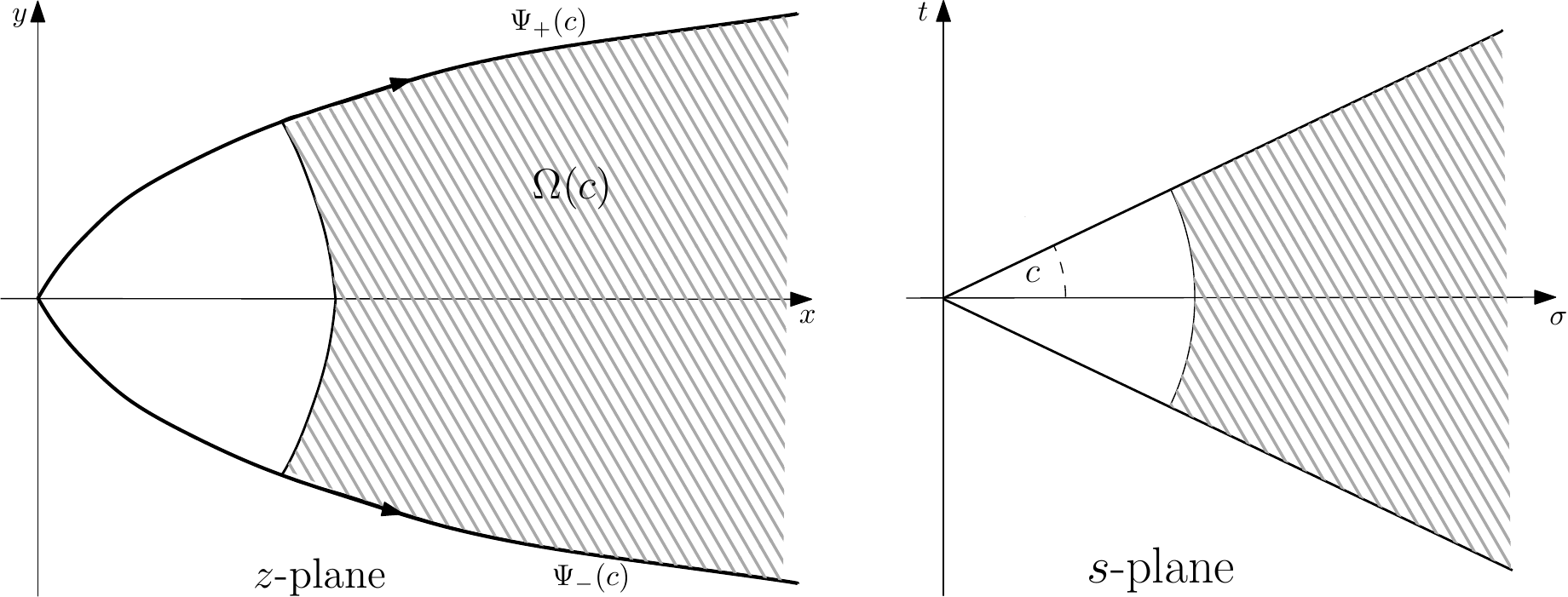}
	\caption{The curves  $\Psi_{\pm }(c)$}
	\label{fig: psi_{+}(c)}
\end{figure}
Given   $c\in[0,\tfrac{\pi}{2})$, we denoted by $\Psi_{\pm }(c)$ curves joining $0$ with $\infty$ in the 1st and 4th quadrants, respectively, and such that, for $|z|$ sufficiently large,  $\Psi_{\pm }(c)$ coincide with the curves $\left\{z=\frac{L(s)}{\widetilde{L}(s)}\exp\left(\frac{1}{\widetilde{L}(s)}\right) : \arg(s)=\pm c\right\}$. In particular, the above asymptotic formula together with Cauchy's theorem yields
\begin{equation}
\int_{\Psi_{\pm }(c)}z^n K_*(z)dz=\int_0^\infty r^n K_*(r)dr=\frac{\gamma(n+1)}{\widetilde{\gamma}(n+1)},\quad 0<c<\pi/2.\label{eq: moments on curve}
\end{equation}

\paragraph{Proof of the inclusion $S_L C_0(L;I)\supseteq S_{\widetilde{L}}C_0(\widetilde{L};\mathbb{R})$.}
\begin{proof}
	As we already mentioned, since $L$ is non-quasianalytic, the set $S_L C_0(L;I)$ does not depend on the interval $I$ (as long as $I$  contains the origin). From here on,  we will assume that  $I=(-1,1)$.
	
	Let $g\in C_0(\widetilde{L};\mathbb{R})$. Then,
	according to Ehrenpreis, there exists a representation
	\[ g(t)=\iint_{\mathbb{C}} e^{iwt}\frac{d\mu(w)}{k(w)}, \]
	where $\mu$ is a
	finite  measure  and  $k\in C(\mathbb{C})$ is a non-negative function such that, for every $a,\;b>0$,  
	\[ \lim_{|s|\to \infty} \frac{k(w)}{\exp\left(a|\im w|+\Lambda_{\widetilde{L}}(b|w|)\right)}=\infty.  \]
	Fix a sufficiently small constant $\delta>0$  that will be chosen later,  set (see Figure \ref{fig: AB_+}) \[A:= \{w: |\im w|>\tfrac{2\delta}{\pi} \Lambda_{\widetilde{L}}(|w|)\},\quad B^\pm :=\{w: |\im w|\leq \tfrac{2\delta}{\pi} \Lambda_{\widetilde{L}}(|w|),\pm \re w> 0\}\]
	and split 
	\[ g(t)=\iint_{\mathbb{C}} e^{iwt}\frac{d\mu(w)}{k(w)}=\iint_{A}+\iint_{B^+}+\iint_{B^-}=g_{e}(t)+g_+(t)+g_-(t). \]
	\begin{figure}[h]
		\centering
		\includegraphics[scale=0.75]{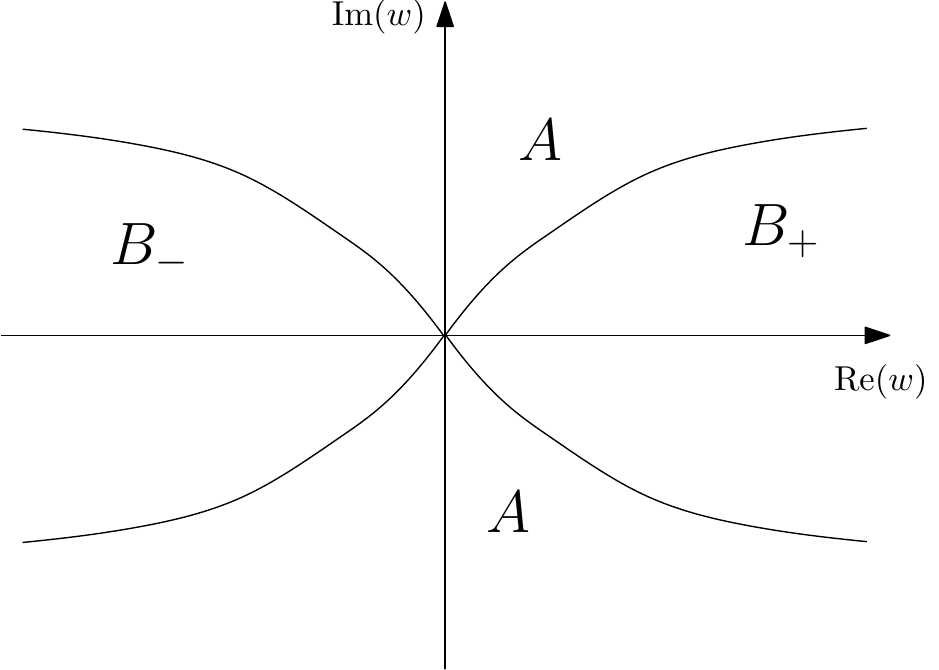}
		\caption{}
		\label{fig: AB_+}
	\end{figure}
We aim to find  functions $f_e,\;f_+,\;f_-\in C_0(L;I)$ such that 
\[ \widehat{f}_e(n)=\widehat{g}_e(n)\frac{\gamma(n+1)}{\widetilde{\gamma}(n+1)}\quad\text{and}\quad  \widehat{f}_\pm(n)=\widehat{g}_\pm(n)\frac{\gamma(n+1)}{\widetilde{\gamma}(n+1)},\quad n\geq 0. \]
	First we treat $g_{e}$. By the definition of the set $A$,
	\[ g_{e}(t)=\iint_{\mathbb{C}} e^{iwt}\frac{d\mu_e(w)}{k_e(w)}, \]
	where $\mu_e$ is a
	finite  measure  and a $k_e\in C(\mathbb{C})$ is a non-negative function such that 
	\[ \lim_{|w|\to \infty}k_e(w)e^{-a|w|}=\infty ,\quad \forall a>0.  \]
	Differentiation yields
	\[  |g_{e}^{(n)}(t)|\leq C_{M,a} \iint_{\mathbb{C}}|w|^n e^{-a|w|}|d\mu_e(w)|,\quad  a>0, |t|<M.\]
	Therefore, 
	\[ |g_{e}^{(n)}(t)|\leq \frac{C_{M,a}}{a^n}n!,\quad  a>0,\; |t|<M, \]
	which means that $g_{e}$ is an entire function. Now, by the Carleson--Ehrenpreis Theorem (we use its corollary stated as \eqref{eq: C&E2}), there exists $f_e\in C_0(L;\mathbb{R})$ such that 
	\[ \widehat{f_e}(n)=\widehat{g_e}(n)\frac{\gamma(n+1)}{\widetilde{\gamma}(n+1)},\quad n\geq 0. \]

	Now we  treat the function $g_+$. Clearly, the function $g_+$ is holomorphic in the upper half-plane and smooth up to its boundary and is represented therein by the same integral. Let us  estimate  the derivatives of  $g_+$ in the upper half-plane.  	Fix a large parameter $b>0$; then for $0<\psi<\pi/2$, we have 
	\[ 	\big|g_+^{(n)}(r e^{i\psi})\big|\leq\iint_{B^+} |w|^n \exp\big(-\im(re^{i\psi}w)\big)\frac{|d\mu(w)|}{k(w)}.\]
	By definition, $\exp\left(2b\Lambda_{\widetilde{L}}(|w|\right)\lesssim_b k(s)$, and for any   $w\in B_+$, we have $|\arg w|\leq \delta\frac{\Lambda_{\widetilde{L}}(|w|)}{|w|}$. Therefore,
	\begin{equation}
	\big|g_+^{(n)}(r e^{i\psi})\big|\lesssim_b \iint_\mathbb{C} |w|^n\exp\left(r|w|\sin\left(\frac{\delta\Lambda_{\widetilde{L}}(|w|)}{|w|}-\psi\right)-2b\Lambda_{\widetilde{L}}(|w|)\right)d|w|.\label{eq: g_+1st}
	\end{equation}
	By Lemma \ref{lem: subReplacmentLem}, part 4 (with $L/\widetilde{L}$ and $1/\widetilde{L}$ instead of $L$ and $\varepsilon$), we have 
\begin{equation}
\widetilde{L}\big(e\widetilde{L}(\rho)\rho\big)\sim\widetilde{L}(\rho)  ,\quad \rho\to\infty.\label{eq: LtildeasympLtilde}
\end{equation}
	Thus, by the first part of Lemma \ref{lem: widetildeE},  
	\[ \Lambda_{\widetilde{L}}(\rho)\sim \frac{\rho}{e\widetilde{L}(\rho)},\quad \rho\to\infty. \]
	Choosing $\delta$ in estimate  \eqref{eq: g_+1st} sufficiently small, we get
	\[\big|g_+^{(n)}(r e^{i\psi})\big|\lesssim_b    \frac{n!\widetilde{\gamma}(n+1)}{b^n} \sup_{\tau>0} \exp\left[\frac{r\tau}{2}\left(\frac{1}{\widetilde{L}(\tau)}-2\psi\right)\right].\]
	By Lemma \ref{lem: subReplacmentLem}, part 1, 
	\[ \widetilde{L}\left(\rho\frac{L(\rho)}{\widetilde{L}(\rho)}\right)\sim\widetilde{L}(\rho),\quad \rho\to\infty. \]
	Combining this   with \eqref{eq: LtildeasympLtilde}, we get 
	\[\widetilde{L} \left(\rho\right)\sim\widetilde{L} \left(\rho L(\rho)\right)\sim \widetilde{L} \left(\rho L^2(\rho)\right),\quad \rho\to\infty,\]
	which in turn yields
	\begin{equation}
	\big|g_+^{(n)}(r e^{i\psi})\big|\lesssim_b \frac{n!\widetilde{\gamma}(n+1)}{b^n} \exp\left[\frac{r\tau_\psi}{2\widetilde{L}(\tau_\psi)}\right],\quad \tau_\psi=\sup\left\{\tau:\;\widetilde{L} \left(\tau L^2(\tau)\right)\leq \psi^{-1} \right\}.\label{eq: g_+2nd}
	\end{equation}

	From here on we assume that $z=r e^{i\psi}\in \Psi_{\pi/3}$ with $r$ sufficiently large,  and that $s=\rho e^{i \pi/3}=s_z$ is related to $s$ by the saddle point equation
	\[ z=\frac{L(s)}{\widetilde{L}(s)}\exp\left(\frac{1}{\widetilde{L}(s)}\right) .\]
	By \eqref{eq:1stLasym} (applied to the function $L/\widetilde{L}$ instead of $L$), 
	\[ \rho\asymp \frac{L(\rho)}{\widetilde{L}(\rho)},\quad  \frac{\pi}{3}\cdot\frac{1}{\widetilde{L}(\rho)}\sim \psi,\quad r\to\infty. \]
In particular, if $r$ is sufficiently large, then 
\[\widetilde{L}\left(\tau_\psi L^2(\tau_\psi)\right)\leq \widetilde{L}(\rho)\quad \Longrightarrow\quad  \tau_\psi\leq \frac{\rho}{L(\rho)}. \]	
	Thus, for $0\leq t\leq 1$, the latter and inequality \eqref{eq: g_+2nd} yield 
	\[ \big|g_+^{(n)}(t z)\big|\lesssim_b \frac{n!\widetilde{\gamma}(n+1)}{b^n} \exp\left[\frac{r\rho}{ L(\rho)\widetilde{L}(\rho)}\right]\lesssim_b  \frac{n!\widetilde{\gamma}(n+1)}{v^n} \exp\left[\frac{1}{4}\cdot\frac{\rho}{  \widetilde{L}(\rho)}\right]. \]

	Set
	\[ f_+(t):=\int_{\Psi_{\pi/3}} g_+(zt) K_*(z)dz,\quad 0\leq t\leq 1.  \]
	By \eqref{eq: K assymptotic },
	\[\log K_*(z)\sim -\frac{1}{2}\cdot\frac{\rho}{\widetilde{L}(\rho)}. \] 
	Therefore, the function $f_+$ is well defined and satisfies
	\[\left|f_+^{(n)}(t)\right|\leq\left|\int_{\Psi_{\pi/3}} z^n g^{(n)}_+(zt) K_*(z)dz\right|\lesssim_b \frac{n!\widetilde{\gamma}(n+1)}{b^n} \left(1+\int_{r_0}^\infty r^n \exp\left[-\frac{1}{10}\cdot\frac{\rho}{\widetilde{L}(\rho)}\right]dr\right), \]
	for any $0\leq t\leq 1$. 	By Lemma \ref{lem: K and E last},  the integral in the right-hand side is $\leq C 2^n\frac{\gamma(n+1)}{\widetilde{\gamma}(n+1)}$. Therefore,
	\[\left|f_+^{(n)}(t)\right|\lesssim_b\left(\frac{2}{b}\right)^n n!\gamma(n+1),\quad 0\leq t\leq 1.  \]
	Since $b>0$ can be taken arbitrarily large, we conclude $f_+\in C_0(L;[0,1])$. 
	
	For $-1\leq t<0$, we define
	\[ f_+(t):=\int_{\Psi_{-\pi/3}} g_+(zt) K_*(z)dz. \]
	The same reasoning as above yields, $f_+\in C_0(L;[-1,0))$. Furthermore,
	\[f_+^{(n)}(0)=g^{(n)}_+(0) \int_{\Psi_{\pm\pi/3}}  z^n K_*(z)dz\,\stackrel{\eqref{eq: moments on curve}}{=}\,g^{(n)}_+(0) \int_{0}^\infty  z^n K_*(z) dz=g^{(n)}_+(0) \frac{\gamma(n+1)}{\widetilde{\gamma}(n+1)}.\]
	Therefore, $f_+\in C_0(L;[-1,1])$ and it is the sought for function. 
	
	For $g_-$ we define
	\[  f_-(t):=\int_{\Psi_{\mp \pi/3}} g_+(zt) K_*(z)dz ,\quad 0\leq \pm t\leq 1. \]
	The method of estimating the derivatives of $f_-$ is the same as the one of $f_+$. This finishes the proof of Theorem \ref{thm:NqaBeu}. 
\end{proof}
	\subsection{Carleman classes, Theorem \ref{thm:NqaCar}.}
	 \begin{thm}\label{thm':NqaCar}
		Let $L$ be a non-quasianalytic function and   $I$ be an open interval that contains the origin. Suppose that the function $\rho\mapsto L(\rho)/\widetilde{L}(\rho)$  satisfies assumptions \emph{(R1), (R2), (R3)} and \emph{(R8)},  and that $I$ is an open interval containing the origin. Then
		\[ S_L C(L;I)=S_{\widetilde{L}} C(\widetilde{L};0). \]
	\end{thm}
Throughout this section we fix a non-quasianalytic function  $L$ that satisfies  the assumptions of Theorem \ref{thm':NqaCar}.

The proof of the inclusion 
$$S_L C(L;I)\subseteq S_{\widetilde{L}} C(\widetilde{L};0)$$
follows the same lines as the analogous inclusion in the Beurling case. Therefore,  we will prove only the opposite direction. 

Recall that in order to prove that
\[S_L C(L;I)\supseteq S_{\widetilde{L}} C(\widetilde{L};\mathbb{R}),  \]
we have used Ehrenpreis representation for functions in the Beurling class $C_0(\widetilde{L};\mathbb{R})$. We are not aware of an  analogous representation for functions in the corresponding Carleman class.   So, in order to show the inclusion 
$$S_L C(L;I)\supseteq S_{\widetilde{L}} C(\widetilde{L};0)$$
we will take a different approach, based on asymptotically holomorphic extensions of smooth functions. We begin with some preliminaries. 

\subsubsection{A decomposition of functions in $C(\widetilde{L};0)$.}
Denote by $C^\omega(0)$ the sets of all function analytic in some neighborhood of the origin.  
\begin{lemma}\label{lem: borichev}
	Let $g\in C(\widetilde{L};0)$. Then there exists functions $g_1,g_2$ such that 
	\begin{enumerate}
		\item[\emph{1.}] $g\equiv g_1+g_2$ is some neighborhood of the origin;
		\item[\emph{2.}] $g_1\in C^\omega(0)$;
		\item[\emph{3.}]  There exists $C>0$ such that 
		\[ \sup_{x\in\mathbb{R}} |g_2^{(n)}(x)|\leq C^{n+1} n!\widetilde{\gamma}(n+1) ,\quad n\geq 0.\]
	\end{enumerate}
\end{lemma}
The proof is based on the theory of almost holomorphic extensions \cite{Dynkin}. The idea to use almost holomorphic extensions was suggested  by Alexander Borichev.
\begin{proof}
	Fix $g\in C(\widetilde{L};0)$. Let $[-a,a]$ ($a>0$) be an interval such that  $g\in C(\widetilde{L};[-a,a])$. According to Dynkin \cite{Dynkin}, $g\in C(\widetilde{L};[-a,a])$ if and only if it can be represented as 
	
	\[g(t)=\iint_{\mathbb{C}}G(z)\frac{dx dy}{z-t},\]
	where $G:\mathbb{C}\to \mathbb{C}$ is  a continuous and compactly supported function such that 
	\[ |G(z)|\leq C_1h\left(C_2\; \text{dist}(z,[-a,a])\right),\quad \text{where}\quad h(r)=\inf_{n\geq 0} \widetilde{\gamma}(n+1) r^n. \]

	Let $\delta<a/8$ be a small parameter and   $\xi$ be a continues function such that $\xi \equiv 0$ for all $z\in \{z: \text{dist}(z,[-a/4,a/4])\leq \delta\}$ and  $\xi\equiv 1 $  for all $z\in\{z:\text{dist}(z,[-a/4,a/4]) \geq 2 \delta\}$.
	Finally,  set 
	\[g(t)=\iint_{\mathbb{C}}G(z) \xi(z)\frac{dx dy} {z-t}+\iint_{\mathbb{C}}G(z)(1-\xi(z) ) \frac{dx dy}{z-t}:=g_1(t)+g_2(t).\]
	Clearly $g_1\in C^{\omega}([-a/4,a/4])$, and by  Dynkin's theorem $g_2\in C(\widetilde{L};[-3a/4 ,-3a/4 ])$.  The function $g_2$ also satisfies 
	\[|g_2^{(n)}(t)|\leq n! \iint |G(z)\xi (z)| \frac{dx dy}{|z-t|^{n+1}}\leq \frac{C n!}{\text{dist}^n (t,[-\tfrac{a}{2},\tfrac{a}{2}])}\leq C^{n+1}n! ,\quad t\notin [-a/2,a/2],\;n\in \mathbb{Z_+} .\]
	Since $\widetilde{L}\uparrow\infty$, the last estimate completes the proof of Lemma \ref{lem: borichev}.
\end{proof}
\subsubsection{}
We will also need a statement similar to  \eqref{eq: C&E2}:
\begin{equation}
S_{\widetilde{L}} C^\omega(0)\subseteq S_L C(L;\mathbb{R}).\label{eq: C&E3}
\end{equation}
It   follows from the theorem of Carleson and Ehrenpreis  and Lemma \ref{lem: L_2}. The proof is  similar to the one given in the case of Beurling classes, so  we will omit it. 
\begin{proof}[Proof of the inclusion $S_L C(L;I)\supseteq S_{\widetilde{L}} C(\widetilde{L};0)$.]
Fix an element in $g\in C(\widetilde{L};0)$, and fix functions $g_1,g_2$ as in Lemma \ref{lem: borichev}.
By  \eqref{eq: C&E3}, there exists a function $f_1\in C(L;\mathbb{R})$ such that 
\[ S_{\widetilde{L}} g_1= S_L f_1.  \]
Put
\[ f_2(t)=\int_0^\infty g_2(tx)K_*(x)dx, \]
with the analytic function $K_*$ defined in the beginning of \ref{par: K_* def}.
Clearly
\[  |f_2^{(n)}(t)|=\int_0^\infty x^n|g_2^{(n)}(xt)||K_*(x)|dx\leq C^{n+1} n!\widetilde{\gamma}(n+1)\int_0^\infty x^n|K_*(x)|dx. \]
Since $K_*$ is eventually positive, this estimate yields
\[ |f_2^{(n)}(t)|\leq C^{n+1} n!\widetilde{\gamma}(n+1)\frac{\gamma(n+1)}{\widetilde{\gamma}(n+1)} =C^{n+1} n!\gamma(n+1).\]
Therefore, $f_2\in C(L;\mathbb{R})$ and $S_{\widetilde{L}}g_2=S_L f_2$. Finally, we define $f=f_1+f_2\in C(L;\mathbb{R})$ and conclude that $S_{\widetilde{L}}g=S_L f$.
\end{proof}
\subsection{Theorem \ref{thm:spliting}.}
\begin{thm}\label{thm':spliting}
	Suppose that the function  $L$  satisfies assumptions \emph{(R1), (R2), (R3), (R5), (R7)} and \emph{(R9)}, and that  $I$ is an open interval containing the origin. Then the singular transform $S_L$ maps $C_0^\pm(L;I)$ bijectively onto the space $A^\pm(L;I)$, with the inverse $R_L^\pm$. 
\end{thm}
\subsubsection{The inclusion $A^\pm(L;I)\subseteq S_L C_0^\pm (L;I)$.}
\begin{proof}
Fix an open  interval $I$ that contains the origin,  and a function $L$  that   satisfies assumptions (R1), (R2), (R3), (R5), (R7) and (R9).
	
	Let $F\in A^+(L;I)$ and put $f=R_L^+ F$. We want to show that $f\in C_0^+(L;I)$. Put $I_+=I\cap [0,\infty)$ and $I_-=I\cap (-\infty,0]$. 
	The proof is broken into three parts: first we show that $f\in C_0(L;I)$, then  that $f\in C^+_0(L;\text{int}(I_\pm) )$, and 	
	finally that $f\in C^+_0(L;I )$.  

\textit{Part 1.} Fix $a_+\in I_+$, and let $c_+\in I_+$ such that $a_+<c_+$.  By the definition of the class $A^+(L;I)$, for any $B>0$ there is $\Delta>0$, such that 
\[|F(r e^{i\psi})|\lesssim_B H\left(\psi+\frac{2B}{r}\right)+ E\left(\frac{r}{c_+}+\Delta r\sin\psi)\right)  \]
for $-\tfrac{B}{r}\leq\psi\leq\tfrac{\pi}{2}+\tfrac{B}{r}$.
By  Cauchy's estimates, we have 
\[ |F^{(n)}(z)|\lesssim \frac {n!}{B^n} \max_{|z-w|=B} |F(w)|. \]
Since $H$ is decreasing and continuous in $(0,\tfrac{\pi}{2})$, we have
\begin{equation}
|F^{(n)}(r e^{i\psi})|\lesssim_B \frac{n!}{B^n} \left[H\left(\psi+\frac{B}{2r}\right)+ E\left(\frac{r}{c_+}+\frac{B\Delta}{c_+}+\Delta r\sin\psi\right) \right] \label{eq: 1st derivative estiamte with H}
\end{equation}
for $0\leq\psi\leq\tfrac{\pi}{r}$.
By the definition of $R_L^+$,
$$f(u)= (R_L^+ F)(u) =\int_{\Psi_+} F(uz)K(z)dz ,\quad 0\leq u<a_+.$$
Thus,
\[|f^{(n)}(u)|\leq  \int_{\Psi_+} |z|^n|F^{(n)}(uz)K(z)||dz| ,\quad 0\leq u<a_+. \]
By \eqref{eq: 1st derivative estiamte with H}, we have  
\[  |f^{(n)}(u)|\lesssim_B \frac{n!}{B^n}\left[1+I_1(u,n)+I_2(u,n) \right]  \]
where 
\[ I_1(u,n)=\int_1^\infty r^n H\left(\psi+\frac{B}{2r\max\{u,1\}}\right)|K(re^{i\psi})|dr, \quad z=re^{i\psi}\in \Psi_+,\]
and
\[\quad I_2(u,n)=\int_1^\infty r^n E\left(\frac{ru}{c_+}+\frac{B\Delta}{c_+}+\Delta ru\sin\psi\right)|K(re^{i\psi})|dr,\quad re^{i\psi}\in \Psi_+.  \]

As a result, in order to show that $f\in C_0(L;[0,a_+])$, it suffices to show that there is a $C>0$, such that   for sufficiently large $B>0$,  
\[ I_1(u,n)\lesssim_B C^{n+1}\gamma(n+1), \quad I_2(u,n)\lesssim_B C^{n+1}\gamma(n+1),\quad 0\leq u\leq a_+. \]

We begin with the integral $I_1$. By Lemma \ref{lem: loglog H}, there exists $A>0$ such that 
\[ H\left(\psi+\frac{A}{r}\right)\leq H(\psi)^{1-\tfrac{3}{r}}. \]
By taking $B$ so large that $B>2A\max\{1,a_+\}$, we get
\[H\left(\psi+\frac{B}{2r\max\{u,1\}}\right)\lesssim_B  H(\psi)^{1-\tfrac{3}{r}}\lesssim|E(z)|^{1-\tfrac{2}{r}},\quad z=r e^{i\psi}\in\Psi_+,\; 0\leq u\leq a_+.  \] 
By Theorems \ref{TheoremK} and \ref{TheoremE}, 
\[ |E(z) K(z)|\sim \frac{\rho_z}{|z|\varepsilon(\rho_z)},\quad z\to\infty,\; z\in \Psi_+.  \]
In particular,
\[ |E(z) K(z)|\lesssim\exp\left(|z|^{-1}\cdot\re\left(i\rho_z\varepsilon(i\rho_z)\right)\right)\lesssim |E(z)|^{\tfrac{1}{|z|}} \,\quad |z|>1, z\in \Psi_+. \]
Thus,
$$I_1(u,n)\leq \int_{\Psi_+\cap\{|z|>1\}} |z|^n|E(z)|^{-\tfrac{1}{|z|}}d|z|,$$
and, by Lemma \ref{lem:spliting H intgral K small},
\[ I_1(u,n)\lesssim_B C^{n+1}\gamma(n+1),\quad u\in[0,a_+]. \]

Now we turn to $I_2$. Since, $c_+>a_+$, and $\psi\to 0$ as $r\to\infty$, $re^{i\psi}\in \Psi_+$, there exists a $\delta=\delta(a_+,c_+)>0$, such that 
\[I_2(u,n) \lesssim_B \int_1^\infty r^n E((1-\delta)r)\left|K(re^{i\psi})\right|dr,\quad u\in [0,a_+].  \] 
By Lemma \ref{lem:KE strong A for spliting.}, there exists $\delta_1>0$ such that 
\[ E((1-\delta)|z|)|K(z)|\lesssim E(\delta_1 |z|),\quad z\in\Psi_\pm,\;|z|>1, \]
whence,
\[ I_2(u,n) \lesssim_B \int_1^\infty \frac{r^n}{E(\delta_1 r)}dr. \]
By definition,  $E(\delta_1 r)\geq \frac{\gamma(n+3)}{\delta_1^{n+1}r^{n+2}}$. Thus,
\[ I_2(u,n) \lesssim_B \delta_1^{-n}\gamma(n+3)\asymp \delta_1^{-n}\gamma(n+1).\]

We have shown that 
\[   |f^{(n)}(u)|\lesssim_B \frac{n!C^n}{B^n}\gamma(n+1),\quad 0\leq u\leq a_+.\]
The constant $B$ can be taken  arbitrarily large, so we conclude that $f\in C_0(L;[0,a_+])$. Since, $0<a_+\in I_+$ was arbitrary, we conclude that $f\in C_0(L;I_+)$. The proof that $f \in C_0(L;I_-)$ is similar.

\textit{Part 2.} Now, we will show that  $f\in C^+_0(L;\text{int}(I_+) )$. Let $u\in \text{int} (I_+)$. We need to show that for sufficiently small $0<v$, $f$ is analytic at $w=u+iv$. 

Let  $c_+\in I_+$ be such that 
\[ \frac{u}{c_+}\leq 1-5\delta, \]
for some $\delta>0$. 
By the definition of the class $A^+(L;I)$, with $B=1$, there is a $\Delta>0$, such that 
\[|F(r e^{i\psi})|\lesssim H\left(\psi\right)+ E\left(\frac{r}{c_+}+\Delta r\sin\psi\right),\quad 0<\psi\leq \tfrac{\pi}{2}.\]
Fix the above $\Delta$ and choose $v>0$ so small  that for $w=u+iv$, we have $\Delta \arg(w)<\delta$ and $\frac{v}{c_+}<\delta.$    Denote by $D_q(w)$ the closed disk $\{w':\; |w-w'|\leq q\}$, and  choose $q$ so small such that  that $v-q>0$ and 
\[\frac{|w'| r}{c_+}+\Delta\sin(\arg w'+\psi) r|w'|\leq (1-\delta)r, \quad   w'\in D_q(w),\, r>r_\delta,\, re^{i\psi}\in \Psi_+. \]
By choosing such a $q$, we have, 

 $$\max_{w'\in D_q(w)}|F(zw')|\lesssim_q  E((1-\delta)|z|) ,\quad z\in\Psi_+.$$
As a result,  Lemma \ref{lem:KE strong A for spliting.} yields
  \[ \max_{w'\in D_q(w)}\int_{\Psi_+} |F(zw')K(z)dz|<\infty. \]
  Thus, the function $f=R_L^+ F$ is analytic in an upper neighborhood of $ \text{int}(I_+)$, and therefore   $f\in C_0^+(L;\text{int}(I_+))$, as claimed.
  The proof that $f\in C_0^+(L;\text{int}(I_-))$ is identical.
  
\textit{Part 3.}  So far we have shown that $f\in C_0(L;I)$, and that $f\in C_0^+(L;\text{int}(I_\pm) )$. We need to show that $f\in C_0^+(L;I)$. This will follow from Morera's theorem, if we can  show that the boundary values  of $f$ on the interval $i[0,\delta)$ as an element of $C_0^+(L;\text{int}(I_+) )$ coincide with the boundary values of $f$ on the same interval as an element of $C_0^+(L;\text{int}(I_-) )$. That is,  we need to show that for  $0\leq v<\delta$, 
\[ \int_{\Psi_+} F(iv z)K(z)dz=\int_{\Psi_-} F(iv z)K(z)dz. \]
To do so, we fix $a_\pm \in \text{int}(I_\pm)$, and   let $0<a<\min\{a_+,-a_-\}$. Denote by $\Omega$ the domain bounded between $\Psi_-$ and $\Psi_+$ (which contains the positive ray). By the definition of the space $A^+(L;I)$, there exists a positive $\delta$, such that 
\[ |F(iu z)|\lesssim E(|z|/2),\quad z\in \Omega, \,0\leq u< \delta. \] 
We fix this value of $\delta$. 
For $z\in \Omega$, we denote by $z^*$ its radial projection on $\Psi_+$, i.e.,  
$|z|=|z^*|, z^*\in \Psi_+$.  By Theorem \ref{TheoremK}, 
\[ |K(z)|\leq |K(z^*)|,\quad |z|>r_0,\; z\in\Omega. \]
Thus,  if $\Psi$ is an arbitrary curve in $\Omega$ joining $0$ and $\infty$, then  by Lemma \ref{lem:KE strong A for spliting.},
\[ \int_\Psi F(ivz)K(z)dz= \int_{\Psi_+} F(ivz)K(z)dz,\quad 0\leq v<\delta. \]
In particular, this is true with $\Psi=\Psi_-$, which is the desired result.

 We have shown that 	 $A^+(L;I)\subseteq S_L C_0^+ (L;I)$, the proof of the second inclusion,   $A^-(L;I)\subseteq S_L C_0^- (L;I)$,  is the same.
\end{proof}
\subsubsection{The inclusion $A^\pm(L;I)\supseteq S_L C_0^\pm (L;I)$}
Here we prove the inclusion $A^\pm(L;I)\supseteq S_L C_0^\pm (L;I)$. We begin with some preliminaries regarding the classes $C_0^\pm (L;I)$.
\paragraph{A decomposition of elements in $C_0(L;I)$.}\label{subsubsec: dec}
\begin{lemma}\label{lem: dec}
	Suppose that $L:[0,\infty)\to[1,\infty)$ is an eventually increasing and  unbounded  function such that the function  $\rho\mapsto \rho\log L(\rho)$ is eventually convex,  and that $I$ is an open interval. Then for any $f\in C_0(L;I)$ and any closed subinterval   $J\subset I$,
	there exists a function $f_\omega \in C^\omega(J)$ and $p\in C(\mathbb{R})$ satisfying  
	\begin{equation}
|p(t)|\lesssim_\delta \frac{\delta^n n! \gamma(n+1)}{|t|^n},\quad \delta>0,\;t\in\mathbb{R} \label{eq: g def}
	\end{equation}
	such that,
	\begin{equation}
	f(x)=f_\omega(x)+\int_{\mathbb{R}} e^{ixt}p(t)dt,\quad x\in J.\label{eq: dec}.
	\end{equation}
\end{lemma}
For functions $f\in C_0^+(L;I)$, the above  lemma admits. 
\begin{corollary}\label{cor: dec+}
Suppose that $L:[0,\infty)\to[1,\infty)$ is an eventually increasing and  unbounded  function such that the function  $\rho\mapsto \rho\log L(\rho)$ is eventually convex,  and that $I$ is an open interval. Then for any $f\in C_0^+(L;I)$ and any closed subinterval   $J\subset I$,
	there exists a function $f_\omega \in C^\omega(J)$ and $p\in C(\mathbb{R})$ satisfying  \eqref{eq: g def}, such that 
	\[ f(x)=f_\omega(x)+\int_{1}^\infty e^{ixt}p(t)dt,\quad x\in J. \] 
\end{corollary}
By  the above, any $f\in C_0(L;I)$ can be written as $f=f_\omega+f_++f_-$, where $f_\omega\in C^\omega(J)$ and $f_\pm\in C_0^\pm(L;\mathbb{R})$. This implies the 
 decomposition $C_0(L;I)=C_0^+(L;I)+C_0^-(L;I)$. The proof of the lemma, is based on the theory of almost holomorphic extensions, and resembles the proof of Lemma \ref{lem: borichev}.
\begin{proof}[Proof of Lemma \ref{lem: dec} ]
	Fix $L$ and the closed interval $J\subset I$.  Let $J'$ be a closed interval such that  $J\subset \text{int}(J')  \subset J'\subset I$.
	According to Dynkin \cite{Dynkin}, $f\in C_0(L;J')$ if and only if $f$ can be represented by 
	\[f(x)=\iint_{\mathbb{C}}F(w)\frac{du dv}{w-x},\quad w=u+iv,\]
	where $F$ is a continuous and compactly supported function such that for every $A>0$, there exists a $C>0$ such that 
	\[ |F(z)|\leq  Ch\left(A\; \text{dist}(z,J')\right),\quad	\text{where}\quad h(r)=\inf_{n\geq 0} \gamma(n) r^n.\]
	For $\eta>0$, let $\xi_\eta:\mathbb{C}\to[0,1]$ be a continuous function such that $\xi_\eta \equiv 0$ on  $ \{w: \text{dist}(w,J)\leq \eta\}$ and  $\xi_\eta \equiv 1 $ on $\{w:\text{dist}(w,J) \geq 2 \eta\}$. Finally, set 
	\[f(t)=\iint_{\mathbb{C}}F(w) \xi_\eta  (w)\frac{du dv} {z-t}+\iint_{\mathbb{C}}F(w)(1-\xi_\eta(w)  ) \frac{du dv}{w-x}:=f_\omega(x)+f_1(x).\]
	By  Dynkin's Theorem,  $f_\omega,f_1 \in C_0(L;J')$ as well. Moreover, $f_\omega\in \text{Hol}\{w: \text{dist}(w,J)\leq \eta\}$ and  $f_1\in \text{Hol}\{w: \text{dist}(w,J)\geq 2\eta\}$. Choosing $\eta $ sufficiently small, we estimate  the derivatives of $f_1$ as follows:
	\[|f_1^{(n)}(x)|\leq n! \iint | F(w)| \frac{du dv}{|w-x|^{n+1}}\leq C^{n+1}\frac{ n!}{(|x|+1)^{n+1}} ,\quad n\in \mathbb{Z_+},\; \]
	for any $x$ such that $x+1,x-1\notin J$. It follows that, for any $\delta>0$, 
	\[ |f_1^{(n)}(x)|\leq C_\delta \delta^n \frac{n!\gamma(n+1)}{(1+|x|)^{n+1}}. \]
	Put
	\[ p(t)=\frac{1}{2\pi} \int_{\mathbb{R}}e^{-ixt}f_1(x)dx. \]
	We will finish the proof by showing that $p$ satisfies \eqref{eq: g def}. Integration by parts yields
	\[   p(t)=\frac{i^n}{2\pi} \int_{\mathbb{R}}\frac{e^{-ixt}}{t^n}f^{(n)}_1(x)dx. \]
	Therefore, for sufficiently large $|t|$ we have
	\[ |p(t)|\leq C_\delta \inf_{n>0} \delta^n\frac{n!\gamma(n+1)}{|t|^n}, \]
	proving the Lemma.
\end{proof}
\begin{proof}[Proof of Corollary \ref{cor: dec+}. ] 
	Fix $f_+\in C_0^+(L;I)$  and a closed interval $J\subset I$. Since $f_+\in C_0(L;I)$,  applying  Lemma \ref{lem: dec} we find $f_\omega\in C^\omega(J)$ and $p$ satisfying \eqref{eq: g def} such that 
	\[ f_+(x)=f_\omega(x)+\int_{\mathbb{R}} e^{ixt}p(t)dt,\quad x\in J. \] 
	The functions $x\mapsto \int_{-\infty}^1 e^{ixt}p(t)dt$ and $x\mapsto \int_{1}^\infty e^{ixt}p(t)dt$ are analytic in the lower and upper half-planes, respectively. Put
	\[ \widetilde{f}_+:= \int_{1}^\infty e^{ixt}p(t)dt,\quad \widetilde{f}_-:=\int_{-\infty}^1 e^{ixt}p(t)dt+f_\omega(x). \] 
	Then $\widetilde{f}_\pm\in C_0^\pm(L;I)$ and $f_+ +0=\widetilde{f}_+ +\widetilde{f}_-$.  We have  found two decompositions of $f_+\in C_0(L;I)$ as a sum of elements in $ C_0^\pm(L;I)$ (one of which is the trivial one,  $f_+=f_+ +0$). Therefore, $f_+$ and $\widetilde{f}_+$ differ by a $C^\omega(J)$ function. This completes the proof.
\end{proof}

To show that $S_L C_0^\pm (L;I)\subset A^\pm(L;I)$ we will proceed by the following plan. Given $f\in C_0^+(L;I)$ and a closed subinterval $J\subset I$, first, we   use the decomposition  of Corollary \ref{cor: dec+}:
	\[ f(x)=f_\omega(x)+\int_{1}^\infty e^{ixt}p(t)dt,\quad x\in J. \]
Then by the linearity of the singular transform, 
	\[ S_L f= S_L f_\omega +S_L \left(\int_{1}^\infty e^{ixt}p(t)dt\right)=S_L f_\omega+\int_1^\infty E_1(xt)p(t)dt,  \]
	where
	\[ E_1(z)=S_L(\exp)(z)=\sum_{n\geq 0} \frac{z^n}{n!\gamma(n+1).} \]
	The treatment of the summands  in the RHS is different.  To obtain estimates for $S_L f_\omega$ we use     Theorem 7$^\prime$, while upper bounds for the second summand will follow from Lemma \ref{lem: mainSpilitEstimate+}.
\paragraph{Proof of the inclusion $A^\pm(L;I)\supseteq S_L C_0^\pm (L;I)$.}
\begin{proof}
Let $f\in C_0(L;I)$, and let $J$ be a compact sub-interval of $I$. By Corollary \ref{cor: dec+}, 
	there exist functions $f_\omega \in C^\omega(J)$ and $p\in C(\mathbb{R})$ satisfying 
	\[ |p(t)|\lesssim_\delta \frac{\delta^n n! \gamma(n+1)}{|t|^n},\quad \delta>0,\;t\in\mathbb{R}, \] 
	such that
\[ f(x)=f_\omega(x)+\int_{1}^\infty e^{ixt}p(t)dt,\quad x\in J. \] 
By the linearity of the singular transform,
\[ S_L f =S_L f_\omega+  S_L\left(\int_{1}^\infty e^{ixt}p(t)dt\right):=F_\omega+F_p.  \]
We claim that 
\[ F_p(z)=\int_1^\infty E_1(itz)p(t)dt ,\quad\text{where}\quad  E_1=S_L (\exp).\]
Indeed, by Lemma \ref{lemma:estimatcoef2}, for any $\delta>0$, 
\[ |p(t)|\lesssim_\delta e^{-\delta^{-1}\Lambda_L( t)}. \]
Therefore, Lemma \ref{lem: widetildeE trvial bound} yields that the function $z\mapsto \int_1^\infty E_1(itz)p(t)dt$ is an entire function. This function, has the same Taylor coefficients at the origin as  $F_p$, and so these two functions coincide, as claimed. 

Fix $B>0$  and $\delta>0$, by Lemma \ref{lem: mainSpilitEstimate+}, 
  \[ |E_1(i t r e^{i\psi})|\leq e^{C_B(\Lambda_L(t)+1)}H\left(\psi+\frac{2B}{r}\right) ,\quad r>2B,\;\,2B/r<\psi\leq\tfrac{\pi}{2},\;\, t>1. \]
 Thus,
 \[ | F_p(re^{i\psi})| \lesssim H\left(\psi+\frac{2B}{r}\right) ,\quad 2B/r<\psi\leq\tfrac{\pi}{2}. \]
In order to show that $S_L f\in A^+(L;I)$, we first  extend  the estimate of $F_p$  to the strip $-\tfrac{B}{r}<\psi<\tfrac{2B}{r}$.

The function $F_p$ is the singular transform of a function in $C_0(L;\mathbb{R})$. Thus, by Theorem \ref{thm': Main thm,sing},  
\[ \max_{|\psi|<\tfrac{3B}{r}}|F_p(r e^{i\psi})|\leq \widetilde{E}\left(\frac{r}{4B}\right)E\left(\delta r\right), \]
	where $\widetilde{E}(z)=\sum_{n\geq 0} \frac{z^n}{\widetilde{\gamma}(n+1)}$.
	By Lemma \ref{lem: K and E ray lemma} (applied both to $E$ and to $\widetilde{E}$), we have,
	\[ \widetilde{E}\left(\frac{r}{4B}\right)E\left(\delta r\right)\lesssim\widetilde{E}^2\left(\frac{r}{4B}\right)+E^2\left(\delta r\right)\lesssim \widetilde{E}\left(\frac{r}{3B}\right)+E(2\delta r)\lesssim H\left(\frac{2B}{r}\right)+E(2\delta r). \]
Combining this with the previous bounds of  $F_p$, we conclude that
\[  | F_p(re^{i\psi})| \lesssim  H\left(\psi+\frac{2B}{r}\right)+E(2\delta r),\quad -\tfrac{B}{r}<\psi\leq\tfrac{\pi}{2}, \]
Due to the symmetry, we obtain
\[  | F_p(re^{i\psi})| \lesssim  H\left(\psi-\frac{2B}{r}\right)+E(2\delta r),\quad \tfrac{\pi}{2}<\psi<\pi+\tfrac{B}{r}. \]

It remains to estimate the function $F_\omega$.
By Theorem \ref{thm: Real analytic}, $F_\omega\in  A^\omega(L;J)$. Thus, for any $c_-<0<c_+$ with  $c_\pm\in J$, there exists a $\Delta>0$, such that 
\[ \left|F_\omega(re^{i\psi})\right|\lesssim E\left(\frac{r}{|c_\pm|}+\Delta r\sin\psi\right).\]
Thus, if   $\delta$ is  so small  that  $2\delta<\max\{c_+,|c_-|\}$,  then,
\[ |S_L f (re^{i\psi}) |\lesssim H\left(\psi\pm\frac{2B}{r}\right)+E\left(\frac{r}{|c_\pm|}+\Delta r\sin\psi\right)\]
whenever $0\leq \pm (\tfrac{\pi}{2}-\psi) \leq \tfrac{\pi}{2}+\tfrac{B}{r}$. Since $J$ is an arbitrary compact subset of $I$, we conclude that $S_L f\in A^+(L;I)$. This finishes the proof of Theorem \ref{thm':spliting}. 
\end{proof}
\section*{Acknowledgment}
First of all  I would like to express my sincere gratitude
to my PhD advisor Mikhail
Sodin for his attentive guidance and many discussions and suggestions. I would also like to thank him for his help in writing this paper and for reading and commenting on many drafts.
I am grateful to   Alexander Borichev for suggesting to use almost holomorphic extensions, to Fedor 
Nazarov for suggesting to study   the spitting in  cases where the punctual image can not be described directly and  to Alon Nishry for  discussions and helpful comments. I thank Andrei Iacob for his help with copy-editing
 this paper.
\begin{appendices} 
	\section{Regular functions}
	\subsection{Proof of Lemma \ref{lem: widetildeE}.}
		\textit{Part 1.} Consider the function 
		\[ f(x,r)= x\log (erL(r)) -x\log\left(x L(x)\right). \]
Differentiation with respect to $x$ yields that $f_x(x,r)=0$ if and only if 
\[ erL(r)=exL(x)\exp(\varepsilon(x)) ,\quad \text{where}\quad \varepsilon(x)=\frac{xL'(x)}{L(x)}.\]
By assumption (R1), $\varepsilon(x)=o(1)$ as $x\to\infty$. Thus 
\[ \sup_{x>0}f(x,r)\sim r,\quad r\to\infty.  \]	
Since
\[ \sup_{x>0}f(x,r)=\Lambda_L(erL(r)), \]
we conclude part 1.

		\vspace{2mm}	\textit{Part 2.} It follows from the definition of $\Lambda_L$, that the function  $r\mapsto r\Lambda_L'(r)$ is the inverse function to $\rho\mapsto \rho L(\rho)\exp(1+\tfrac{\rho L'(\rho)}{L(\rho)})$. Denote the later function by $v(\rho)$. We therefore have, 
		\[ \Lambda_L(r)-\Lambda_L(r_0):=\int_1^r t\Lambda_L'(t) \frac{dt}{t}\stackrel{(t=v(u))}{=}\int_{v^{-1}(r_0)}^{v^{-1}(r)}\frac{uv'(u)}{v(u)}du=\int_{v^{-1}(r_0)}^{v^{-1}(r)}1+\varepsilon(u)+u\varepsilon'(u) du,\quad u\to\infty. \]
		By assumption (R3), there exists $r_0>0$ such that  $\varepsilon(u)+u\varepsilon'(u)>0$, for $u\geq r_0$. With such choice of $r_0$, we obtain 
			\[ \Lambda_L(r)-\Lambda_L(1)\geq \int_{v^{-1}(1)}^{v^{-1}(r)} du=r\Lambda_L'(r)-r_0\Lambda_L'(r_0).\]
			Thus,  $\frac{\Lambda_L'(r)}{\Lambda_L(r)}\leq \frac{1}{r}$ for sufficiently large $r$, and hence the part 2.

			\vspace{2mm}\textit{Part 3.} By part 1, 
		\[ \Lambda_L\left(r L(r)e)\right)\sim r\to \infty.\]
Since the function $L$ is slowly varying, we have $$e(ar)L(ar)\sim a(erL(r)), \quad r\to\infty$$ for any $a>0$. Therefore,
$\Lambda_L(ar)\sim a\Lambda_L(r)$ as $r\to\infty$. This completes the proof of  part 3.

	\vspace{2mm}		\textit{Part 4.} We fix a large $k>0$, and put 
		$g(n) =n \left(\log L(k)-\log L (n)\right)$.
		Since, $L$ is regularly varying, we have
		$g(\tfrac{k}{2})=O(k)$ and $g'(\tfrac{k}{2})=O(1)$ as $k\to\infty.$
		Moreover,  $g$ is eventually concave, and therefore it is bounded by its tangent line at the point $n=\tfrac{k}{2}$.\qed
		\subsection{Lemma \ref{lem: subReplacmentLem}.} The proof of Lemma \ref{lem: subReplacmentLem} is based on the following lemma from \cite[\S3.12.2]{Bingham}:
		\begin{lemma}\label{lem: Bingham}
			Let $L_j:[0,\infty)\to[0,\infty)$, $j=1,2$, be two eventually non-decreasing and $C^1$ smooth functions such that the functions  $\varepsilon_j(\rho)=\tfrac{\rho L_j'(\rho)}{L_j(\rho)}$ are bounded in $[0,\infty)$. Then $L_1(\rho L_2(\rho))\sim L_(\rho)$, as $\rho\to\infty$, if and only if
			\[ \varepsilon_1(\rho)\log L_2(\rho)=o(1),\quad \rho\to\infty. \]   
		\end{lemma}
	\begin{proof}[Proof of Lemma \ref{lem: subReplacmentLem}.]
\textit{Part 1.} The assertion $\varepsilon(\rho L(\rho))\sim \varepsilon(\rho)$, as $\rho\to\infty$ follows from Lemma \ref{lem: Bingham}, by taking $L_1=1/\varepsilon$ and $L_2=L$. The assertion $\varepsilon(\Lambda_L(\rho))\sim \varepsilon(\rho)$, as $\rho\to\infty$ follows from the previous assertion by Lemma \ref{lem: widetildeE}, part 1.

	\vspace{2mm}\textit{Part 2.} The assertion $L(\rho L(\rho))\sim L(\rho)$, as $\rho\to\infty$ follows from Lemma \ref{lem: Bingham}, by taking $L_1=L_2=L$. The assertion $L(\Lambda_L(\rho))\sim L(\rho)$, as $\rho\to\infty$ follows form the previous assertion by Lemma \ref{lem: widetildeE}, part 1.

	\vspace{2mm}\textit{Part 3.} The assertion follows from Lemma \ref{lem: Bingham}, by taking $L_1=L$ and $L_2(\rho)=\tfrac{1}{\rho|\varepsilon'(\rho)|}$.

	\vspace{2mm}\textit{Part 4.} The assertion $L(\rho L(\rho))\sim L(\rho)$, as $\rho\to\infty$ follows from Lemma \ref{lem: Bingham}, by taking $L_1=L_2=1/\varepsilon$.
\end{proof} 
\subsection{Proof of Lemma \ref{lem: LcomplexLemma}.}
\textit{Part 1.} Put $s=\rho e^{i\theta}$.
\[ \log L(s)=\int_0^s\frac{\varepsilon(w)}{w}dw=\int_0^\rho \frac{\varepsilon(u)}{u}du+i\int_0^\theta \varepsilon(\rho e^{i \psi})d\psi. \]
By assumption (R8), 
\[ \int_0^\theta \varepsilon(\rho e^{i \psi})d\psi=i\theta \varepsilon(\rho)(1+o(1)),\quad \rho\to\infty. \]
Thus,
\[ \log L(s)=\log L(\rho)+i\theta \varepsilon(\rho)(1+o(1)),\quad\rho\to\infty, \]
concluding part 1.

\vspace{2mm}\textit{Part 2.} The proof of part 2 is the same as the proof of part 1. \qed
\subsection{Proof of Lemma \ref{lemma:estimatcoef2}.}
First, we  show that 
\[ \sup_{n\leq r} \log\frac{r^n}{n!\gamma(n+1)}\asymp \Lambda_L(r). \]
Put 
\[ f(x,r) =x\log r-x\log (x L(x)).\]
Since $L$ is increasing and unbounded, 
\[ \sup_{n\leq r} \log\frac{r^n}{n!\gamma(n+1)}\asymp \sup_{n\in \mathbb{N}}\log\frac{r^n}{n!\gamma(n+1)}. \]
	By Stirling's formula, the RHS is $\asymp \sup_{n\in \mathbb{N}} f(n,r)$. Thus, it is enough to show that 
		\[\sup_{n\in \mathbb{N}} f(n,r)\asymp  \sup_{x>0} f(x,r)=\Lambda_L(r).\]
		By assumption, the function $x\mapsto x\log (xL(x))$ is an eventually  convex function of $\log x$. Thus, for sufficiently large $r$, the supremum in  $\sup_{x>0} f(x,r)$ is attained in a single point, which we denote by $x_r$.  By Taylor's theorem, there exists $\lfloor x_r \rfloor\leq c\leq x_r$ such that 
		\[ f(\lfloor x_r \rfloor, r)=f(x_r,r)+\frac{f_{xx}(c,r)}{2}(x_r-\lfloor x_r \rfloor). \]
Since, by assumption (R3),
\[ f_{xx}(c,r)=-\frac{1+\varepsilon(x)+x\varepsilon'(x)}{x}=o(1),\quad x\to\infty, \]
we get
\[ \left|f(\lfloor x_r \rfloor, r)-f(x_r,r)\right|=o(1),\quad r\to\infty. \]
Therefore, 
\[ \sup_{n\leq r} \log\frac{r^n}{n!\gamma(n+1)}\asymp \Lambda_L(r), \]
as claimed.  By Lemma \ref{lem: widetildeE}, part 1, 
\[ \Lambda_L(\delta^{-1}r)\sim\delta^{-1} \Lambda_L(r),\quad r\to\infty. \]
Therefore, for any $\delta$, 
\[  \sup_{n\leq r} \log\frac{r^n}{n!\gamma(n+1)\delta^n} \asymp \delta^{-1}\Lambda_L(r), \]
where the implicit constant is independent of $\delta$. This completes the proof of Lemma \ref{lemma:estimatcoef2}.
\qed
\subsection{Proof of Lemma \ref{lem: L_2}.}
	Fix a a  non-quasianalytic and  slowly growing function  $L$. It follows from \cite[\S1.6]{Bingham} that the functions  $\widetilde{L}$ and $L/\widetilde{L}$ are also slowly varying, with 
	$L(\rho)/\widetilde{L}(\rho)\to\infty$ as $\rho\to\infty$, and that 
	\[ \Lambda_L(r)=o\left(\int_0^\infty \frac{r}{r^2+t^2}\Lambda_L(t)dt\right),\quad r\to\infty. \]
	Since
	\[  \Lambda_L(r)\asymp \int_0^r \frac{r}{r^2+t^2}\Lambda_L(t)dt, \]
	we have
	\[ \int_0^\infty \frac{r}{r^2+t^2}\Lambda_L(t)dt\asymp r\int_r^\infty \frac{\Lambda_L(t)}{t^2}dt .\]
	Changing  variables by  $t=uL(u)$ and making use of Lemma \ref{lem: widetildeE}, part 1, yields 
	\[ \int_0^\infty \frac{r}{r^2+t^2}\Lambda_L(t)dt\asymp r\int_{\Lambda_L(r)}^{\infty}\frac{du}{uL(u)} =r\frac{\widetilde{L}(\Lambda_L(r))}{L(\Lambda_L(r))}\asymp \Lambda_L(r)\widetilde{L}(\Lambda_L(r))\]
	Put $\ell_1=L/\widetilde{L}$. It remains to show that there exists a function $L_*$ such that $L_*\asymp \ell_1$ and 

	\[ \Lambda_{L_*}(r)=\Lambda_L(r)\widetilde{L}(\Lambda_L(r)) \]
 By Lemma \ref{lem: widetildeE}, part 1,  applied to the function $L_*$, we have 
	\[ r\asymp \Lambda_L(r)\widetilde{L}(\Lambda_L(r))L_*\big(\Lambda_L(r)\widetilde{L}(\Lambda_L(r))\big). \]
	Applying the same lemma to the function $L$ and then to the function $\widetilde{L}$ yields
	\[ rL(r)\asymp r\widetilde{L}(r)L_*\big(r\widetilde{L}(r)\big)\quad \Rightarrow\quad L_*(r)\asymp\ell_1(\Lambda_{\widetilde{L}}(r))  \]
	
	By definition, $$\frac{\rho\ell_1'(\rho)}{\ell_1(\rho)}\log \widetilde{L}(\rho)=\frac{\log \widetilde{L}(\rho)}{ \widetilde{L}(\rho)}=o(1),\quad \rho\to\infty.$$
	Thus, by Lemma \ref{lem: Bingham} (applied with $L_1=\ell_1$ and $L_2=\widetilde{L}$), 
	\[ \ell_1(\rho)\asymp\ell_1\big(\rho \widetilde{L}(\rho)\big). \]
	By Lemma \ref{lem: widetildeE}, part 1, $\Lambda_{\widetilde{L}}\left(\rho\widetilde{L}(\rho)\right)\asymp \rho$, and therefore
	\[ \ell_1\left(r\right)\asymp \ell_1\left(\Lambda_{\widetilde{L}}(r)\right)\asymp L_*(r). \]	
	This completes the proof.\qed
 	\section{The functions $K$ and $E$}
In this section  we prove the  lemmas of Section 6  related to the asymptotics  of the functions $K$ and $E$.
\subsection{Lemma~\ref{lem: K and E ray lemma}. }
Fix a slowly growing function $L:[0,\infty)\to [1,\infty)$ with $\lim_{\rho\to\infty}L(\rho)=\infty$\footnote{All the assertions of Lemma \ref{lem: K and E ray lemma} are asymptotic, so it suffices to prove it for $L$ increasing and strictly positive.}. Put
\[ E(z)=\sum_{n\geq 0} \frac{z^n}{L(n+1)^{n+1}} \]
on the positive ray. Since the function $E$ depends only on the values of $L$ on the positive integers, we can change it on non-integer values (see \cite[pp. 17--18]{Seneta}) and assume that $L \in C^\infty[0,\infty)$ and satisfies
\begin{equation}
\lim_{\rho\to\infty}\frac{\rho L'(\rho)}{L(\rho)}=\lim_{\rho\to\infty}\frac{\rho^2 L''(\rho)}{L(\rho)}=0,\label{eq: additional ell assum}
\end{equation}
without changing the fact that $L$ is slowly growing. We shall assume so from here on.

Put
\[ \mu(r)=\sup_{\rho\geq 0} \frac{r^\rho}{L(\rho)^\rho}. \]
We claim  that \textit{the function  $\rho \mapsto \rho\log L(\rho)$ is an eventually  convex function of $\log \rho$.}
Indeed,  by \eqref{eq: additional ell assum}, 
\[\lim_{\rho \to \infty} \frac{d^i}{d\rho^i } \log L (e^\rho)=0,\quad i=1,2\]
which in turns implies 
\[ \frac{d^2}{d\rho^2 } \left[e^\rho L \ell(e^{\rho})\right]\sim e^\rho\log L(e^\rho)>0,\quad \rho>\rho_0,\]
as claimed.

Also note that for any fixed $r$, $\lim_{\rho \to \infty }\frac{r^\rho}{L(\rho)^\rho}=0$. So,    
for large enough $\rho$, the supremum in the definition of $\mu$ is achieved at a single point, which we  denote by $\rho_r$.
A simple computation shows that $r$ and $\rho_r$ are related by 
\begin{equation}
r=L(\rho_r)\exp\left(\rho_r\frac{L'(\rho_r)}{L(\rho_r)}\right)\quad \text{and} \quad \log \mu_L(r)=\rho_r^2\frac{L'(\rho_r)}{L(\rho_r)}.\label{eq: rho_r}
\end{equation} 
\subsubsection{Proof of the inequality  $L^{-1}(\eta r)\lesssim_\eta \log E( r)$.}
Set $\nu(r)=\sup_{n\in\mathbb{Z}_+}\frac{r^n}{L(n+1)^{n+1}}$. Clearly, $\nu(r)\leq E(r)$. Since $\lim_{n\to\infty}\log \frac{\log L(n+2)}{\log L (n)}= 1$, we have $\log \mu(r)\lesssim\log \nu(r)$, which in turn implies
$\log \mu(r)\lesssim\log E(r) $. Now for $\eta<1$
$$\log \mu (\eta^{-1} r)=\sup_{\rho>0} \left[\rho\log \eta^{-1} r-\rho\log L(\rho)\right]\stackrel{\left(\rho=L^{-1}(r)\right)}{\geq}  L^{-1}(r) \log \eta^{-1}, $$ and therefore $L^{-1}(\eta r)\lesssim_\eta \log E( r)$.\qed

\subsubsection{Proof of  the inequality $ \log E( r)\lesssim L^{-1}(r)$.}
We will  show that  $ \log \left(L(r) E( L(r))\right)\lesssim r$.
Since $L$ is slowly growing and positive, the function $\rho\mapsto L(\rho)/\rho$ is eventuality decreasing. Thus, for sufficiently large $r$, 
\[ L(r) E( L(r))=\sum_{n\geq 1}\left(\frac{L(r)}{L(n)}\right)^n\lesssim re^{r}+\sum_{n>r} \frac{r^n}{n^n}. \]
By Stirling's formula, $\sum_{n>r} \frac{r^n}{n^n}\leq e^{C r}$, which establishes the lemma. \qed
\subsubsection{Proof of the  inequality $E^2(r)\lesssim_\delta E\left((1+\delta)r\right)$.}
By the previous parts of this lemma, it is enough to show that
\[ 2L^{-1}(r)\leq L^{-1}\left((1+\delta)r\right),\quad r>r_\delta. \]
But this follows immediately from the fact that
\[ \log \frac{L^{-1}\left((1+\delta)r\right)}{L^{-1}(r)} =\int_r^{(1+\delta)r}\frac{1}{\varepsilon\left(L^{-1}(r)\right)}\frac{du}{u}\to\infty ,\quad r\to\infty. \]
\subsection{Proof of  Lemma \ref{lem: KE-A}.} 
For $r>r_0>0$, denote by $\rho(r)$ the solution to the saddle-point equation $r=L(\rho)e^{\varepsilon(\rho)}$. Let $0<\delta<1$. It follows from  Theorems \ref{TheoremK}  and \ref{TheoremE} that  in order to prove  Lemma \ref{lem: KE-A}, it is enough to find  $0<\delta_1$ such that 
 \[\rho(\delta_1 r))\varepsilon\left(\rho((\delta_1 r)\right)+\rho(r(1-\delta))\varepsilon\left(\rho(r(1-\delta))\right)-\rho(r)\varepsilon\left(\rho(r)\right)\lesssim 1,\quad r>r_0. \]
Thus, it is sufficient to show that there exist $\alpha>0$ (independent of $\delta$), such that 
\[ \rho(r(1-\delta))\varepsilon\left(\rho(r(1-\delta)\right)\leq (1-\delta)^\alpha \rho(r)\varepsilon\left(\rho(r)\right).\]

Let us prove the last assertion:
\begin{multline*}
 \log\frac{ \rho(r)\varepsilon\left(\rho(r)\right)}{\rho(r(1-\delta))\varepsilon\left(\rho(r(1-\delta)\right)}=\int_{r(1-\delta)}^{r}\frac{\varepsilon\left(\rho(u)\right)+\rho(u)\varepsilon'\left(\rho(u)\right)}{\rho(u)\varepsilon\left(\rho(u)\right)} \rho'(u) du\\
= (1+o(1))\int_{r(1-\delta)}^{r}\frac{\rho'(u)}{\rho(u)}  du,\quad r\to\infty,
\end{multline*}
  were in the last equality, we have used $\rho|\varepsilon'(\rho)|=o\left(\varepsilon(\rho)\right)$ as $\rho\to\infty$. 
  Differentiating, $\log r=\log L(\rho)+\varepsilon(\rho)$, yields 
 $$\frac{\rho'(r)}{\rho(r)}=\frac{1}{r}\cdot \frac{1}{\varepsilon\left(\rho(r)\right)+\rho(r)\varepsilon'\left(\rho(r)\right)}=\frac{1}{r\varepsilon\left(\rho(r)\right)}(1+o(1)),\quad r\to\infty.$$
 So,
 \[ \log\frac{ \rho(r)\varepsilon\left(\rho(r)\right)}{\rho(r(1-\delta))\varepsilon\left(\rho(r(1-\delta)\right)}=(1+o(1))\int_{r(1-\delta)}^{r}\frac{1}{\varepsilon\left(\rho(u)\right)}  \frac{du}{u}. \]
 Since $\varepsilon$ is bounded from above and positive, there exists $\alpha>0$, such that
 \[  \log\frac{ \rho(r)\varepsilon\left(\rho(r)\right)}{\rho(r(1-\delta))\varepsilon\left(\rho(r(1-\delta)\right)}> \alpha\log  \frac{1}{1-\delta}. \]
The last inequality completes the proof. \qed

\subsubsection{Proof of Lemma \ref{lem: E complex lemma}.}
We assume that $z=r e^{i\psi}$, $s=\rho e^{i\theta}$, are related trough the saddle-point equation
\[ \log z=\log L(s)+\varepsilon(s). \]  
Comparing the real  and imaginary parts of the saddle-point equation and making use of \eqref{eq: L saddele point assym}, we find that
\[ r=L(\rho)e^{\varepsilon(\rho)}(1+o(1)),\quad  \theta \varepsilon(\rho)(1+o(1))=\psi,\quad r\to\infty. \]
Let $\delta>0$ be a small number. By Theorem \ref{TheoremE}, $|E(z)|=O(1)$, as $z\to\infty$, uniformly in the set $\mathbb{C}\setminus\Omega(\tfrac{\pi}{2}+\delta)$. Since,
$$\partial\Omega(\tfrac{\pi}{2}+\delta):= \left\{z\;:\; \theta=\pm\tfrac{\pi}{2}+\delta,\; \rho>\rho_0 \right\},$$
 it is enough to show that 
\[ \varepsilon(\rho)\lesssim \varepsilon\left(L^{-1}(r)\right),\quad r>r_0. \]
Indeed,  write,
\[ \log \frac{\varepsilon(\rho)}{\varepsilon\left(L^{-1}(r)\right)}=-\int_\rho^{\rho e^{\varepsilon\left(L^{-1}(\rho)\right)}(1+o(1))} \frac{\varepsilon'(u)}{\varepsilon(u)}du.\]
By the regularity  assumption (R2), the function  $\varepsilon$ is eventually non-increasing. Thus, by the regularity assumption (R3), we get
\[ -\frac{\varepsilon'(u)}{\varepsilon(u)}\leq \frac{1}{u},\quad u>\rho_0. \]
Therefore
\[  \log \frac{\varepsilon(\rho)}{\varepsilon\left(L^{-1}(r)\right)} \lesssim \varepsilon\left(L^{-1}(\rho)\right)\lesssim 1,  \]
which completes the proof of  Lemma \ref{lem: E complex lemma}.\qed
\subsection{An auxiliary lemma}
Here we give an auxiliary result that will be used in the proofs of Lemmas  \ref{lem: K and E last}, \ref{lem:KE strong A for spliting.}, and \ref{lem:spliting H intgral K small}.

\begin{lemma}\label{lem: Apn}
	Suppose that $L_1$ is a function that satisfies regularity assumptions \emph{(R1)} and \emph{(R2)}, and that $L_2 $ is a function that satisfies regularity assumptions \emph{(R1)}.  Assume further that $\varepsilon_1(\rho) \log L_2(\rho)=o(1)$, as $\rho\to\infty$, where $\varepsilon_1(\rho)=\frac{\rho L_1'(\rho)}{L_1(\rho)}$.
 If 
	$\rho=\rho(r)$ satisfies 
	\[ r=L_1(\rho)(1+O(\varepsilon_1(\rho))),\quad r\to\infty, \]
	then for any $\delta>0$, there exists $r_\delta>0$, such that
	\[ L_1^{-1}\left((1-\delta)r\right) \leq\frac{\rho}{L_2(\rho)}\leq\rho \leq  L_1^{-1}\left((1+\delta)r\right),\quad r>r_{\delta}.
\]  
\end{lemma}
\begin{proof}
	The inequalities
		\[ L_1^{-1}\left((1-\delta)r\right) \leq \rho \leq  L_1^{-1}\left((1+\delta)r\right),\quad r>r_{\delta}.
	\]  
	follow immediately from
		\[ L_1\left(L_1^{-1}((1\pm\delta)r)\right)(1+O\left(\varepsilon_1\left(L_1^{-1}((1\pm\delta)r)\right)\right))=(1\pm\delta)r+o(1),\quad r\to\infty. \]
		Since $L_2(\rho)\to\infty$ as $\rho\to\infty$, it suffices to show 
	\[ L_1^{-1}\left((1-\delta)r\right)\leq \frac{L_1^{-1}(r)}{L_2\left(L_1^{-1}(r)\right)},\quad r>r_\delta.  \]
	To show this, we begin with
	\[ \frac{L_1^{-1}(r)}{ L_1^{-1}\left((1-\delta)r\right)}=\exp\left(\int_{(1-\delta)r}^{r}\frac{d}{du}\left(\log L_1^{-1}(u)\right) du\right)=\exp\left( \int_{(1-\delta)r}^{r}\frac{1}{\varepsilon_1\left(L_1^{-1}(u)\right)}\frac{du}{u}\right). \]
	Since the function  $\varepsilon$ is eventually decreasing, 
	\[  L_1^{-1}(r)\geq L_1^{-1}\left((1-\delta)r\right)\exp\left(\frac{-\log(1-\delta)}{\varepsilon_1\left(L_1^{-1}(r)\right)}\right),\quad r>r_0.\] 
	Thus, it is enough to show that
	\[   \exp\left(\frac{-\log(1-\delta)}{\varepsilon_1\left(L_1^{-1}(r)\right)}\right)>L_2\left(L_1^{-1}(r)\right),\quad r>r_\delta.\]	
	The latter follows immediately from the assumption $\varepsilon_1(\rho) \log L_2(\rho)=o(1)$, as $\rho\to\infty$.
The proof is complete.
\end{proof}
\subsection{Proof of Lemma \ref{lem: K and E last}.}
Fix $\alpha<\pi/2$ and $0<\delta<1/3$. For $z\in \Omega(\alpha)$ with sufficiently large $|z|$, denote by $s=\rho e^{i \theta}$ the unique  solution to the saddle-point equation. 
\[  \log z= \log L(\rho e^{i\theta})+\varepsilon(\rho e^{i\theta}). \]
By \eqref{eq: L saddele point assym}, 
\[ \log L(\rho e^{i\theta})+\varepsilon(\rho e^{i\theta})=\log L(\rho)+\varepsilon(\rho)\left(1+i\theta+o(1)\right),\quad \rho\to\infty. \]
Thus, comparing the real parts of the saddle-point equation, we obtain
\[ |z|=L(\rho)\exp(\varepsilon(\rho)), \quad |z|\to\infty. \] 
By Lemma \ref{lem: Apn}, applied with $L_1=L$, $L_2=\tfrac{1}{\varepsilon}$, 
\[ -\delta_1 \rho\varepsilon(\rho)\leq C-L^{-1}\left(\frac{2}{3}  |z|\right)  \]
 By Lemma \ref{lem: K and E ray lemma},
\[ -L^{-1}\left(\frac{2}{3}  |z|\right)\leq C-\log E\left(\frac{3}{5} |z|\right).\]
Thus, by the matching between the growth of $E$ and the decay of $K$ (i.e., Theorems \ref{TheoremK} and \ref{TheoremE}), we have 
\[ -\delta_1 \rho\varepsilon(\rho)\leq C+\log K\left(\frac{|z|}{2}\right).  \]
	Therefore, 
	\[ \int_{r_0}^\infty |z|^n e^{-\delta_1 \frac{\rho}{\widetilde{L}(\rho)}}d|z|\leq C\int_{0}^\infty |z|^n K\left(\frac{|z|}{2}\right)d|z| \leq C 2^n\frac{\gamma(n+1)}{\widetilde{\gamma}(n+1)}.\]
	This completes the proof.\qed
	\subsection{Proof of Lemma \ref{lem:spliting H intgral K small}.}
	For $z\in \Psi_+$ sufficiently large, let $i \rho$ be related to $z$ by the saddle-point equation:
	\[ \log z=\log L(i\rho)+\varepsilon (i\rho) .\]
	By \eqref{eq: L saddele point assym},
	\[ \log L(s)+s\frac{L'(s)}{L(s)} =\log L(\rho)+\varepsilon(\rho)+i\left(\theta+o(1)\right)\varepsilon(\rho),\quad s=\rho e^{i\theta},\,\; \rho\to\infty. \]
	Thus, comparing the real parts of the saddle point equation yields
	\begin{equation}
	 |z|=L(\rho)(1+O(\varepsilon(\rho))) ,\quad |z|\to\infty.\label{eq:Lastasym}
	\end{equation}
	By Theorem \ref{TheoremE}, 
	\[ \log |E(z)|\sim \re(i\rho\varepsilon(i\rho)),\quad |z|\to\infty. \]
	By Lemma \ref{lem: LcomplexLemma}, 
	\[ \varepsilon(i\rho)=\varepsilon(\rho)+i\frac{\pi}{2}\rho\varepsilon'(\rho)(1+o(1)),\quad \rho\to\infty. \]
	The function $\varepsilon'$ is eventually negative, therefore
	\begin{equation}
	\log |E(z)|\sim \frac{\pi}{2}\rho^2|\varepsilon'(\rho)| ,\quad z\in \Psi_+,\;|z|\to\infty,\label{eq: EK-critical}
	\end{equation}
		which together and \eqref{eq:Lastasym}  shows that
		\[  -\frac{\log |E(z)|}{z}\leq -\frac{\rho^2|\varepsilon'(\rho)|}{L(\rho)},\quad |z|>r_0. \]
		Put $L_2(\rho):=-\frac{L(\rho)}{\rho|\varepsilon'(\rho)|}$. 
		By assumption, $L_2$ is slowly growing (i.e., satisfies assumption (R1)), and $\varepsilon(\rho)\log L_2(\rho)=o(1)$ as $\rho\to\infty$. Thus,   Lemma \ref{lem: Apn} yields
		\[ -\frac{\rho^2|\varepsilon'(\rho)|}{L(\rho)}\leq -L^{-1}\left(\frac{3}{4}|z|\right). \]
Finally, by Lemma \ref{lem: K and E ray lemma} and Theorems \ref{TheoremK} and \ref{TheoremE} (the reasoning is the same as in  the previous Lemma),
	\[ -\frac{\rho^2|\varepsilon'(\rho)|}{L(\rho)}\leq -K\left(\frac{|z|}{2}\right). \]
Hence, 
	\[ \int_{\Psi_+\cap\{|z|>1\}} |z|^n \left|E(z)\right|^{-1/|z|}d|z|\leq C+ C\int_{r_0}^\infty K\left(\frac{|z|}{2}\right)\leq C 2^n \gamma(n+1),\]
	which completes finishes the proof.\qed
\subsection{Proof of Lemma \ref{lem:KE strong A for spliting.}.}
Let $\delta>0$. By Lemma \ref{lem: K and E ray lemma},
\[ \log E((1-\delta)|z|)\leq C+ L^{-1}((1-\delta)|z|). \]
If $z\in \Psi_+$, then by theorem Theorems \ref{TheoremK} and \ref{TheoremE}, 
\[ \log |K(z)|\sim -\log |E(z)|\stackrel{\eqref{eq: EK-critical}}{\sim}\frac{\pi}{2}\rho^2|\varepsilon'(\rho)| ,\quad z\in \Psi_+,\;|z|\to\infty. \]
where $\rho(z)= L(|z|)(1+O(\varepsilon(|z|)))$, as $|z|\to\infty$. Put $L_2(\rho)=\tfrac{1}{\rho|\varepsilon'(\rho)|}$. By assumption, $L_2$ is slowly growing (i.e., satisfies assumption (R1)), and $\varepsilon(\rho)\log L_2(\rho)=o(1)$ as $\rho\to\infty$.  Thus, by  Lemma \ref{lem: Apn}, 
\[ \log |K(z)|\leq C-L^{-1}\left((1-\tfrac{\delta}{2}|z|)\right). \]
Since $L$ is slowly varying, 
\[  2L^{-1}\left((1-\delta)|z|\right)-L^{-1}\left((1-\tfrac{\delta}{2})|z|\right)\leq C\]
We obtained,
\[ E^2((1-\delta)|z|)||K(z)|\leq C,\quad z\in\Psi_+, |z|>1, \]
and hence the lemma.\qed

	\subsection{Proof of Lemma  \ref{lem: loglog H}. }
	Put $q=\log \log H$. It is enough to prove the lemma for sufficiently small $\psi>0$ and sufficiently large $r$. By definition, 
	\[ \log H(\psi)=\re\left(i\rho\varepsilon(i\rho)\right), \]
	where 
	\[ \psi=\im\left(\log L(i\rho)+\varepsilon(i\rho)\right). \] 
	Differentiation with respect to $\psi$ yields
	\[ 1=\im\left(\frac{\varepsilon(i\rho)+\rho\varepsilon(i\rho)}{\rho}\right)\frac{d\rho}{d\psi}.  \]
	By Lemma \ref{lem: LcomplexLemma}, the RHS is $\sim \frac{\pi}{2}\varepsilon'(\rho)\frac{d\rho}{d\psi}$ as $\rho\to\infty$. The same lemma also gives
	\[ \re\left(i\rho\varepsilon(i\rho)\right)\sim\frac{\pi}{2}\rho^2|\varepsilon'(\rho)|,\quad \frac{d}{d\rho}\re\left(i\rho\varepsilon(i\rho)\right)=\frac{\pi}{2}\rho|\varepsilon'(\rho)|,\quad \rho\to\infty. \]
	Therefore,
	\[ \frac{dq}{d\psi}=\frac{dq}{d\rho}\frac{d\rho}{d\psi}=\frac{2}{\pi\rho\varepsilon'(\rho)}(1+o(1)),\quad \rho\to\infty. \]
	By assumption, the RHS tends to $-\infty$ as $\rho\to\infty$ and thus, as $\psi\to 0$. In particular for sufficiently small $\psi$, $\frac{dq(\psi)}{d\psi}\leq -1$, which completes the proof.\qed
	\section{The function $E_1$}
	\subsection{Proof of Lemma \ref{lem: widetildeE trvial bound}.}
	Put 
	\[ \widetilde{\Lambda}(r):=\log\left(\sup_{n\geq 0}\frac{r^n}{n!\gamma(n+1)}\right). \] 
	By the definition of the function $E_1$,
	\[E_1(r)=\sum_{n\geq 0}\frac{r^n}{n!\gamma(n+1)}=\sum_{n\leq re} \frac{r^n}{n!\gamma(n+1)}+\sum_{n> re} \frac{r^n}{n!\gamma(n+1)}. \]
	The first summand on the RHS is bounded by $(re+1)e^{\widetilde{\Lambda}(r)}$, while the second summand is bounded by $\sum_{n> re} \frac{n^n}{e^n n!\gamma(n+1)}\leq\sum_{n\geq 0} \frac{1}{\gamma(n+1)}<\infty$. We conclude that $\log E_1(z)\lesssim \widetilde{\Lambda}(r)$. 
	By Lemma \ref{lemma:estimatcoef2},  $\widetilde{\Lambda}(r)\lesssim\Lambda_L(r)$, which completes the proof.\qed
\subsection{Proof of  Lemma \ref{lem: widetilde E 1st}.}    Let $x>0$ be sufficiently large. We define $s=\rho e^{i\theta}$ as the solution to the saddle-point equation
\[ \log ix= \frac{\Gamma'(s)}{\Gamma(s)} + \log L(s)+\varepsilon(s).\]
By \eqref{eq: apen 8.6},
\[ x \asymp \rho L(\rho), \quad  \frac{\pi}{2}=\theta\big(1+\varepsilon(\rho)(1+o(1))\big) ,\quad \rho\to\infty.\]
Thus, part 1 of Lemma \ref{lem: widetildeE} yields $\rho\asymp \Lambda_L(x)$. Since the function $\varepsilon$ is slowly varying, we also get
\[ \theta =\frac{\pi}{2}-\varepsilon(\Lambda_L(x))(1+o(1)),\quad x\to\infty.  \]
By Theorem \ref{TheoremE}, 
\[   \log E_1(ix)\sim  s\left(1+\varepsilon(s)\right),\quad x\to\infty. \]
Therefore,
\[   \log|E_1(ix)|\sim \rho\cos \theta \asymp \Lambda_L(x) \varepsilon\left(\Lambda_L(x)\right). \]
By Lemma \ref{lem: subReplacmentLem}, part 1,  $\varepsilon\left(\Lambda_L(x)\right)\asymp \varepsilon(x)$. We get
$\displaystyle \log|E_1(ix)| \asymp \Lambda_L(x)\varepsilon(x),$ and hence the lemma.
\qed
\subsection{Lemma \ref{lem: mainSpilitEstimate+}.}
We will use the following auxiliary result.
\begin{lemma}\label{lem: estimateE1}
	Suppose that $L$ satisfy assumptions \emph{(R1), (R2), (R3)} and \emph{(R9)}. If $\delta>0$ is  sufficiently small, then there exists $R_0>0$, such that 
	\[ \log |E_1 (z)|\leq C+\re_+ \left(s(1+\varepsilon(s)) \right),\quad  |\arg z|\leq \tfrac{\pi}{2}+\delta,\; |z|>R_0,\]
	\[  \log |E_1 (z)|\leq C,\quad   |\arg z|> \tfrac{\pi}{2}+\delta,\; |z|>R_0, \]
	where $s$ and $z$ are related by the saddle-point equation  $sL(s)e^{\varepsilon(s)}=z$ and $\re_+(w)=\max \{\re (w),0\}$.
\end{lemma}	
\begin{proof}
	It is easy to check that if $s\mapsto\gamma(s)$ satisfies assumptions (R3) and (R8), then so does $s\mapsto\Gamma(s)\gamma(s)$, where $\Gamma$ is the Euler Gamma function. Thus, Theorem \ref{TheoremE} is applicable to the function 
	\[ E_1(z)=\sum_{n\geq 0}\frac{z^n}{n!\gamma(n+1)}. \]
	The saddle-point equation for the function $E_1$ is 
	\[ \log z= \frac{\Gamma'(s)}{\Gamma(s)} + \log L(s)+\varepsilon(s).\]
	From here on we assume that $s$ is related to $z$ by the above saddle-point equation.
	By Stirling's formula, $$\frac{\Gamma'(s)}{\Gamma(s)}=\log s+O(|s|^{-1}),$$ uniformly in $|\arg s|<\pi-\delta$ as $s\to\infty$. 
	Thus
	\begin{equation}
	\log z= \log s + \log L(s)+\varepsilon(s)+O(s^{-1}),\quad |s|\to\infty.\label{eq: L saddele point no assym}
	\end{equation}
	Combining the later with \eqref{eq: L saddele point assym}, we find that
	\[ \log z=\log(\rho e^{i\theta})+\log L(\rho)+\varepsilon(\rho)+i\theta\varepsilon(\rho)(1+o(1))+O(\rho^{-1}),\quad s=\rho e^{i\theta},\;s\to\infty, \]
	uniformly  for $|\theta|\leq \tfrac{\pi}{2}+\delta$. 
	Comparing the real and imaginary parts separately, we obtain,
	\begin{equation}
	z \asymp \rho L(\rho), \quad  \arg(z)=\theta(1+\varepsilon(\rho)+o(\varepsilon(\rho))) ,\quad \rho\to\infty. \label{eq: apen 8.6}
	\end{equation}
	
	The function $L$ is slowly varying, therefore, $\varepsilon(\rho)\to 0$ as $\rho\to \infty$. Thus, \eqref{eq: apen 8.6} yields, 
	\[ |\arg(z)|\geq \frac{\pi}{2}+\delta \Rightarrow |\theta|> \frac{\pi+\delta}{2},\quad \rho>\rho_\delta. \]
	In particular, in this case, Theorem \ref{TheoremE}  implies that
	\[ |E_1(z)|\leq C_\delta.\]
	
	On the other hand, if $|\arg(z)|\leq \frac{\pi}{2}+\delta$,
	then the same theorem yields
	\[ \log|E_1(z)|\leq C_\delta+\re_+\left( \frac{\Gamma'(s)}{\Gamma(s)}-\frac{\log \Gamma(s)}{s}+s\varepsilon(s)\right) \leq   C'_\delta +\re_+\left(s(1+\varepsilon(s))\right). \] 
	
	Let $s_*$ be the solution to 
	\[ \log z=\log s_*+\log L(s_*)+\varepsilon(s_*).\]
	By \eqref{eq: L saddele point no assym}, for sufficiently large $|z|$, we have
	\[ \re\left(s(1+\varepsilon(s))\right)\leq \re\left(s_*(1+\varepsilon(s_*))\right) +C. \]
	This completes the proof.
\end{proof}	
\begin{proof}[Proof of Lemma  \ref{lem: mainSpilitEstimate+}]
	Fix  $0<\delta$ (small) and  $\rho_0>0$ (big)  such that    the function $\rho\mapsto \im\big(\log L(i\rho) \\+~\varepsilon( i \rho)\big)$ is decreasing and smaller then $\delta$ in the ray $[\rho_0,\infty)$. 	
	Also  Fix  $t\geq 1$,  and $0<\psi\leq \tfrac{\pi}{2}$. 	
	
	If  $\psi\geq \delta$,  then Lemma \ref{lem: estimateE1} yield
	\[\big|E_1(i t r e^{i\psi})\big|\lesssim 1, \]
	and the Lemma holds. So, we will assume from here on that $0< \psi < \delta$. 
	
	Put
	\[ w(s)=(1+\varepsilon(s))s,\quad \widetilde{w}(s)=\log s+\log L(s)+\varepsilon(s). \]
	Since assumptions (R3) and (R8) hold, for $r>r_0$, there  exists a unique solution to the equation
	\begin{equation}
	\log t+\log r+i\left(\tfrac{\pi}{2}+\psi\right)=\widetilde{w}(s). \label{eq: s(t)def}
	\end{equation}
	We  denote this  solution by $s(r)$. By Lemma \ref{lem: estimateE1},
	\[ \big|E_1(i t r e^{i\psi})\big|\leq C+\re_+ \left(w(s(r))\right) \]
	Our goal  is  to  show that 
	\[ \re \left(w(s(r))\right)\leq  C+ H(\psi).\]
	We begin with showing  that the function $r\mapsto \re \left(w(s(r))\right)$  has a unique maximum in the interval $[r_0,\infty)$. The functions $w$ and $\widetilde{w}$ are related by 
	\[ s\widetilde{w}'(s)=w'(s). \]
	
	By \eqref{eq: L saddele point assym}, the real and imaginary parts of \eqref{eq: s(t)def} yield
	$$ |s(r)|\to\infty,\quad \arg(s(r))\big(1+\varepsilon(|s(r)|)(1+o(1))\big)=\frac{\pi}{2}+\psi $$
	as $r\to\infty$. By assumption, the function $x\mapsto \varepsilon(x)$ is eventually positive and decreasing to zero as $x\to\infty$. Thus, $\re w(s(r)) $ is eventually negative. Differentiation of \eqref{eq: s(t)def}  with respect to $r$ yields
	\[  s'(r)\widetilde{w}'(s(r))=\frac{1}{r}.\]
	Thus,
	\[ \frac{d}{dr}w(s(r))=w'(s(r))s'(r)= \frac{1}{rs(r)}.\]
	In particular, $\frac{d}{dr}\re(w(s(r)))=0 $ if and only if $\re \left(s(r)\right)=0$, or, which is the same, if and only if $s(r)=i\tau$ for some $\tau>0$. Comparing the imaginary parts of \eqref{eq: s(t)def}, such a $\tau$ satisfies
	\[ \im\left( \log L(i \tau) +\varepsilon(i\tau)\right)=\psi. \]
	Since, $0<\psi<\delta$ and $r$ is sufficiently large, such a $\tau$  exists and unique. Moreover,
	\[ \re(w(i\tau)) =\re\left(i\tau\varepsilon(i\tau)\right) >0. \]
We conclude that
	\[ \max_{r\geq r_0} \re \left(w(s(t))\right)=\re\left(i\tau\varepsilon(i\tau)\right) =\log H(\psi). \]
	
	Given $B>0$, we assume that $r_0$ is large enough, and that $\frac{B}{r}<\psi<\delta$. Denote by $\tau_B$ the unique solution to 
	\[ \im\left(\log L(i\tau_B)+\varepsilon( i \tau_B)\right) =\psi+\frac{B}{r} \]
	By \eqref{eq: L saddele point assym}, 
	\[ \im\left(\log L(i\rho)+\varepsilon( i \rho)\right)=\frac{\pi}{2}\varepsilon(\rho)(1+o(1)) ,\quad \rho\to\infty,\]
	while
	\[ \frac{d}{d\rho}  \re\left(w(i\rho) \right)=\re(\varepsilon(i\rho))(1+o(1))=o(\varepsilon(\rho)),\quad \rho\to\infty. \]
	Thus,  by mean value theorem, 
	\[ \re\left(w(i\tau)-w(i\tau_B)\right)\leq C_B\frac{\tau}{r}.  \] Comparing the real parts of \eqref{eq: s(t)def},  and making use of \eqref{eq: L saddele point assym}, yields
	\[r t= \tau L(\tau)(1+O(1)),\quad r\to\infty.   \]
	Thus, by part 1 of Lemma \ref{lem: widetildeE},
	\[ \frac{\tau}{r}\leq C\Lambda_L(t).  \]
	We have established
	\[ \max_{r\geq r_0} \re \left(w(s(t))\right)\leq C_B\left(\Lambda_L(u)+1\right)+\log H\left(\psi+\frac{B}{r}\right), \]
	which completes the proof. 
\end{proof}
 \end{appendices}

\end{document}